\documentclass{article}
\usepackage{graphicx} 
\usepackage[margin=1in]{geometry}
\usepackage[T1]{fontenc}
\usepackage{times}
\usepackage{amsmath, amsfonts, amssymb, amsthm}
\usepackage{xcolor,hyperref,cleveref}
\hypersetup{colorlinks=true,breaklinks=true,urlcolor=black,linkcolor=black,citecolor=black}
\usepackage{authblk}

\newtheorem{theorem}{Theorem}[section]
\newtheorem{lem}[theorem]{Lemma}
\newtheorem{prop}[theorem]{Proposition}
\newtheorem{cor}[theorem]{Corollary}
\newtheorem{conj}[theorem]{Conjecture}
\theoremstyle{remark}
\newtheorem{rem}{Remark}

\def\N{\mathbb{N}}
\def\Z{\mathbb{Z}}

\def\R{\mathbb{R}}

\def\m{\!-\!}
\def\mm{\!\!-\!\!}
\def\p{\!+\!}

\newcommand{\cv}[2]{\underset{#1\to#2}{\longrightarrow}}

\def\O{\mathcal{O}}
\newcommand{\Landau}[3]{\underset{#2}{#1}\left(#3\right)}
\newcommand{\card}[1]{\left|#1\right|}

\def\E{\mathbb{E}}
\newcommand{\mean}[1]{\mathbb{E}\left[#1\right]}
\def\P{\mathbf{P}}
\newcommand{\prob}[1]{\mathbf{P}\left(#1\right)}
\def\V{\mathrm{Var}}
\newcommand{\var}[1]{\mathrm{Var}\left[#1\right]}
\newcommand{\cov}[2]{\mathrm{Cov}\left[#1,#2\right]}
\def\Leb{\mathrm{Leb}}
\newcommand{\One}[1]{\mathbf{1}_{#1}}
\newcommand{\Unif}[1]{{\rm Unif}\left(#1\right)}
\newcommand{\Exp}[1]{{\rm Exp}\left(#1\right)}
\newcommand{\Gam}[2]{\mathrm{Gamma}\left(#1 , #2\right)}

\newcommand{\Binomial}[2]{{\rm Binom}\left(#1 , #2\right)}
\newcommand{\Normal}[2]{{\cal N}\left(#1 , #2\right)}

\def\LIS{\mathrm{LIS}}
\def\LDS{\mathrm{LDS}}

\def\S{\mathfrak{S}}
\def\fs{\mathfrak{s}}
\def\pat{\mathrm{pat}}
\newcommand{\hrec}[1]{\mathrm{rec}^\uparrow\left(#1\right)}
\newcommand{\lrec}[1]{\mathrm{rec}^\downarrow\left(#1\right)}
\def\LSKV{\mathrm{LSKV}}
\def\carre{[0,1]^2}
\def\opencarre{(0,1)^2}

\def\pts{\mathcal{P}}
\newcommand{\perm}[1]{{\rm Perm}\left(#1\right)}
\newcommand{\st}[1]{\mathrm{st}\langle#1\rangle}
\def\Dx{\Delta x}
\def\Dy{\Delta y}

\def\Usjo{\mathcal{U}}
\def\Psjo{\mathfrak{P}}

\def\L{\mathcal{L}}
\def\Vert{\mathcal{V}}
\def\Edges{\mathcal{E}}
\def\s{\mathfrak{s}}

\title{A geometric approach to\linebreak conjugation-invariant random permutations}
\author{Victor Dubach}
\date{}

\begin{document}

\maketitle

\begin{abstract}
    We propose a new approach to conjugation-invariant random permutations.
    Namely, we explain how to construct uniform permutations in given conjugacy classes from certain point processes in the plane.
    This enables the use of geometric tools to study various statistics of such permutations.
    For their longest decreasing subsequences, we prove universality of the $2\sqrt n$ asymptotic.
    For Robinson--Schensted shapes, we prove universality of the Vershik--Kerov--Logan--Shepp limit shape, thus solving a conjecture of Kammoun.
    For the number of records, we establish a phase transition phenomenon as the number of fixed points grows.
    For pattern counts, we obtain an asymptotic normality result, partially answering a conjecture of Hamaker and Rhoades.
\end{abstract}

\noindent{\bf Acknowledgments:} The author is very grateful to Valentin F\'eray for insightful discussions and valuable suggestions, and to Mohamed Slim Kammoun for helpful comments.

\section{Introduction}

\subsection{Conjugation-invariant random permutations}

We say that a random permutation $\tau$ of $[n] := \{1,\dots,n\}$ is {\em conjugation-invariant} when for any fixed permutation $\rho$, the conjugate $\rho\circ\tau\circ\rho^{-1}$ follows the same law as $\tau$.
Standard examples of conjugation-invariant permutations are uniformly random permutations, random involutions, Ewens random permutations \cite{E72}, their generalizations \cite{BUV11,K78,T00}, and many more.

An equivalent description of conjugation-invariance makes use of the cycle decomposition.
Define the {\em cycle type} of $\tau$ as the sequence $t=(t_1,t_2,\dots,t_n)$ where $t_p$ is the number of $p$-cycles in $\tau$.
We also say that $t$ has size $\sum_{p=1}^n pt_p = n$ and that $\tau$ is $t$-cyclic.
Since the cycle type is a total invariant for conjugation in the symmetric group, a random permutation $\tau$ is conjugation-invariant if and only if, when conditioned on its cycle type $t$, $\tau$ is a uniform $t$-cyclic permutation.
Thus, studying conjugation-invariant random permutations often amounts to studying uniform permutations with given cycle types.

There are reasons to believe in universality phenomena for several statistics of conjugation-invariant permutations, with a dependence on the number of fixed points and sometimes $2$-cycles.
Recent advances motivate this idea by generalizing asymptotic properties of uniformly random permutations to conjugation-invariant random permutations under certain hypotheses \cite{F13,FKL22,HR22,K18}.
The aim of this paper is to make further progress in this direction by introducing a ``geometric'' construction for random permutations in given conjugacy classes, see \Cref{lem: geometric construction t-cyclic} in \Cref{sec: geometric construction}.
This new point of view allows us to establish asymptotic results on the longest monotone subsequences, Robinson--Schensted shape, number of records, and pattern counts.

\begin{description}
    \item[Longest monotone subsequences.] The maximum length of a monotone subsequence in a uniform permutation famously behaves as $2\sqrt n$ \cite{VK77}.
    Baik and Rains \cite{BR01} showed similar asymptotics for random involutions, with a dependence on the number of fixed points, and Kammoun \cite{K18} proved that the $2\sqrt n$ asymptotic holds for conjugation-invariant random permutations with few cycles.
    In \Cref{sec: results LIS LDS} we extend these results to general conjugation-invariant permutations for their longest {\em decreasing} subsequences.
    Longest increasing subsequences are harder to handle, and we only provide such asymptotics up to a multiplicative constant.
    \item[Robinson--Schensted shape.] The Robinson--Schensted shape of a uniform permutation converges, after suitable rescaling, to a limit curve \cite{LS77,VK77}.
    Kammoun \cite{K18} extended this result to conjugation-invariant permutations with few cycles, and conjectured that this assumption could be lifted.
    In \Cref{sec: results RS} we provide a positive answer to this conjecture.
    \item[Records.] The numbers of records in uniform permutations are known to satisfy asymptotic normality.
    In \Cref{sec: results records} we prove a general limit theorem for the number of high records (or left-to-right maxima) in conjugation-invariant permutations, showcasing a phase transition as the number of fixed points grows.
    More precisely, the number of high records interpolates from logarithmic with Gaussian fluctuations to asymptotically Gamma, when the number of non-fixed points is of order $n/\sqrt{\log n}$.
    We also find the first order asymptotics of low records (or left-to-right minima), and their second order asymptotics in some cases.
    \item[Pattern counts.] Janson~et~al.~\cite{JNZ15} established asymptotic normality for all pattern counts in uniform permutations.
    Kammoun \cite{K22} and Hamaker and Rhoades \cite{HR22} extended this to conjugation-invariant permutations, under certain assumptions on their cycles.
    The authors of \cite{HR22} conjectured that the (non-degenerate) asymptotic normality of pattern counts holds for all conjugation-invariant permutations, where the asymptotic variance depends on the numbers of fixed points and $2$-cycles.
    In \Cref{sec: results pattern} we partially answer this conjecture, by proving asymptotic normality for all conjugation-invariant permutations, and non-degeneracy in some cases.
\end{description}

Our results can be applied to most models of conjugation-invariant permutations, with virtually no restrictions on their cycle types.
For example, they are novel for random permutations with cycle weights, some of which have a lot of cycles \cite{BUV11,EU14}, and for central virtual permutations, some of which have a macroscopic number of fixed points \cite{T00}.

\subsection{A geometric construction}
\label{sec: geometric construction}

We begin by recalling the notion of standardization:
if $(y_1, \dots, y_n)$ is a sequence of pairwise distinct numbers, define $\tau = \st{y_1,\dots,y_n}$ as the unique permutation of $[n]$ such that:
\[
    \text{for all }1\le i,j\le n,\quad \tau(i)<\tau(j) \Leftrightarrow y_i<y_j .
\]
The following two results are classical and easy to check:

\begin{lem}\label{lem: standardization}
    Let $\mu$ be an atomless probability distribution on $\R$, and let $Y_1, \dots, Y_n$ be i.i.d.~random variables distributed under $\mu$.
    Then $Y_1, \dots, Y_n$ are a.s.~pairwise distinct, and the permutation $\st{Y_1,\dots,Y_n}$ is uniformly random in $\S_n$.
\end{lem}

\begin{lem}\label{lem: randomly conjugating}
    Let $t$ be a cycle type of size $n$, and let $\fs$ be an arbitrary fixed $t$-cyclic permutation.
    If $\sigma$ is a uniformly random permutation of $\S_n$ then $\sigma\circ \fs\circ \sigma^{-1}$ is a uniformly random $t$-cyclic permutation.
\end{lem}

Now, consider a finite subset $\pts=\{Z_1=(X_1,Y_1),\dots,Z_n=(X_n,Y_n)\}$ of $\R^2$ with no redundant x- or y-coordinates.
We can define a permutation $\tau$ of $[n]$ by letting $\tau(i)=j$ when the $i$-th point from the left in $\pts$ is the $j$-th point from the bottom;
equivalently, $\tau = \st{Y_{(1)}, \dots, Y_{(n)}}$ where $Z_{(1)}, \dots, Z_{(n)}$ are ordered such that $X_{(1)}< \dots< X_{(n)}$.
We denote by $\tau=\perm{\pts}$ this permutation.

By considering adequate planar point processes $\pts$, we may obtain random permutations $\tau$ with various laws: 
for instance if $Z_1,\dots,Z_n$ are i.i.d.~$\Unif{\carre}$ variables then $\tau$ is uniformly random, and if we symmetrize this family with respect to the diagonal of $\carre$ then $\tau$ is a uniform involution of size $2n$ with no fixed point.
The latter construction was notably used by Baik and Rains in \cite{BR01}.
An asset of this point of view is that some properties of $\tau$ may be easier to derive directly from $\pts$: see e.g.~\cite{AD95} for a study of increasing subsequences in uniform permutations, and \cite{K06} for the case of involutions.
\Cref{lem: geometric construction t-cyclic} generalizes the geometric construction of Baik and Rains for random involutions \cite{BR01} to uniform permutations in any conjugacy class.
Although simple, this lemma is fundamental to our analysis:
general conjugation--invariant permutations do not benefit from the same valuable algebraic properties as uniform permutations or involutions, and our geometric construction bypasses this problem.

\begin{lem}\label{lem: geometric construction t-cyclic}
    Let $t$ be a cycle type of size $n$, and let $\fs\in\S_n$ be an arbitrary $t$-cyclic permutation.
    Let $\left(U_i\right)_{i\in[n]}$ be a sequence of i.i.d.~$\Unif{[0,1]}$ random variables.
    Set $Z_i := \left(U_i,U_{\fs(i)}\right)$ for each $i\in[n]$, and $\pts := \left\{ Z_i, i\in[n] \right\}$.
    Then $\tau:=\perm{\pts}$ is a uniform $t$-cyclic permutation.
\end{lem}

\begin{figure}
    \centering
    \includegraphics[scale=1.4]{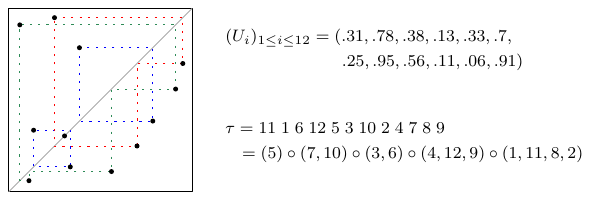}
    \caption{The geometric construction of a random $t$-cyclic permutation for $t=(1,2,1,1)$.
    Here we used the arbitrary $t$-cyclic permutation $\fs := (1)\circ (2,3)\circ (4,5)\circ (6,7,8)\circ (9,10,11,12)$.
    The resulting permutation $\tau$ is written in one-line notation and in cycle product notation, and its cycles are highlighted on the point set.
    For example, the points $(U_6,U_7), (U_7,U_8), (U_8,U_6)$ are linked by the red dotted line, and correspond to the $3$-cycle $(9,4,12)$ in $\tau$.}
    \label{fig: exemple_t_cyclic}
\end{figure}

See \Cref{fig: exemple_t_cyclic} for an illustration.
In practice, instead of an arbitrary $t$-cyclic permutation $\fs$, we will use a canonical one on an \textit{ad hoc} set of indices:
let $\Vert_t := \left\{\big. (p,k,l) \,:\, p\in[n],k\in[t_p],l\in[p]\right\}$, and $\s : (p,k,l) \mapsto (p,k,l\p1)$ where $l\p1$ is taken modulo $p$.
Then the construction of \Cref{lem: geometric construction t-cyclic} works with a family $\left(U_{p,k}^l\right)_{(p,k,l)\in\Vert_t}$ of i.i.d.~$\Unif{[0,1]}$ random variables, the points $Z_{p,k}^l := \left(U_{p,k}^l,U_{p,k}^{l+1}\right)$, and the point set $\pts := \left\{Z_{p,k}^l : (p,k,l)\in\Vert_t\right\}$.
With this notation, we say that $\pts$ is a \emph{geometric construction of $t$}.

\begin{proof}[Proof of \Cref{lem: geometric construction t-cyclic}]
    Let $\sigma = (\sigma_1, \dots, \sigma_n) := \st{U_1, \dots, U_n}$, i.e.~$U_i$ is the $\sigma_i$-th lowest number in $\mathcal{U} := \left\{U_i, i\in[n]\right\}$.
    Observe that $\mathcal{U}$ is the list of x-coordinates of $\pts$ and its list of y-coordinates as well;
    consequently, $U_i$ is the $\sigma_i$-th lowest x-coordinate of points in $\pts$ and $U_{\fs(i)}$ is the $\sigma_{\fs(i)}$-th lowest y-coordinate of points in $\pts$..
    Since $\left( U_i, U_{\fs(i)} \right)$ is in $\pts$, this precisely means that $\tau(\sigma_i) = \sigma_{\fs(i)}$.
    Therefore, $\tau = \sigma\circ \fs\circ \sigma^{-1}$.
    Using Lemmas~\ref{lem: standardization} and~\ref{lem: randomly conjugating}, this concludes the proof.
\end{proof}

As we explain in \Cref{sec: results LIS LDS,sec: results RS,sec: results records}, monotone subsequences and records of $\tau=\perm{\pts}$ can be read directly on $\pts$.
This allows us to study these statistics of the permutation by investigating the behavior of the point process in certain regions of the plane.
To this aim, a list of useful properties of the geometric construction may be found in \Cref{sec: prelim}.
Regarding pattern counts, we rather use the ``weak dependency'' of points in $\pts$ (see \Cref{sec: results pattern} for more details).

\subsection{Notation}
\label{sec: notation}

We denote by $\S_n$ the set of permutations of $[n]$ and by $\Delta$ the diagonal of $\carre$.
Throughout this paper we consider the standard partial order on $\R^2$ defined by $(x,y) \le (x',y')$ when $x\le x'$ and $y\le y'$.\\

Let $t$ be a cycle type of size $n$, and let $\pts$ be a geometric construction of $t$.
We can decompose it into its cycles $\pts = \bigcup_{p\in [n] , k\in [t_p]}\pts_{p,k}$ where $\pts_{p,k} := \left\{ Z_{p,k}^l : l\in[p] \right\}$.
We can also decompose it into $\pts_\Delta := \pts\cap\Delta$ and $\check\pts := \pts\setminus\Delta$, and we write $\check t := (0,t_2,t_3,\dots,t_n)$ for the cycle type of size $\check n := n-t_1$ with no fixed point.
It is straightforward to check that $\check\pts$ is a geometric construction of $\check t$.

Define a graph $\L_t:=(\Vert_t,\Edges_t)$ with vertice $\Vert_t$ and edges between $i$ and $\s(i)$ for each $i\in\Vert_t$ (forgetting self-loops).
We refer to $\L_t$ as the ``dependency graph'' associated with $t$.
See \Cref{fig: ex dependency graph} for an example, and \Cref{sec: dependency graph} for an explanation of this name.\\

\begin{figure}
    \centering
    \includegraphics[scale=.8]{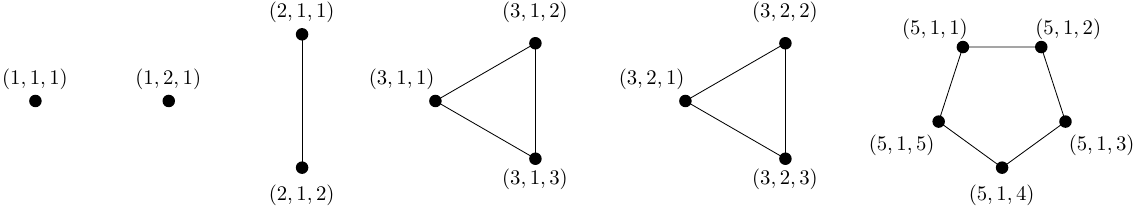}
    \caption{The dependency graph $\L_t$ associated with the cycle type $t=(2,1,2,0,1)$.}
    \label{fig: ex dependency graph}
\end{figure}

Let $(a_n)_{n\in\N}$ be a sequence of positive numbers.
We say that a sequence of random variables $(X_n)_{n\in\N}$ is a $\O_\P(a_n)$ as $n\to\infty$ if 
\begin{equation*}
    \sup_{n\in\N} \prob{\card{X_n} > a_n M} \cv{M}{+\infty} 0 .
\end{equation*}
Also, $(X_n)_{n\in\N}$ is a $o_\P(a_n)$ as $n\to\infty$ if $X_n/a_n \cv{n}{\infty} 0$ in probability.
Finally, a sequence $(A_n)_{n\in\N}$ of events happens with high probability (w.h.p.) as $n\to\infty$ if $\prob{A_n} \cv{n}{\infty} 1$.

\section{Results: universality and beyond}

\subsection{Longest monotone subsequences}
\label{sec: results LIS LDS}

Let $\tau$ be a permutation of $[n]$.
An increasing subsequence of $\tau$ is a sequence of indices $i_1<\dots<i_\ell$ such that $\tau(i_1)<\dots<\tau(i_\ell)$, and the maximum length of an increasing subsequence of $\tau$ is denoted by $\LIS(\tau)$.
We may define analogously $\LDS(\tau)$ as the maximum length of a decreasing subsequence of $\tau$.

The study of monotone subsequences in permutations has been an active and captivating field of research; see \cite{R15} for an introduction.
In particular when $\tau_n$ is uniformly random, the law of $\LIS(\tau_n)$ (which is the same as that of $\LDS(\tau_n)$ in this case) is now well-understood: it behaves asymptotically as $2\sqrt n$ \cite{VK77}, satisfies a large deviation principle \cite{DZ99,S98}, and admits Tracy--Widom fluctuations of order $n^{1/6}$ \cite{BDJ99}.

One might expect these results to mostly hold when the law of $\tau_n$ is conjugation-invariant.
The first advance on this question was obtained by Baik and Rains, who studied random involutions \cite{BR01}.
More recently, Kammoun \cite{K18,K23} and Guionnet \cite{GK23} could settle the case of conjugation-invariant random permutations with few cycles.
Here we start by stating first order asymptotic results and concentration inequalities for the longest decreasing subsequences of uniform permutations with given cycle types and no fixed points.

\begin{theorem}\label{th: LDS t-cyclic}
    For any $\delta>0$ there exists $c_\delta>0$ such that the following holds.
    For each $n$, let $t^{(n)}$ be a cycle type of size $n$ such that $t_1^{(n)}=0$ and let $\tau_n$ be a uniform $t^{(n)}$-cyclic permutation.
    Then:
    \begin{equation*}
        \prob{\frac{1}{\sqrt{n}}\LDS(\tau_n) < 2(1-\delta)} \le \exp\left(\big.-c_\delta n\right)
        \quad\text{;}\quad 
        \prob{\frac{1}{\sqrt{n}}\LDS(\tau_n) > 2(1+\delta)} \le \exp\left(-c_\delta \sqrt{n}\right) .
    \end{equation*}
    In particular $\frac{1}{\sqrt{n}}\LDS(\tau_n)\cv{n}{\infty}2$ in probability and in $L^q$ for all $q\ge1$.
\end{theorem}

It is not difficult to add fixed points in the previous theorem, and to consider random cycle types instead of deterministic ones.
This yields the following corollary.

\begin{cor}\label{cor: LDS conjugation-invariant}
    For any $\delta>0$ there exists $c_\delta>0$ such that the following holds.
    Let $(\tau_n)_n$ be a sequence of conjugation-invariant random permutations, where $\tau_n$ has size $n$ and (random) cycle type $t^{(n)}$.
    Then for any $n$:
    \begin{equation*}
    \left\{
        \begin{array}{ll}
        \prob{\LDS(\tau_n) < 2(1-\delta)\sqrt{n-t_1^{(n)}}} \le \mean{\exp\left(\big.-c_\delta \left(n-t_1^{(n)}\right)\right)} ;\vspace{.5em}\\ 
        \prob{\LDS(\tau_n) > 2(1+\delta)\sqrt{n-t_1^{(n)}}} \le \mean{\exp\left(-c_\delta \sqrt{n-t_1^{(n)}}\right)} .
        \end{array}
    \right.
    \end{equation*}
    In particular if $n-t_1^{(n)}\to\infty$ in probability then $\frac{1}{1+\sqrt{n-t_1^{(n)}}}\LDS(\tau_n) \to 2$ in probability and in $L^q$ for all $q\ge1$.
\end{cor}

Although increasing and decreasing subsequences share the same distribution in uniformly random permutations, this fact does not hold for conjugation-invariant random permutations.
The study of increasing subsequences in this context is trickier, and we were only able to derive their first order asymptotics up to a multiplicative constant.
As before, we start by stating our result for uniform permutations with given cycle types and no fixed points, and then we generalize it to all conjugation-invariant permutations.

\begin{prop}\label{th: LIS t-cyclic}
    For any $\delta>0$ there exists $c_\delta>0$ such that the following holds.
    For each $n$, let $t^{(n)}$ be a cycle type of size $n$ such that $t_1^{(n)}=0$ and let $\tau_n$ be a uniform $t^{(n)}$-cyclic permutation.
    Then:
    \begin{equation*}
        \prob{\frac{1}{\sqrt{n}} \LIS(\tau_n) < 2(1-\delta) } \le \exp\left(-c_\delta n\right)
        \quad\text{ ; }\quad
        \prob{\frac{1}{\sqrt{n}} \LIS(\tau_n) > 2(1+\delta)\sqrt{3} } \le \exp\left(-c_\delta \sqrt{n}\right).
    \end{equation*}
\end{prop}

\begin{cor}\label{cor: LIS conjugation-invariant}
    For any $\delta>0$ there exists $c_\delta>0$ such that the following holds.
    Let $(\tau_n)_n$ be a sequence of conjugation-invariant random permutations, where $\tau_n$ has size $n$ and (random) cycle type $t^{(n)}$.
    Then for any $n$:
    \begin{equation*}
    \left\{
        \begin{array}{ll}
        \prob{\big.\LIS(\tau_n) < 2(1-\delta)\sqrt{n}}
        \le \exp\left(-c_\delta n\right) ;\vspace{.5em}\\
        \prob{\LIS(\tau_n) > t_1^{(n)} + 2(1+\delta)\sqrt{3\left(n-t_1^{(n)}\right)}}
        \le \mean{\exp\left(-c_\delta \sqrt{n-t_1^{(n)}}\right)} .
        \end{array}
    \right.
    \end{equation*}
    In particular if $\frac{1}{\sqrt n}t_1^{(n)}\to\alpha\ge0$ in probability then $\frac{1}{\sqrt n}\LIS(\tau_n)$ is asymptotically bounded between $2$ and $\alpha+2\sqrt 3$ in probability.
\end{cor}

Let us briefly explain the techniques used to prove the results of this section.
Deuschel and Zeitouni \cite{DZ95} studied the first order asymptotics of longest monotone subsequences in $\perm{Z_1,\dots,Z_n}$, where $Z_1,\dots,Z_n$ are i.i.d.~under certain probability densities on $\carre$ (such random permutations are called {\em locally uniform} in \cite{S23}).
This is made possible by the following key observation: 
increasing and decreasing subsequences of $\tau=\perm{\pts}$ correspond to up-right and down-right paths of points in $\pts$.
The authors of \cite{DZ95} then slice the unit square into small rectangles in which the points are almost uniformly distributed, so that they may locally apply the result of \cite{VK77}, and finally stitch these local paths of points into one global path.
When studying uniform permutations in given conjugacy classes, the same method can be adapted to the geometric construction of \Cref{lem: geometric construction t-cyclic}.

Finally, we state a reasonable conjecture in the direction of \Cref{th: LIS t-cyclic} and \Cref{cor: LIS conjugation-invariant}.
It was already established for random involutions in \cite{BR01}, and for conjugation-invariant random permutations satisfying certain cycle constraints in \cite[Theorem~1.2]{K18} and \cite[Theorem~3]{K23}.

\begin{conj}\label{conj: LIS conjugation-invariant w/o fixed points}
    Let $(\tau_n)_n$ be a sequence of conjugation-invariant random permutations, where $\tau_n$ has size $n$ and (random) cycle type $t^{(n)}$.
    Suppose that $t_1^{(n)}/\sqrt{n}\rightarrow0$ in probability.
    Then
    \begin{equation*}
        \frac{1}{\sqrt{n}}\LIS(\tau_n) \cv{n}{\infty} 2
    \end{equation*}
    in probability.
\end{conj}

\begin{rem}
    Baik and Rains actually proved \cite[Theorems~3.2~and~3.4]{BR01} that if $\tau_n$ is a uniform $t^{(n)}$-cyclic permutation where $t^{(n)}~\hspace{-.3em}=~\hspace{-.3em}\left(t_1^{(n)},t_2^{(n)},0,\dots\right)$ satisfies $t_1^{(n)}/\sqrt n \to \alpha\ge0$, i.e.~if $\tau_n$ is a uniform involution with $\alpha\sqrt n + o(\sqrt n)$ fixed points, then
    \begin{equation*}
        \frac{1}{\sqrt n}\LIS(\tau_n) 
        \cv{n}{\infty} 
        \left\{
        \begin{array}{ll}
            2 &\text{if } \alpha\le1 \vspace{.5em}\\
            \alpha+\frac{1}{\alpha} &\text{if } \alpha\ge1
        \end{array}
        \right.
    \end{equation*}
    in probability.
    We could expect the same first order asymptotics for any sequence of conjugation-invariant random permutations such that $t_1^{(n)}/\sqrt n \to \alpha$ in probability.
    However we chose not to include this refinement of \Cref{conj: LIS conjugation-invariant w/o fixed points} in its statement, as it might be more speculative.
\end{rem}

The results of this section are proved in \Cref{sec: proof LIS LDS}.

\subsection{Robinson--Schensted shape}
\label{sec: results RS}

Let $\tau$ be a permutation of $[n]$.
If $x\ge0$, we call {\em $x$-decreasing subsequence} any union of $\lfloor x\rfloor$ individually decreasing subsequences, and we define $\LDS_x(\tau)$ as the maximal size of an $x$-decreasing subsequence of $\tau$. 

The study of $x$-decreasing subsequences in permutations is partly motivated by their fundamental link with the {\em Robinson--Schensted correspondence}.
This is a one-to-one correspondence between permutations and pairs of standard Young tableaux with the same shape.
This shape, which is a Young diagram, is called the {\em RS shape} of the permutation:
it appears in several domains, such as integrable probability or representation theory.
A well-known theorem of Greene \cite{G74} states that for any integer $k$, the number of boxes in the first $k$ columns of the RS shape of $\tau$ equals $\LDS_k(\tau)$.
Equivalently, the length of the $k$-th column equals $\LDS_k(\tau) - \LDS_{k\m1}(\tau)$.
A similar connection can be made between increasing subsequences of $\tau$ and the row lengths of its RS shape.

Historically, studying the asymptotics of $\LIS(\tau_n)$ and $\LDS(\tau_n)$ when $\tau_n$ is uniformly random actually required finding the asymptotics of its entire RS shape.
The law of this shape is known as the Plancherel measure, for which Vershik and Kerov and simultaneously Logan and Shepp established the following.
\begin{theorem}[\cite{LS77,VK77}]\label{th: LSKV uniform}
    For each $n$, let $\tau_n$ be a uniformly random permutation of $[n]$.
    Then there exists an explicit nondecreasing, concave function $F_\mathrm{LSKV} : \R_+ \to \R_+$ such that for each $r\ge0$:
    \begin{equation*}
        \frac1n \LDS_{r\sqrt n}(\tau_n) \cv{n}{\infty} F_\mathrm{LSKV}(r)
    \end{equation*}
    in probability.
\end{theorem}
According to Greene's theorem, $F_\LSKV(r)$ thus describes the asymptotic proportion of boxes in the first $\lfloor r\sqrt n\rfloor$ columns of the RS shape of $\tau_n$, and the derivative of $F_\LSKV$ describes the limit curve of this shape.

A classical property of the Robinson--Schensted correspondence is that it induces a bijection between involutions and (single) standard Young tableaux.
If $(\tau_n)$ is a sequence of uniformly random involutions, the proofs of \cite{LS77} and \cite{VK77} can then be adapted to show that \Cref{th: LSKV uniform} still holds (see e.g.~Equation~(24) in \cite{M11} and the discussion below).
More generally if $(\tau_n)$ is a sequence of conjugation-invariant random permutations, Kammoun proved that \Cref{th: LSKV uniform} still holds under the assumption that the number of cycles in $\tau_n$ is sublinear \cite[Theorem~1.8]{K18}.
He later conjectured that this assumption could be lifted, provided we take into account the proportion of fixed points \cite[Conjectures 6~and~7]{K23}.
The main difficulty is that none of the previous approaches (via hook-length formula \cite{LS77}, representation theory \cite{IO02,M11}, coupling \cite{K18}...) work in this wider setting.
Here, using our geometric construction, we solve this conjecture.

\begin{theorem}\label{th: limit shape w/o fixed points}
    For each $n$, let $t^{(n)}$ be a cycle type of size $n$ with no fixed point and let $\tau_n$ be a uniform $t^{(n)}$-cyclic permutation.
    Then for each $r\ge0$:
    \begin{equation*}
        \frac1n \LDS_{r\sqrt n}(\tau_n) \cv{n}{\infty} F_\mathrm{LSKV}(r)
    \end{equation*}
    in probability.
\end{theorem}

\begin{cor}\label{cor: limit shape conjugation-invariant}
    Let $(\tau_n)_n$ be a sequence of conjugation-invariant random permutations, where $\tau_n$ has size $n$ and (random) cycle type $t^{(n)}$.
    Suppose that $n-t_1^{(n)}\to\infty$ in probability.
    Then for each $r\ge0$:
    \begin{equation*}
        \frac{1}{n-t_1^{(n)}} \LDS_{r\sqrt{n-t_1^{(n)}}}(\tau_n) \cv{n}{\infty} F_\mathrm{LSKV}(r)
    \end{equation*}
    in probability.
\end{cor}

Let us briefly discuss the strategy of proof.
As in \Cref{sec: results LIS LDS}, $k$-decreasing subsequences of $\tau = \perm{\pts}$ have a nice visual interpretation:
they correspond to unions of $k$ down-right paths in $\pts$.
In the continuity of \cite{DZ95}, where the longest monotone subsequences of locally uniform permutations were studied with geometric tools, Sj\"ostrand \cite{S23} recently studied the limit RS shape of locally uniform permutations using more advanced analysis.
His approach broadly consisted in computing the ``score'' associated with a bundle of decreasing curves, by decomposing the unit square into small regions where the sampling density is almost constant and the curves are almost parallel.
In \Cref{lem: iid outside diagonal t-cyclic} we show that the points of our geometric construction, when restricted to any rectangle outside the diagonal of the unit square, are i.i.d.~uniform.
This simple fact is at the basis of how we apply and adapt the results and methods of \cite{S23}.

The results of this section are proved in \Cref{sec: proof RS}.

\subsection{Records}
\label{sec: results records}

Let $\tau$ be a permutation of $[n]$.
We say that a position $i\in[n]$ is a {\em high record} (or left-to-right maximum) for $\tau$ if for any $j<i$, we have $\tau(j)<\tau(i)$.
We say that it is a {\em low record} (or left-to-right minimum) for $\tau$ if for any $j<i$, we have $\tau(j)>\tau(i)$.
We denote by $\hrec{\tau}$ and $\lrec{\tau}$ the numbers of high and low records in $\tau$, respectively.
Records and their many variants are standard statistics in enumerative combinatorics.

The numbers of low and high records in uniformly random permutations are well-understood: thanks to a classical bijection they follow the same law as the number of cycles (see e.g.~\cite[Section~1.2]{ABNP16}), and this yields the following asymptotic normality (found e.g.~in \cite[Equation~(1.31)]{ABT03}).

\begin{theorem}\label{th: records uniform}
    For each $n$, let $\tau_n$ be a uniformly random permutation of $[n]$. 
    Then we have the following convergence in distribution:
    \begin{equation*}
        \frac{\hrec{\tau_n}-\log n}{\sqrt{\log n}} \cv{n}{\infty} \Normal{0}{1} .
    \end{equation*}
    The same holds for $\lrec{\tau_n}$.
\end{theorem}

Records can be interpreted nicely when the permutation is obtained from a family of points in the plane.
If $\tau = \perm{\pts}$, a point $Z=(X,Y)\in\pts$ corresponds to a high record in $\tau$ when no other point lies in its up-left corner $[0,X]\times[Y,1]$.
Similarly, $Z$ corresponds to a low record in $\tau$ when no other point lies in its down-left corner $[0,X]\times[0,Y]$.
We write $\hrec{\pts}$ and $\lrec{\pts}$ as shortcuts for $\hrec{\perm{\pts}}$ and $\lrec{\perm{\pts}}$.

There is no more link between the numbers of cycles and records in conjugation-invariant permutations, but the geometric construction enables a new approach.
We prove a limit theorem for the number of high records, showcasing a phase transition as the number of fixed points grows.
First recall that, for $n\in\N^*$ and $\lambda>0$, the $\Gam{n}{\lambda}$ distribution is defined as the $n$-th convolution power of the $\Exp{\lambda}$ distribution.
By convention, $\Gam{n}{+\infty}$ is the a.s.~null distribution.

\begin{theorem}\label{th: high records t-cyclic}
    For each $n$, let $t^{(n)}$ be a cycle type of size $n$ and let $\tau_n$ be a uniform $t^{(n)}$-cyclic permutation.
    Write $\check n := n-t_1^{(n)}$ for the number of non-fixed points.
    Suppose that $\check n \cv{n}{\infty} \infty$ and that
    \begin{equation*}
        \frac{\check n}{n/\sqrt{\log n}} \cv{n}{\infty} \alpha
    \end{equation*}
    for some $\alpha\in[0,\infty]$.
    Then we have the following convergence in distribution:
    \begin{equation*}
        \frac{\hrec{\tau_n} - \log\left(\check n\right)}{t_1^{(n)}/\check n + \sqrt{\log\left(\check n\right)}}
        \cv{n}{\infty}
        \frac{\alpha}{\alpha+1} Y
        + \frac{1}{\alpha+1} \Gamma_2
    \end{equation*}
    where $Y$ and $\Gamma_2$ are independent $\Normal{0}{1}$ and $\Gam{2}{1}$ random variables.
    In particular if $\check n = \omega\left(\frac{n}{\sqrt{\log n}}\right)$ as $n\to\infty$ then:
    \begin{equation*}
        \frac{\hrec{\tau_n} - \log n}{\sqrt{\log n}}
        \cv{n}{\infty}
        \Normal{0}{1} .
    \end{equation*}
\end{theorem}

Similar to the case of monotone subsequences, the study of low records in conjugation-invariant random permutations differs from that of high records.
This asymmetry is already seen in some regimes for the first order asymptotics, as shown in the following theorem.

\begin{theorem}\label{th: low records t-cyclic}
    For each $n$, let $t^{(n)}$ be a cycle type of size $n$ and let $\tau_n$ be a uniform $t^{(n)}$-cyclic permutation.
    \begin{enumerate}
        \item If $t_1^{(n)} = \O\left(\sqrt n\right)$ then we have the following convergence in probability:
        \begin{equation*}
            \frac{\lrec{\tau_n}}{\log n} \cv{n}{\infty} 1 .
        \end{equation*}
        \item If $\sqrt n \ll t_1^{(n)} \ll n$, i.e.~$\sqrt n = o\left( t_1^{(n)} \right)$ and $t_1^{(n)} = o(n)$, then we have the following convergence in probability:
        \begin{equation*}
            \frac{\lrec{\tau_n}}{\log n - \log t_1^{(n)}} \cv{n}{\infty} 2 .
        \end{equation*}
        \item If $t_1^{(n)}=\Theta\left( n \right)$ then $\lrec{\tau_n} = \O_\P(1)$.
    \end{enumerate}
\end{theorem}

We could not compute fluctuations for the number of low records in general.
To simplify the analysis, we restrict ourselves to products of $2$-cycles and products of $3$-cycles, for which exhibit different fluctuations.

\begin{prop}\label{prop: low records fluctuations}
    For each $n$, let $t^{(n)}$ be a cycle type of size $n$ such that $t_1^{(n)}=\O\left(\sqrt n\right)$, and let $\tau_n$ be a uniform $t^{(n)}$-cyclic permutation.
    \begin{enumerate}
        \item Suppose that $t_1^{(n)}+2t_2^{(n)}=n$.
        Then we have the following convergence in distribution:
        \begin{equation*}
            \frac{\lrec{\tau_n}-\log n}{\sqrt{\log n}}
            \cv{n}{\infty} \Normal{0}{2} .
        \end{equation*}
        \item Suppose that $t_1^{(n)}+3t_3^{(n)}=n$.
        Then we have the following convergence in distribution:
        \begin{equation*}
            \frac{\lrec{\tau_n}-\log n}{\sqrt{\log n}}
            \cv{n}{\infty} \Normal{0}{1} .
        \end{equation*}
    \end{enumerate}
\end{prop}

It would be interesting to find the interpolation between the two items of \Cref{prop: low records fluctuations}.
However this would notably require understanding the interaction between $2$-cycles and other cycles, and its effect on the law of low records, which seems difficult to grasp.

\begin{rem}
    As in Corollaries \ref{cor: LDS conjugation-invariant}, \ref{cor: LIS conjugation-invariant}, \ref{cor: limit shape conjugation-invariant} and \ref{cor: pattern normality conjugation-invariant}, the results of this section remain true if $(\tau_n)$ is a sequence of conjugation-invariant random permutations and the hypotheses hold in probability.
\end{rem}

The results of this section are proved in \Cref{sec: proof records}.

\subsection{Pattern counts}
\label{sec: results pattern}

Let $\tau\in\S_n$ and $I=\{i_1<\dots<i_r\}$ be a subset of $[n]$.
We may define a permutation $\pi\in\S_r$ by the following rule:
for any $j,k\in[r]$, $\pi(j)>\pi(k)$ if and only if $\tau(i_j)>\tau(i_k)$.
We write $\pat(I,\tau)$ for this permutation, and say that this is the {\em pattern} induced by $\tau$ on $I$.

Permutation patterns are a natural notion of substructures for permutations, and as such they have attracted a lot of interest.
A central question is that of pattern counts: given a permutation $\pi$ of $[r]$, how many subsets of $\tau$ induce $\pi$ as a pattern?
In other words, we are interested in the statistic
\begin{equation*}
    X_\pi(\tau) := \sum_{I\in\binom{[n]}{r}} \One{\pat(I,\tau) = \pi}
\end{equation*}
where $\binom{[n]}{r}$ denotes the subsets of $[n]$ of size $r$.
For example, $X_{2\,1}(\tau)$ counts the number of inversions in $\tau$,
and $X_{1\,2\,\dots\,r}(\tau)$ counts the number of increasing subsequences of length $r$ in $\tau$.

Pattern counts in uniformly random permutations are known to be asymptotically normal since the works of B\'ona for monotone patterns \cite{B10}, and later Janson~et~al.~for all patterns:
\begin{theorem}[\cite{JNZ15}]\label{th: pattern counts in uniform}
    For each $n$, let $\tau_n$ be a uniformly random permutation of $[n]$.
    Then for any $r\in\N^*$:
    \begin{equation*}
        \left( \frac{X_\pi(\tau_n)-\binom{n}{r}\frac{1}{r!}}{n^{r-1/2}} \right)_{\pi\in\S_r}
        \cv{n}{\infty} \Normal{0}{\Sigma}
    \end{equation*}
    for some matrix $\Sigma$, of rank $(r\m1)^2$ and such that $\Sigma_{\pi,\pi}>0$ for all $\pi\in\S_r$ if $r\ge2$.
\end{theorem}

The convergence result in this theorem can be established via the method of {\em dependency graphs}.
To prove non-degeneracy, i.e.~that $\Sigma_{\pi,\pi}$ is positive, B\'ona was able to find a lower bound for the variance of $X_\pi$ by expressing it in terms of covariances, while the authors of \cite{JNZ15} used the method of {\em U-statistics}.

In \cite{F13}, F\'eray established a central limit theorem for pattern counts in Ewens random permutations, and conjectured its non-degeneracy.
Kammoun \cite[Proposition~31]{K22} and Hamaker and Rhoades \cite[Theorem~8.8]{HR22} could generalize the non-degenerate asymptotic normality of \Cref{th: pattern counts in uniform} to conjugation-invariant permutations, under certain conditions on their cycles.
It was conjectured in \cite[Problem~9.9]{HR22} that those conditions could be lifted, and that the asymptotic variance would only depend on the proportions of fixed points and $2$-cycles.

Here, using the geometric construction of \Cref{lem: geometric construction t-cyclic}, we extend the convergence of \Cref{th: pattern counts in uniform} to most conjugation-invariant permutations.
If $\tau_n$ is a uniform $t$-cyclic permutation then we can write, for any $\pi\in\S_r$:
\begin{equation}\label{eq: write X as sum of partly dependent}
    X_\pi(\tau_n) = \sum_{I\in\binom{\Vert_t}{r}} \One{\perm{Z_i,i\in I} = \pi}
\end{equation}
with the notation of \Cref{sec: geometric construction}.
Thanks to the dependency graph $\L_t$ defined in \Cref{sec: notation}, we derive a limit theorem for $X_\pi(\tau_n)$.
More precisely, Stein's method yields explicit bounds on the speed of convergence.
In the following, we denote by $d_K$ the Kolmogorov distance between two probability distributions.
\begin{theorem}\label{th: pattern Stein}
    Let $Y\sim\Normal{0}{1}$.
    For each $n$, let $t^{(n)}$ be a cycle type of size $n$ and $\tau_n$ be a uniform $t^{(n)}$-cyclic permutation.
    Fix $r\in\N^*$ and $\pi\in\S_r$.
    Then for any $n$:
    \begin{equation*}
        d_K\left(\frac{X_\pi(\tau_n)-\mean{X_\pi(\tau_n)}}{\sqrt{\var{X_\pi(\tau_n)}}} \,,\, Y\right)
        \le 42\frac{r}{(r\m1)!^2} \frac{n^{2r-3/2}}{\var{X_\pi(\tau_n)}}
        + 72\frac{r}{(r\m1)!^3} \frac{n^{3r-2}}{\var{X_\pi(\tau_n)}^{3/2}} .
    \end{equation*}
    Furthermore, if $\var{X_\pi(\tau_n)}=\sigma^2 n^{2r-1}+o(n^{2r-1})$ for some $\sigma\ge0$ as $n\to\infty$, then
    \begin{equation*}
        \frac{X_\pi(\tau_n)-\mean{X_\pi(\tau_n)}}{n^{r-1/2}}
        \cv{n}{\infty} \Normal{0}{\sigma^2}
    \end{equation*}
    in distribution, and if $\sigma>0$ then the previous distance is bounded by $\O(n^{-1/2})$.
\end{theorem}

Thanks to \cite[Proposition 7.2 and Theorem 8.14]{HR22}, the condition $\var{X_\pi(\tau_n)}=\sigma^2 n^{2r-1}+o(n^{2r-1})$ holds under the simple assumption that $t_1^{(n)}/n$ and $t_2^{(n)}/n$ converge.
We may still want to find an explicit formula for $\sigma$, and to understand in which cases $\sigma>0$.
For this we use the method of $U$-statistics developed by Hoeffding \cite{H48}, as was done in \cite{JNZ15}.
Before stating our results, we need a few definitions.

Fix $r\in\N^*$ and $p_1\in[0,1]$.
Then let $\hat Z_1,\dots,\hat Z_r$ be i.i.d.~variables distributed under $p_1\Leb_\Delta+(1\m p_1)\Leb_{\carre}$, where $\Leb_\Delta$ denotes the Lebesgue measure on the diagonal $\Delta$ of $\carre$.
Define, for any $z\in\carre$:
\begin{equation*}
    \psi_\pi^{p_1}(z) := \prob{ \perm{\hat Z_1,\dots,\hat Z_{r-1},z} = \pi }
\end{equation*}
and
\begin{equation*}
    \mu_\pi^{p_1} := \prob{ \perm{\hat Z_1,\dots,\hat Z_{r}} = \pi } = \mean{\psi_\pi^{p_1}\left(\hat Z_r\right)} .
\end{equation*}

\begin{theorem}\label{th: pattern normality t-cyclic with fixed points}
    Let $(t^{(n)})_n$ be a sequence of cycle types of size $n$ such that $t_1^{(n)} = n p_1 + o(\sqrt{n})$ and $2t_2^{(n)} = n p_2 + o(n)$ as $n\to\infty$, for some $p_1,p_2\in[0,1]$.
    For each $n$, let $\tau_n$ be a uniform $t^{(n)}$-cyclic permutation.
    Then for any $r\in\N^*$, we have the following convergence in distribution:
    \begin{equation*}
        \left( \frac{X_\pi(\tau_n) - \binom{n}{r}\mu_{\pi}^{p_1}}{n^{r-1/2}} \right)_{\pi\in\S_r}
        {\cv{n}{\infty}}
        \Normal{0}{\Sigma^{p_1,p_2}}
    \end{equation*}
    where for any $\pi,\rho\in\S_r$, if $U,V,W$ are i.i.d.~$\Unif{[0,1]}$ variables:
    \begin{multline*}
        (r\m1)!^2\, \Sigma_{\pi,\rho}^{p_1,p_2} = p_1\cov{\psi_\pi^{p_1}(U,U)}{\psi_\rho^{p_1}(U,U)} + (1\m p_1)\cov{\psi_\pi^{p_1}(U,V)}{\psi_\rho^{p_1}(U,V)} \\+ p_2\cov{\psi_\pi^{p_1}(U,V)}{\psi_\rho^{p_1}(V,U)} + 2(1\m p_1\m p_2)\cov{\psi_\pi^{p_1}(U,V)}{\psi_\rho^{p_1}(V,W)} .
    \end{multline*}
\end{theorem}

\begin{cor}\label{cor: pattern normality conjugation-invariant}
    Let $(\tau_n)_n$ be a sequence of conjugation-invariant random permutations, where $\tau_n$ has size $n$ and (random) cycle type $t^{(n)}$.
    Suppose that $t_1^{(n)} = np_1+o_\P(\sqrt n)$ and $2t_2^{(n)} = np_2+o_\P(n)$ as $n\to\infty$, for some $p_1,p_2\in[0,1]$.
    Then the convergence of \Cref{th: pattern normality t-cyclic with fixed points} holds.
\end{cor}

Note that \Cref{th: pattern normality t-cyclic with fixed points} states a joint convergence, rather than the marginal convergence stated in \Cref{th: pattern Stein}.
We stress that the true novelty of \Cref{th: pattern normality t-cyclic with fixed points} lies in the computation of the variance, carried out by the method of $U$-statistics, and not in the joint convergence.
Indeed, \Cref{th: pattern Stein} could have been stated for joint convergence, but we chose not to do so for convenience.

The matrix $\Sigma^{p_1,p_2}$ turns out to be rather difficult to study when $p_1\in(0,1)$.
Without further hypotheses, we could only prove that $X_\pi(\tau_n)$ satisfies non-degenerate asymptotic normality if $\pi$ is an involution.

\begin{prop}\label{prop: non degenerate involution pattern}
    If $p_1<1$ and $\pi\in\S_r$, $r\ge2$ satisfies $\pi=\pi^{-1}$, then $\Sigma_{\pi,\pi}^{p_1,p_2} >0$.
\end{prop}

When fixed points vanish, i.e.~when $p_1=0$, \Cref{th: pattern normality t-cyclic with fixed points} greatly simplifies.
This allows us to prove non-degeneracy for any non-trivial pattern and to compute the dimension spanned by the limiting Gaussian variable, as in \Cref{th: pattern counts in uniform} of \cite{JNZ15}.
We use the notation of \Cref{th: pattern normality t-cyclic with fixed points} and drop the index $p_1$ when it is null.

\begin{theorem}\label{th: pattern normality t-cyclic w/o fixed points}
    Let $(t^{(n)})_n$ be a sequence of cycle types of size $n$ such that $t_1^{(n)} = o(\sqrt{n})$ and $2t_2^{(n)} = n p_2 + o(n)$ as $n\to\infty$, for some $p_2\in[0,1]$.
    For each $n$, let $\tau_n$ be a uniform $t^{(n)}$-cyclic permutation.
    Then for any $r\in\N^*$, we have the following convergence in distribution:
    \begin{equation*}
        \left( \frac{X_\pi(\tau_n) - \binom{n}{r}\frac{1}{r!}}{n^{r-1/2}} \right)_{\pi\in\S_r}
        {\cv{n}{\infty}}
        \Normal{0}{\Sigma^{p_2}}
    \end{equation*}
    where for any $\pi,\rho\in\S_r$, if $U,V$ are i.i.d.~$\Unif{[0,1]}$ variables:
    \begin{equation*}
        (r\m1)!^2\, \Sigma_{\pi,\rho}^{p_2} = \cov{\psi_\pi(U,V)}{\psi_\rho(U,V)} + p_2\cov{\psi_\pi(U,V)}{\psi_\rho(V,U)} .
    \end{equation*}
    The matrix $\Sigma^{p_2}$ has rank $(r\m1)^2$ for any $p_2<1$, and rank $r(r\m1)/2$ for $p_2=1$.
    Moreover, $\Sigma_{\pi,\pi}^{p_2}>0$ for any $p_2\in[0,1]$ and $\pi\in\S_r$, $r\ge2$.
\end{theorem}

\begin{rem}
    In parallel to this work, F\'eray and Kammoun \cite{FK23} were able to prove a more general result than \Cref{th: pattern normality t-cyclic with fixed points} and \Cref{prop: non degenerate involution pattern}, using the method of {\em weighted} dependency graphs.
    They could prove non-degeneracy for any (classical) non-trivial pattern, i.e.~that $\Sigma_{\pi,\pi}^{p_1,p_2} >0$ for any $p_1<1$ and $\pi\in\S_r$, $r\ge2$.
    However, bounds on the speed of convergence as in our \Cref{th: pattern Stein} seem to be out of reach for their methods.
\end{rem}

\begin{rem}
    Our method of (non-weighted) dependency graphs also allows the application of \cite[Theorem~2.1]{J04} to \Cref{eq: write X as sum of partly dependent} for large deviation estimates.
    Namely, for any $t>0$:
    \begin{equation*}
        \prob{\big. \card{X_\pi(\tau_n)-\mean{X_\pi(\tau_n)}} \ge t}
        \le 2\exp\left( \frac{-2 (r-1)!}{3} \frac{t^2}{n^{2r-1}} \right)
    \end{equation*}
\end{rem}

\begin{rem}
    A popular generalization of patterns is that of {\em vincular patterns}, where we require some positions of the pattern to be adjacent.
    Descents, peaks and valleys are classical examples of vincular patterns.
    Asymptotic normality of vincular pattern counts in uniform permutations was established in \cite{H18}.
    Regarding conjugation-invariant permutations, the convergence results of \cite{FK23,HR22,K22} all hold within the general context of vincular patterns.
    Unfortunately, it seems unlikely that our methods would work for vincular patterns.
\end{rem}

The results of this section are proved in \Cref{sec: proof pattern}.

\section{Properties of the geometric construction}
\label{sec: prelim}

In this section, we present several interesting facts about the geometric construction of \Cref{lem: geometric construction t-cyclic}.
These will be essential to prove the results of this paper, generally by reducing the study of conjugation-invariant permutations to uniformly random permutations.
We use the notation of \Cref{sec: geometric construction,sec: notation}.

\begin{lem}\label{lem: decomposition t-cyclic}
    Let $t$ be a cycle type of size $n$ such that $t_1=0$, and let $\pts$ be a geometric construction.
    Then it can be decomposed into an a.s.~disjoint union
    $$\pts = \pts^{(1)}\cup\pts^{(2)}\cup\pts^{(3)}$$
    where each $\pts^{(i)}$ is a family of i.i.d.~uniform points in $\carre$, of (deterministic) sizes $n_1,n_2,n_3$ bounded between $n/3\m1$ and $n/3\p1$.
\end{lem}

\begin{proof}
    The idea is to define the three subsets $\pts^{(1)},\pts^{(2)},\pts^{(3)}$ by assigning them the points of each cycle in an alternating order.
    The rule for ensuring the i.i.d.~uniformity property is that whenever some point $Z_{p,k}^l$ lies in $\pts^{(i)}$, then the adjacent points $Z_{p,k}^{l\pm1}$ shall not belong to $\pts^{(i)}$.
    To construct the $\pts^{(i)}$'s in a balanced way while satisfying this rule, we do the following:
    \begin{itemize}
        \item For each $p$ such that $p=0\!\mod3$ and $1\le k\le t_p$, the points of $\pts_{p,k}$ can be put in the three subsets by simply alternating:
        \begin{equation*}
            Z_{p,k}^1\in\pts^{(1)} , 
            Z_{p,k}^2\in\pts^{(2)} , 
            Z_{p,k}^3\in\pts^{(3)} , 
            Z_{p,k}^4\in\pts^{(1)} , 
            \dots ,
            Z_{p,k}^p\in\pts^{(3)} .
        \end{equation*}
        Then each $\pts_{p,k}\cap\pts^{(i)}$ has size $p/3$ and contains i.i.d.~uniform points.
        \item For each $p$ such that $p=1\!\mod3$ and $1\le k\le t_p$, the points of $\pts_{p,k}$ can be put in the three subsets by simply alternating, but then one subset will get a surplus. 
        Thus two of the subsets $\pts_{p,k}\cap\pts^{(i)}$ have size $(p\m1)/3$, the other one has size $1+(p\m1)/3$, and they each contain i.i.d.~uniform points.
        Namely if $(i_1,i_2,i_3)$ is a permutation of $(1,2,3)$ then the choice
        \begin{equation*}
            Z_{p,k}^1\in\pts^{(i_1)} , 
            Z_{p,k}^2\in\pts^{(i_2)} , 
            Z_{p,k}^3\in\pts^{(i_3)} , 
            Z_{p,k}^4\in\pts^{(i_1)} , 
            \dots ,
            Z_{p,k}^{p-1}\in\pts^{(i_3)} ,
            Z_{p,k}^p\in\pts^{(i_2)} 
        \end{equation*}
        yields a surplus for $\pts^{(i_2)}$.
        For each such $(p,k)$ we can choose which subset gets a surplus in order to balance out those surpluses over the $(p,k)$'s. 
        Consequently the sets
        \begin{equation*}
            \bigcup_{p=1\hspace{-.7em}\mod3 , 1\le k\le t_p}\pts_{p,k}\cap\pts^{(i)}
        \end{equation*}
        have sizes of the form $m+\epsilon_i$ for some integer $m$ and $\epsilon_i\in\{0,1\}$, $i=1,2,3$.
        \item For each $p=2\!\mod3$ and $1\le k\le t_p$ a similar reasoning can be made, except this time one subset will get a deficit.
        By balancing out the deficits as before, we can construct the sets 
        \begin{equation*}
            \bigcup_{p=2\hspace{-.7em}\mod3 , 1\le k\le t_p}\pts_{p,k}\cap\pts^{(i)}
        \end{equation*}
        in such a way that they each contain i.i.d.~uniform points and have sizes of the form $m'-\delta_i$ for some integer $m'$ and $\delta_i\in\{0,1\}$, $i=1,2,3$.
    \end{itemize}
    In the end, by balancing out the surpluses with the deficits, the subsets $\pts^{(i)}$ have sizes $m''+\eta_i$ for some integer $m''$ and $\eta_i\in\{0,1\}$, $i=1,2,3$, and each contain i.i.d.~uniform points.
    The lemma follows.
\end{proof}

Thanks to \Cref{lem: decomposition t-cyclic}, we can apply Chernoff-type bounds to control the number of points in any zone of the plane.
Recall that $\check\pts := \pts\setminus\Delta$ is a geometric construction for $\check t:=(0,t_2,\dots)$ of size $\check n:=n-t_1$.

\begin{cor}\label{cor: controlling number of points t-cyclic}
    Let $t$ be a cycle type of size $n$ and $\pts$ be a geometric construction.
    If $C$ is a measurable subset of $\carre$ then $\card{\check\pts\cap C}$ is distributed like a (correlated) sum of $\Binomial{n_i}{\Leb(C)}$ variables, where $\card{n_i - \check n /3} \le 1$ for $i=1,2,3$.
    In particular for any $\delta>0$ there exists a constant $c>0$, which depends only on $\delta$ and $\Leb(C)$, such that:
    \begin{equation*}
        \prob{\card{\frac{1}{\check n}\card{\check\pts\cap C} - \Leb(C)} > \delta} \le \exp\left({-c\check n}\right) .
    \end{equation*}
\end{cor}

The following two lemmas informally state that conjugation-invariant random permutations are, in some sense, ``locally uniform''.

\begin{lem}\label{lem: iid outside diagonal t-cyclic}
    Let $t$ be a cycle type and $\pts$ be a geometric construction. 
    If $I,J$ are two essentially disjoint intervals of $[0,1]$ then the set $\pts\cap (I\times J)$ is a (random-sized) family of i.i.d.~uniform points in $I\times J$.
\end{lem}

\begin{proof}
    It suffices to prove the lemma for each set $\pts_{p,k}$.
    Fix $(p,k)$.
    Note that if some point $\big(U_{p,k}^l,U_{p,k}^{l+1}\big)$ is in $I\times J$ then a.s.~neither $\big(U_{p,k}^{l-1},U_{p,k}^{l}\big)$ nor $\big(U_{p,k}^{l+1},U_{p,k}^{l+2}\big)$ are in $I\times J$, where all exponents are taken modulo $p$ in $[p]$.
    Moreover, $\big(U_{p,k}^l,U_{p,k}^{l+1}\big)$ is independent of
    $$\pts_{p,k}\setminus\left\{ \big(U_{p,k}^{l-1},U_{p,k}^{l}\big) , \big(U_{p,k}^{l},U_{p,k}^{l+1}\big) , \big(U_{p,k}^{l+1},U_{p,k}^{l+2}\big) \right\} .$$
    Thus for any $l_1,\dots,l_j$, the event 
    \begin{equation*}
        \left\{\Bigg. \pts_{p,k}\cap(I\times J) = \left\{ \big(U_{p,k}^{l_1},U_{p,k}^{l_1+1}\big) ,\dots, \big(U_{p,k}^{l_j},U_{p,k}^{l_j+1}\big) \right\} \right\}
    \end{equation*}
    is either negligible or essentially rewrites as 
    \begin{equation*}
        \big(U_{p,k}^{l_1},U_{p,k}^{l_1+1}\big) \in I\!\times\! J
        \,,\dots,\, 
        \big(U_{p,k}^{l_j},U_{p,k}^{l_j+1}\big) \in I\!\times\! J
        \;\;,\;\;
        (I\!\times\! J)\cap\pts_{p,k}\setminus\left\{ \big(U_{p,k}^{l_1-1},U_{p,k}^{l_1}\big) ,\dots, \big(U_{p,k}^{l_j+1},U_{p,k}^{l_j+2}\big) \right\} = \emptyset .
    \end{equation*}
    Conditionally on this event when it is non-negligible, these points are i.i.d.~uniform in $I\times J$.
    This concludes the proof.
\end{proof}

\begin{lem}\label{lem: conditionally on coordinates}
    Let $t$ be a cycle type of size $n$ and $\pts = \left\{\big. (U_i,U_{\s(i)}) \,:\, i\in\Vert_t \right\}$ be a geometric construction.
    Then conditionally given the unordered set $\{U_i \,:\, i\in\Vert_t\}$, the following holds a.s.:
    \begin{enumerate}
        \item $\perm{\pts}$ is a uniform $t$-cyclic permutation.
        \item If $I,J$ are two essentially disjoint intervals of $[0,1]$ then, conditionally given $\card{\pts\cap(I\times J)}$, the permutation $\perm{\pts\cap(I\times J)}$ is uniformly random of size $\card{\pts\cap(I\times J)}$.
    \end{enumerate}
\end{lem}

\begin{proof}
    The conditional law of $\pts$ given $\{U_i \,:\, i\in\Vert_t\}$ can be a.s.~re-described as follows.
    Sort $\{U_i \,:\, i\in\Vert_t\}$ as $U_{(1)}<\dots<U_{(n)}$, let $\rho$ be a uniformly random bijection from $[n]$ to $\Vert_t$, then set $V_i := U_{(\rho^{-1}(i))}$ for any $i\in\Vert_t$, and finally $\pts = \left\{\big. (V_i,V_{\s(i)}) \,:\, i\in\Vert_t \right\}$.
    The first claim then follows from the same argument as in the proof of \Cref{lem: geometric construction t-cyclic}.

    For the second claim, notice that $\pts\cap(I\times J)$ is a random subset of the finite grid $\left\{\left(U_{(i)},U_{(j)}\right) : i,j\in[n]\right\}\cap(I\times J)$.
    Consider $0\le m\le n$ and indices $i_1,\dots,i_m, j_1,\dots,j_m \in [n]$ such that for each $r\in[m]$, $U_{(i_r)}\in I$ and $U_{(j_r)}\in J$.
    The sets of indices $\{i_1,\dots,i_m\}$ and $\{j_1,\dots,j_m\}$ are then a.s.~disjoint.
    Furthermore, the probability
    \begin{equation*}
        \prob{\forall r\in[m],\; \left(U_{(i_r)} , U_{(j_r)}\right)\in\pts}
        = \prob{\forall r\in[m],\; \left(V_{\rho(i_r)} , V_{\rho(j_r)}\right)\in\pts}
        = \prob{\big. \forall r\in[m],\; \rho(j_r) = \s(\rho(i_r))}
    \end{equation*}
    only depends on $i_1,\dots,i_m,j_1,\dots,j_m$ through $m$.
    This readily implies that conditionally given $\{U_i,i\in\Vert_t\}$ and $\card{\pts\cap(I\times J)}$, the permutation $\perm{\pts\cap(I\times J)}$ is uniformly random.
\end{proof}

\section{Proofs of the results on monotone subsequences}
\label{sec: proof LIS LDS}

\subsection{Decreasing subsequences}

\begin{proof}[Proof of \Cref{th: LDS t-cyclic}]
    Throughout the proof we fix $\delta>0$ and let $\pts$ be a geometric construction of $\tau_n$.
    
    \paragraph{Upper tail bound:}
    Suppose, without loss of generality, that $\delta=1/\beta$ for some integer $\beta$.
    Fix an arbitrary integer $K>16\beta^2$ and set $\Dx:=1/K$.
    Slice $\carre$ into a regular grid $\left( C_{i,j} \right)_{1\le i,j\le K}$ made of square cells
    \begin{equation*}
        C_{i,j} := [(i-1)\Dx,i\Dx]\times[(j-1)\Dx,j\Dx] .
    \end{equation*}
    Notice that any down-right path of points can only visit at most one diagonal square cell $C_{i,i}$. 
    Therefore:
    \begin{equation}\label{eq: decoupe diag LDS t-cyclic}
        \LDS(\tau_n) = \LDS(\pts) \leq \LDS\big(\widetilde\pts\big) + \max_{1\le i\le K}\LDS\left(\pts\cap C_{i,i}\right)
    \end{equation}
    where $\widetilde\pts$ is the restriction of $\pts$ to the cells $C_{i,j}$ with $i\ne j$.
    For each $i$, \Cref{lem: decomposition t-cyclic} asserts that $\pts\cap C_{i,i}$ is a superposition of three (correlated) random-sized sets of i.i.d.~uniform points in $C_{i,i}$.
    Writing $L_m$ for the $\LDS$ of a uniform permutation of size $m$, we obtain:
    \begin{equation*}
        \prob{\Big. \LDS(\pts\cap C_{i,i})>2\delta\sqrt{n}}
        \le 3\max_{1\le j\le 3} \prob{\LDS\big(\pts^{(j)}\cap C_{i,i}\big) >\frac{2\delta}{3}\sqrt{n}}
        \le 3\max_{1\le j\le 3} \prob{L_{\card{\pts^{(j)}\cap C_{i,i}}}>\frac{2\delta}{3}\sqrt{n}} .
    \end{equation*}
    Since each size $\card{\pts^{(j)}\cap C_{i,i}}$ is stochastically dominated by $\Binomial{\big.n}{K^{-2}}$ and the variables $L_m$ are stochastically nondecreasing in $m$, this yields:
    \begin{equation}\label{eq: LDS cases diago decoupe}
        \prob{\max_{1\le i\le K}\LDS(\pts\cap C_{i,i})>2\delta\sqrt{n}}
        \le 3K\prob{ L_{\Binomial{n}{K^{-2}}} >\frac{2\delta}{3} \sqrt{n}} .
    \end{equation}
    First off, Chernoff bounds imply
    \begin{equation}\label{eq: LDS cases diago Chernoff}
        \prob{\big.\Binomial{n}{K^{-2}}>2nK^{-2}}
        \le \exp\left(-c_\delta n\right)
    \end{equation}
    for some constant $c_\delta>0$ which depends only on $\delta$ and whose value may change throughout the proof.
    Then, since $2\sqrt2 K^{-1}<\sqrt2 \delta^2/8<2\delta/3$, \cite[Theorem~2]{S98} implies
    \begin{equation}\label{eq: LDS cases diago DZ}
        \prob{L_{2nK^{-2}}>\frac{2\delta}{3} \sqrt{n}}
        \le \exp\left(-c_\delta\sqrt{n}\right) 
    \end{equation}
    for some other $c_\delta>0$.
    From \Cref{eq: LDS cases diago decoupe,eq: LDS cases diago Chernoff,eq: LDS cases diago DZ} we deduce 
    \begin{align*}
        \prob{\max_{1\le i\le K}\LDS(\pts\cap C_{i,i})>2\delta\sqrt{n}}
        \le \exp\left(-c_\delta\sqrt{n}\right)
    \end{align*}
    for some constant $c_\delta>0$.
    Then by \eqref{eq: decoupe diag LDS t-cyclic}:
    \begin{equation}\label{eq: decoupe diag proba LDS t-cyclic}
        \prob{\Big. \LDS(\tau_n)>2(1+3\delta)\sqrt{n}}
        \le \prob{\big. \LDS\big(\widetilde\pts\big)>2(1+2\delta)\sqrt{n}}
        + \exp\left(-c_\delta\sqrt{n}\right) .
    \end{equation}
    To study the remaining term, we can broadly rewrite the proof of \cite[Lemma~9]{DZ95}.
    Set $\Dy := \Dx/\beta = 1/(K\beta)$ and slice $\carre$ into a thinner grid $\left( \widetilde C_{i,j} \right)_{1\le i\le K,1\le j\le K\beta}$ made of rectangular cells
    \begin{equation*}
        \widetilde C_{i,j} := [(i-1)\Dx,i\Dx]\times[(j-1)\Dy,j\Dy] .
    \end{equation*}

    \begin{figure}
        \centering
        \includegraphics[scale=.5]{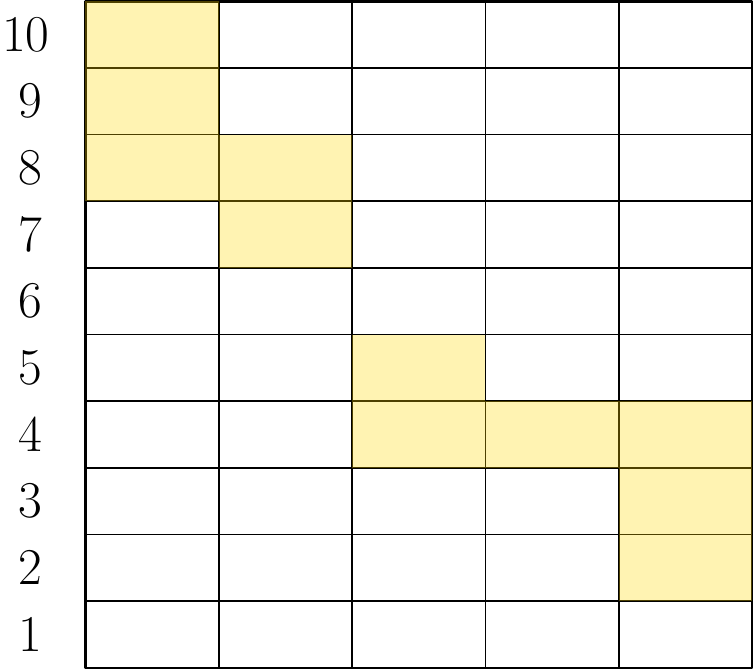}
        \caption{Representation of the down-right sequence of cells (highlighted) encoded by the admissible sequence $\big((10,8),(8,7),(5,4),(4,4),(4,2)\big)$, with $K=5$ and $\beta=2$.}
        \label{fig: DZ_grid}
    \end{figure}
    
    Say a sequence of indices $(\mathbf{a},\mathbf{b}) = (a_i,b_i)_{1\le i\le K}$ is {\em admissible} if it satisfies $1\le b_{i+1}\le a_{i+1}\le b_i\le a_i\le K\beta$ for all $i\in[K-1]$.
    Informally, it encodes a down-right sequence of cells where $a_i$, resp.~$b_i$, is the highest, resp.~lowest, cell in the $i$-th column.
    See \Cref{fig: DZ_grid} for a representation.
    Fix such $(\mathbf{a},\mathbf{b})$ and write $\widetilde\pts_{(\mathbf{a},\mathbf{b})}$ for the restriction of $\widetilde\pts$ to the cells associated with $(\mathbf{a},\mathbf{b})$. 
    For each $1\le i\le K$, write $$R_i := [(i-1)\Dx,i\Dx]\times[(b_i-1)\Dy,a_i\Dy]$$ for the $i$-th column of $(\mathbf{a},\mathbf{b})$, and set $h_i:=a_i-b_i+1$.
    Since $(\mathbf{a},\mathbf{b})$ encodes a down-right path, we can write
    \begin{equation}\label{eq: LDS along down-right path of cells}
        \LDS\big(\widetilde\pts_{(\mathbf{a},\mathbf{b})}\big)
        \le \sum_{i=1}^K \LDS(\widetilde\pts\cap R_i) .
    \end{equation}
    Furthermore, at most one column $R_i$ can intersect the diagonal square cells $C_{i,i}$.
    When this happens, write $i_0$ for the corresponding index, and arbitrarily set $i_0:=1$ otherwise.
    Now notice that, by the Cauchy--Schwarz inequality:
    \begin{align*}
        \sum_{i=1}^K \sqrt{h_i\Dx\Dy}
        \le \sqrt{K\sum_{i=1}^K h_i\Dx\Dy}
        \le \sqrt{K\Dx\Dy(K+K\beta)}
        = \sqrt{1+\delta}
    \end{align*}
    where the middle inequality is obtained by noticing that a down-right path has at most $K+K\beta$ cells.
    This yields
    \begin{equation}\label{eq: decoupe delta pour LDS}
        1+2\delta \ge 
        \delta + \sqrt{1+\delta}\sum_{i=1}^K \sqrt{h_i\Dx\Dy} .
    \end{equation}
    Using \Cref{eq: LDS along down-right path of cells,eq: decoupe delta pour LDS}, we deduce:
    \begin{multline}\label{eq: bound LDS on admissible t-cyclic}
        \prob{ \LDS\big(\widetilde\pts_{(\mathbf{a},\mathbf{b})}\big) > 2(1+2\delta)\sqrt{n}}
        \\\le \prob{\exists i\neq i_0,\;\LDS(\pts\cap R_i)>2\sqrt{1+\delta}\sqrt{n h_i\Dx\Dy}}
        + \prob{ \LDS\big(\widetilde\pts\cap R_{i_0}\big) >2\delta\sqrt{n}} .
    \end{multline}
    We can decompose the family $\widetilde\pts\cap R_{i_0} = (\pts\cap A) \sqcup (\pts\cap B)$ into its points above and below $C_{i_0,i_0}$.
    Since we have $\Leb(A)\le\Dx<\delta^2/16$, \Cref{cor: controlling number of points t-cyclic} yields:
    \begin{equation*}
        \prob{\card{\pts\cap A} > \frac{\delta^2}{16}n}
        \le \exp\left(-c_\delta n\right),
    \end{equation*}
    and likewise for $B$.
    \Cref{lem: iid outside diagonal t-cyclic} asserts that the sets $\pts\cap A$ and $\pts\cap B$ consist of i.i.d.~uniform points, hence:
    \begin{equation*}
        \prob{ \LDS\big(\widetilde\pts\cap R_{i_0}\big) >2\delta\sqrt{n}}
        \le \prob{L_{\card{\pts\cap A}}>\delta\sqrt{n}} + \prob{L_{\card{\pts\cap B}}>\delta\sqrt{n}}
        \le 2\prob{L_{\frac{\delta^2}{16}n}>\delta\sqrt{n}} + \exp(-c_\delta n) .
    \end{equation*}
    Therefore, by \cite[Theorem~2]{S98}:
    \begin{equation}\label{eq: LDS control lone band t-cyclic}
        \prob{ \LDS\big(\widetilde\pts\cap R_{i_0}\big) >2\delta\sqrt{n}}
        \le \exp\left(-c_\delta\sqrt{n}\right) .
    \end{equation}
    The first term of \eqref{eq: bound LDS on admissible t-cyclic} is handled in a similar way.
    If $i\ne i_0$, by \Cref{cor: controlling number of points t-cyclic}:
    \begin{align*}
        \prob{\Big. \card{\pts\cap R_i} > (1+\tfrac\delta2)nh_i\Dx\Dy}
        \le \exp\left(-c_\delta n\right) .
    \end{align*}
    We can then use \Cref{lem: iid outside diagonal t-cyclic} and \cite[Theorem~2]{S98} as before:
    \begin{multline}\label{eq: LDS control one path t-cyclic}
        \prob{\exists i\neq i_0,\;\LDS(\pts\cap R_i) > 2\sqrt{1+\delta}\sqrt{nh_i\Dx\Dy}}
        \\\le K\exp\left(-c_\delta n\right) + K\max_{i\neq i_0}\prob{L_{(1+\frac\delta2)nh_i\Dx\Dy} > 2\sqrt{1+\delta}\sqrt{nh_i\Dx\Dy}}
        \le \exp\left(-c_\delta \sqrt{n}\right).
    \end{multline}
    Finally, putting together \Cref{eq: bound LDS on admissible t-cyclic,eq: LDS control lone band t-cyclic,eq: LDS control one path t-cyclic}:
    \begin{equation*}
        \prob{ \LDS\big(\widetilde\pts_{(\mathbf{a},\mathbf{b})}\big) > 2(1+2\delta)\sqrt{n}} \le \exp\left(-c_\delta \sqrt{n}\right).
    \end{equation*}
    Since any down-right path of points defines an admissible sequence and the number of admissible sequences is bounded by $(K\beta)^{2K}$, we deduce:
    \begin{equation*}
        \prob{\Big. \LDS(\check\pts) > 2(1+2\delta)\sqrt{n}} 
        \le \exp\left(-c_\delta \sqrt{n}\right)
    \end{equation*}
    for some constant $c_\delta>0$ which depends only on $\delta$.
    Using \eqref{eq: decoupe diag proba LDS t-cyclic}, this concludes the proof of the upper tail bound.

    \paragraph{Lower tail bound:}
    Define $C:=[0,1/2]\times[1/2,1]$ and $C^*:=[1/2,1]\times[0,1/2]$.
    Then:
    \begin{equation*}
        \LDS(\pts)\ge \LDS(\pts\cap C)+\LDS(\pts\cap C^*).
    \end{equation*}
    We can apply \Cref{cor: controlling number of points t-cyclic} to obtain
    \begin{equation*}
        \prob{\Big. \card{\pts\cap C} < (1-\delta/2)^2 n/4 }\le \exp\left(-2c_\delta n\right)
    \end{equation*}
    for some $c_\delta>0$.
    Therefore, by \Cref{lem: iid outside diagonal t-cyclic} and \cite[Theorem~1]{DZ99}:
    \begin{equation*}
        \prob{\Big. \LDS(\pts\cap C) < (1-\delta)\sqrt{n}}
        \le \exp\left(-2c_\delta n\right) + \prob{L_{(1-\delta/2)^2 n/4} < 2(1-\delta)\sqrt{n/4}}
        \le \exp\left(-c_\delta n\right)
    \end{equation*}
    and likewise for $C^*$.
    Finally:
    \begin{equation*}
        \prob{\Big. \LDS(\tau_n) < 2(1\m \delta)\sqrt{n} }
        \le \prob{\Big. \LDS(\pts\cap C) < (1\m \delta)\sqrt{n} }
        + \prob{\Big. \LDS(\pts\cap C^*) < (1\m \delta)\sqrt{n} }
        \le \exp\left(-c_\delta n\right)
    \end{equation*}
    for some $c_\delta>0$, as desired.
    
    \bigskip
    Convergence in probability readily follows from these inequalities.
    To obtain $L^q$ convergence for all $q\ge1$ we simply need to prove boundedness in $n$ of $\mean{\left(\LDS(\tau_n)/\sqrt{n}\right)^{q'}}$ for any ${q'}\ge1$.
    This follows from
    \begin{equation*}
        \mean{\left(\LDS(\tau_n)/\sqrt{n}\right)^{q'}}
        \le (2+\delta)^{q'} + (n/\sqrt{n})^{q'}\exp\left(-c_\delta\sqrt{n}\right)
        = \Landau{\O}{n\to\infty}{1}
    \end{equation*}
    where $\delta>0$ is arbitrary.
    This concludes the proof of \Cref{th: LDS t-cyclic}.
\end{proof}

\begin{proof}[Proof of \Cref{cor: LDS conjugation-invariant}]
    We work conditionally given $t^{(n)}$.
    Then for each $n$, $\tau_n$ is a uniform $t^{(n)}$-cyclic permutation.
    Let $\pts$ be a geometric construction of $\tau_n$ and $\check\pts$ be its points outside the diagonal.
    Define $\check t^{(n)}$ as the cycle type $(0,t_2^{(n)},t_3^{(n)},\dots)$ of size $\check n := n-t_1^{(n)}$.
    Then according to \Cref{lem: geometric construction t-cyclic}, $\check\tau_n := \perm{\check\pts}$ is a uniform $\check t^{(n)}$-cyclic permutation.
    Since fixed points contribute at most $1$ to $\LDS(\tau_n)$, we can write:
    \begin{equation*}
        \LDS(\check\tau_n) \le \LDS(\tau_n) \le 1 + \LDS(\check\tau_n) .
    \end{equation*}
    Thus by \Cref{th: LDS t-cyclic}, for any $\delta>0$ there exists a universal constant $c_\delta>0$ such that for any $n$:
    \begin{equation*}
        \prob{\LDS(\tau_n) > 2(1+\delta)\sqrt{\check n}}
        \le \mean{ \prob{\left. \LDS(\check\tau_n) > 2(1+\delta)\sqrt{\check n} \quad\right| t^{(n)}} } 
        \\ \le \mean{\exp\left(-c_\delta \sqrt{\check n}\right)} ,
    \end{equation*}
    and likewise for the lower tail bound.
    If $\check n \to \infty$ in probability then the right-hand side goes to $0$ as $n\to\infty$, and this proves convergence in probability.
    Convergence in $L^q$ follows as in the proof of \Cref{th: LDS t-cyclic}, bounding $\LDS(\tau_n)$ above by $n-t_1^{(n)}+1$ and using the fact that the previous tail inequalities hold a.s.~conditionally given $t^{(n)}$.
\end{proof}

\subsection{Increasing subsequences}

The study of increasing subsequences is subtler than that of decreasing ones in our geometric construction.
Indeed, while down-right paths tend to stay away from the diagonal of $\carre$, allowing us to use \Cref{lem: iid outside diagonal t-cyclic}, up-right paths tend to stay very close to it.
Consequently, we settle for mere bounds here.

\begin{proof}[Proof of \Cref{th: LIS t-cyclic}]
    Fix $\delta>0$.
    Let $\pts$ be a geometric construction of $\tau_n$ as in \Cref{lem: geometric construction t-cyclic}.
    
    \paragraph{Lower tail bound:}
    Set $K=\lceil 2/\delta \rceil$ and let $(C_{i,j})_{1\le i,j\le K}$ be a regular square grid defined by
    \begin{equation*}
        C_{i,j} = [(i\m1)/K,i/K]\times[(j\m1)/K,j/K] .
    \end{equation*}
    Then:
    \begin{equation}\label{eq: lower bound LIS}
        \LIS(\tau_n) = \LIS(\pts) \ge \sum_{i=1}^{K-1} \LIS(\pts\cap C_{i,i\p1}) .
    \end{equation}
    Using \Cref{cor: controlling number of points t-cyclic}, there exists $c_\delta>0$ such that:
    \begin{equation*}
        \prob{\Big. \card{\pts\cap C_{i,i\p1}} < (1-\delta/2)^2 n/K^2}
        \le \exp\left({-2c_\delta n}\right) .
    \end{equation*}
    Then by \cite[Theorem~1]{DZ99} along with \Cref{lem: iid outside diagonal t-cyclic}, denoting by $L_m$ the $\LIS$ of a uniform permutation of size $m$:
    \begin{equation*}
        \prob{\big.\LIS(\pts\cap C_{i,i\p1}) < (1-\delta)\frac{2\sqrt{n}}{K-1}}
        \le \prob{L_{(1-\delta/2)^2 n/K^2} < (1-\delta)\frac{2\sqrt{n}}{K-1}} + \exp\left(-2c_\delta n\right)
        \le \exp\left({-c_\delta n}\right) 
    \end{equation*}
    since $\frac{1-\delta/2}{K}>\frac{1-\delta}{K-1}$.
    Finally, using \Cref{eq: lower bound LIS}:
    \begin{equation*}
        \prob{\Big. \LIS(\tau_n) < (1-\delta)2\sqrt{n}}
        \le (K-1)\max_{1\le i\le K-1}\prob{\big.\LIS(\pts\cap C_{i,i\p1}) < (1-\delta)\frac{2\sqrt{n}}{K-1}}
        \le \exp\left({-c_\delta n}\right)
    \end{equation*}
    for some $c_\delta>0$.
    
    \paragraph{Upper tail bound:} For this, simply use \Cref{lem: decomposition t-cyclic} and a rough union bound.
    Since
    \begin{equation*}
        \LIS(\tau_n)=\LIS(\pts)\le \LIS(\pts^{(1)})+\LIS(\pts^{(2)})+\LIS(\pts^{(3)}) ,
    \end{equation*}
    we can use \cite[Theorem~2]{S98} to deduce
    \begin{equation*}
        \prob{\LIS(\tau_n) > (1+\delta)2\sqrt{3n}}
        \le 3\prob{L_{\lceil n/3\rceil+1} > (1+\delta)2\sqrt{n/3}}
        \le \exp\left(-c_\delta\sqrt{n}\right) .
    \end{equation*}
    for some $c_\delta>0$.
    This concludes the proof.
\end{proof}

\begin{proof}[Proof of \Cref{cor: LIS conjugation-invariant}]
    The method is the same as in the proof of \Cref{cor: LDS conjugation-invariant}, using rather the bounds
    \begin{equation*}
        \max\left( t_1^{(n)} , \LIS(\check\tau_n) \right) \le \LIS(\tau_n) \le t_1^{(n)} + \LIS(\check\tau_n)
    \end{equation*}
    and discussing whether $t_1^{(n)}\ge (1-\delta)2\sqrt n$ for the upper tail bound.
\end{proof}

\section{Proofs of the results on Robinson--Schensted shapes}
\label{sec: proof RS}

In \cite{S23}, Sj\"ostrand studied the limit RS shape of so-called \textit{locally uniform permutations}.
The latter are random permutations of the form $\sigma_n = \perm{\pts_n}$ where $\pts_n = \{Z_1,\dots,Z_n\}$ are i.i.d.~random points distributed under some absolutely continuous distribution $\rho$.
In particular, he proved \cite[Theorem~10.2]{S23} that the limit
\[
    \lim_{n\to\infty} \frac1n \LDS_{r\sqrt n}\left(\sigma_n\right)
\]
exists for any $r\ge0$.
It is denoted by $F_{\max}(r)$, where the dependence in $\rho$ is implicitly understood, and is characterized by some variational problem:
\[
    F_{\max}(r) = \sup_{u\in \Usjo_{0,r}\left(\opencarre\right)} F_\rho(u) .
\]
Precise definitions are postponed to \Cref{sec: proof RS prelim}.
The approach in \cite{S23} relies on a shift in point of view:
instead of looking for large $r\sqrt n$-decreasing subsequences in $\sigma_n$, one looks for suitable functions $w_n \in \Usjo_{r\sqrt n}\left( \opencarre \right)$ that ``fit'' the point set $\pts_n$.
The quantity $F_\rho(u)$ is constructed so that it can be thought of as a ``score'' associated with $u$, that is, a measurement of how well the function $u$ should fit the decreasing subsequences of $\pts_n$ after rescaling.
\cite[Proposition~4.5]{S23} shows that the function $F_\rho$ can be used to \textit{locally} approximate $\LDS_{r\sqrt n}\left(\sigma_n\right)$;
it is also reformulated in handier terms into \cite[Lemma~8.2]{S23} which yields a global upper bound, and \cite[Lemma~6.3]{S23} which yields a local lower bound.

The difference between our setting and Sj\"ostrand's is that the points of our geometric construction $\pts$ are not globally independent.
Nonetheless, they are \textit{locally} i.i.d.~thanks to \Cref{lem: iid outside diagonal t-cyclic}, allowing us to use the same techniques as in \cite{S23}.
Moreover, the local sampling density is constant in our case ($\rho:=1$), thus greatly simplifying several steps of the proof.

This section is organized as follows.
First, in \Cref{sec: proof RS prelim}, we recall the definitions and results from \cite{S23} that will be relevant in our setting.
Then, in \Cref{sec: proof RS uniform}, we illustrate them with the case of i.i.d.~points sampled uniformly inside an arbitrary domain.
Our contribution comes in \Cref{sec: proof RS our setting}, where we use these results to establish \Cref{th: limit shape w/o fixed points}, and finally \Cref{cor: limit shape conjugation-invariant}.

\subsection{Preliminaries}\label{sec: proof RS prelim}

\begin{figure}
    \centering
    \includegraphics[scale=1]{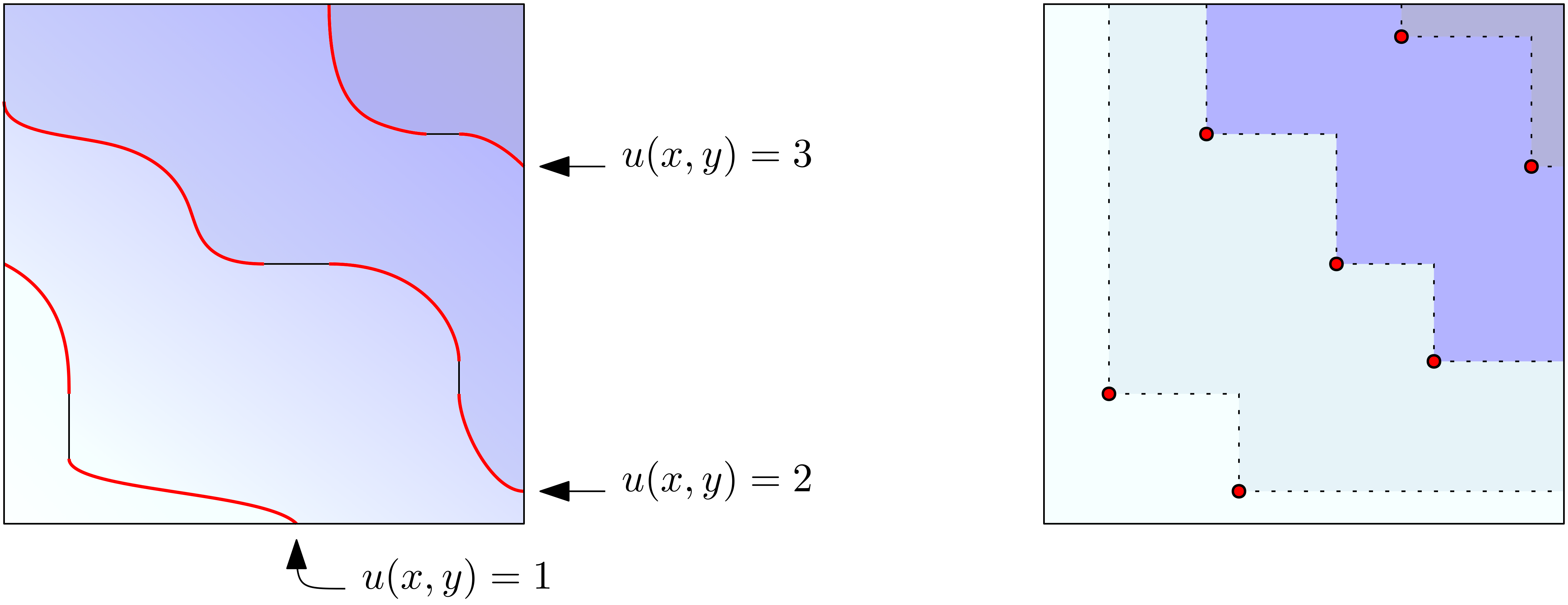}
    \caption{Left: a doubly increasing function $u\in\Usjo_{0,r}\left( \opencarre \right)$ for some $3\le r< 4$, with its integer level lines.
    Darker colors indicate greater values of $u$.
    The set $D(u)$ is highlighted in red.
    Right: in red, a $3$-decreasing point set $P$.
    The function $\kappa_P$ indeed satisfies $D(\kappa_P) = P$.
    Its values range from $0$ to $3$, with lighter colors indicating lower values and darker colors indicating higher values.}
    \label{fig: Sjo D and kappa}
\end{figure}

Let $\Omega$ be an open subset of $\R^2$.
If $u:\Omega\to\R$ is Lebesgue-measurable, write $\lVert u \rVert_{1,\Omega} := \int_\Omega \card{u} d\Leb$ for its $L^1$ norm. 
Recall that we use the partial order on $\R^2$ defined by:
\[
    (x,y)\le (x',y') \quad\text{ if and only if }\quad x\le x' \text{ and } y\le y' .
\]
Then, we call $u$ {\em doubly increasing} if $u(x,y)\le u(x',y')$ whenever $(x,y)\le(x',y')$.
Denote by $\Usjo(\Omega)$ the set of doubly increasing functions on $\Omega$, and by $\Usjo_{h,r}(\Omega)$ its subset consisting of functions with values in $[h,h+r]$.
For any $u\in\Usjo(\Omega)$, define
\begin{equation*}
    D(u) := \left\{\Big. (x,y)\in\Omega \,:\, u(x,y)\in\Z \text{ and } u(x',y')<u(x,y) \text{ for any } (x',y')\in\Omega\setminus\{(x,y)\} \text{ s.t. } (x',y')\le (x,y) \right\} .
\end{equation*}
In words, $D(u)$ is the set of ``south-west corners'' in the integer level lines of $u$.
Also, if $P$ is a finite subset of $\Omega$, define $\kappa_P : (x,y)\in \Omega \mapsto \LIS\left(\big. P \cap \left( (-\infty,x] \times (-\infty,y] \right) \right)$.
See \Cref{fig: Sjo D and kappa} for an illustration.

It is easy to check that if $u\in\Usjo_{0,r}(\Omega)$ then $D(u)$ is an $\lfloor r\rfloor$-decreasing subset of $\Omega$, see \cite[Lemma~6.2]{S23}.
Reciprocally, one can also check that if $P$ is a $k$-decreasing subset of $\Omega$ then $\kappa_P$ is in $\Usjo_{0,k}(\Omega)$ and satisfies $D(\kappa_P)=P$, see again \cite[Lemma~6.2]{S23}.
Now, if $\pts$ is a finite point set in $\Omega$, recall that $\LDS_k(\pts)$ can be expressed as the maximum of $\card{\pts \cap P}$ where $P$ ranges over $k$-decreasing subsets of $\Omega$.
Subsequently, by \cite[Lemma~6.2]{S23}:
\[
    \LDS_k(\pts) = \max_{u\in\Usjo_{0,k}(\Omega)} \card{D(u)\cap\pts} .
\]
We shall use this equation after rescaling, that is:
\begin{equation}\label{eq: link LDS Sjo D}
    \LDS_{r\sqrt n}(\pts) = \max_{w\in\Usjo_{0,r}(\Omega)} \card{D(w\sqrt n)\cap\pts} .
\end{equation}

The main tool that we will use from \cite{S23} is a function $F_\Omega : \Usjo(\Omega) \to \R_+$ that serves as a ``score'' associated with $u$: 
the quantity $F_\Omega(u)$ is a limit approximation of $\frac1n \card{D(w\sqrt n)\cap \sigma_n}$, where $\sigma_n$ is a homogeneous Poisson point process with intensity $n$ on $\Omega$ and $w$ is close to $u$.
The definition of $F_\Omega$ hinges on the following lemma:

\begin{lem}[Theorem~3.2 in \cite{S23}]
    For $\beta>0$, define
    \[
        \Omega_\beta := \left\{ (x,y)\in\R^2 \,:\, 0<x+y<\sqrt2 \text{ and } 0<y-x<\beta\sqrt2 \right\} .
    \]
    For $\gamma>0$, let $\sigma_\gamma$ be a homogeneous Poisson point process on $\R^2$ with intensity $\gamma$.
    Then for each $r\ge0$, as $\beta$ and $\gamma$ tend to infinity, the random variable $\frac{1}{\beta\gamma}\LDS_{r\sqrt\gamma}\left( \sigma_\gamma \cap \Omega_\beta \right)$ converges in $L^1$ to a constant $\phi(r)$.
\end{lem}

Then, for $\theta\ge0$, define $L(\theta) := \phi\big( \sqrt{2\theta} \big)$.
Finally, recall that each $u\in\Usjo(\Omega)$ admits partial derivatives almost everywhere, and define $F_\Omega(u) := \lVert L(\partial_x u \, \partial_y u) \rVert_{1,\Omega}$.
The following two lemmas justify the aforementioned idea that $F_\Omega(u)$ can be interpreted as a score associated with $u$:

\begin{lem}[From Lemma~8.2 in \cite{S23}]\label{lem: F serves as a global upper bound}
    Let $\Omega$ be an arbitrary open subset of $\R^2$, and let $\sigma_n$ be a homogeneous Poisson point process on $\Omega$ with intensity $n$.
    Then for any $u\in\Usjo(\Omega)$ and any $\epsilon>0$, there exists $\delta>0$ such that, with high probability as $n\to\infty$:
    \[
        \sup_{w\in\Usjo(\Omega) , \lVert w-u \rVert_{1,\Omega} < \delta} \,
        \frac1n \card{ D(w\sqrt n) \cap \sigma_n } < F_\Omega(u) + \epsilon .
    \]
\end{lem}

\begin{lem}[From Lemma~6.3 in \cite{S23}]\label{lem: F serves as a local lower bound}
    For $a,b,\beta>0$, define
    \[
        \Omega_{a,b,\beta} := \left\{ (x,y)\in\R^2 \,:\, \card{ax+by}<1 \text{ and } \card{ax-by}<\beta \right\} 
    \]
    and for $c\ge0$, define 
    \[
        u_{a,b,c} : (x,y)\in \Omega_{a,b,\beta} \mapsto c(ax+by) .
    \]
    Let $\sigma_n$ be a homogeneous Poisson point process on $\Omega_{a,b,\beta}$ with intensity $n$.
    Then with high probability as $n\to\infty$, for any $d\in\R$, there exists $w^{(n)} \in \Usjo_{d,2c}(\Omega_{a,b,\beta})$ such that:
    \[
        \frac{1}{n} \card{ D\left( w^{(n)}\sqrt n \right) \cap \sigma_n } \ge F_{\Omega_{a,b,\beta}}\left( u_{a,b,c} \right) - \frac6\beta \Leb\left( \Omega_{a,b,\beta} \right) ,
    \]
    or equivalently:
    \[
        \frac{1}{n \cdot \Leb\left( \Omega_{a,b,\beta} \right)} \card{ D\left( w^{(n)}\sqrt n \right) \cap \sigma_n } \ge L\left( ca \cdot cb \right) - \frac6\beta .
    \]
\end{lem}

Contrary to \Cref{lem: F serves as a global upper bound}, which yields a \textit{global} upper bound, \Cref{lem: F serves as a local lower bound} yields a \textit{local} lower bound.
In order to get a \textit{global} lower bound, the broad idea is to cover most of the global domain $\Omega$ by a family of adequate local parallelograms, apply \Cref{lem: F serves as a local lower bound} on each of them, and carefully ``patch'' the resulting family of local doubly increasing functions into a global one.
Let us introduce some notions from \cite{S23} to that aim.
Throughout, $\Omega$ denotes an arbitrary open subset of $\R^2$.

If $u\in\Usjo(\Omega)$ and $\iota>0$, we call $(u,\iota)$-parallelogram any parallelogram $P$ of the form
\begin{align*}
    P = \left\{ (x,y)\in\R^2 \,:\, 
    \card{ \widetilde u_x^P (x-x_P) + \widetilde u_y^P (y-y_P) } \le \iota c_P \text{ and }
    \card{ \widetilde u_x^P (x-x_P) - \widetilde u_y^P (y-y_P) } \le c_P 
    \right\}
\end{align*}
where $(x_P,y_P)$ is a differentiability point of $u$, $c_P>0$, $\widetilde u_x^P := \iota^3 \vee \partial_x u (x_P,y_P)$ and $\widetilde u_y^P := \iota^3 \vee \partial_y u (x_P,y_P)$.
We write $u_x^P$ and $u_y^P$ for the partial derivatives of $u$ at $(x_P,y_P)$, and say that $P$ is well-behaved if $u_x^P = \widetilde u_x^P$ and $u_y^P = \widetilde u_y^P$.

\begin{lem}[Lemma~7.4 in \cite{S23}]\label{lem: Sjo simple bounds inside parallelogram}
    Let $P$ be a $(u,\iota)$-parallelogram.
    Then for each $(x,y)\in P$:
    \[
        \card{x-x_P} + \card{y-y_P} \le c_P (1+\iota) \iota^{-3} .
    \]
\end{lem}

\begin{lem}[From Lemma~7.7 in \cite{S23}]\label{lem: Sjo nice covering}
    Let $u\in\Usjo(\Omega)$ and $\epsilon>0$.
    For each $0<\iota<1/2$ there exists a measurable subset $S_\iota \subseteq \Omega$ on which $u$ is differentiable, and a finite collection $\Psjo_\iota$ of disjoint $(u,\iota)$-parallelograms, such that:
    \begin{itemize}
        \item[(a)] $S_\iota \subseteq \bigcup_{P\in\Psjo_\iota} P$;
        \item[(b)] $\partial_x u$ and $\partial_y u$ are bounded on $\bigcup_{0<\iota<1/2} S_\iota$;
        \item[(c)] $\Leb( \Omega \setminus S_\iota ) < \epsilon + o_\iota(1)$;
    \end{itemize}
    and for each $P\in\Psjo_\iota$:
    \begin{itemize}
        \item[(a')] $P\subseteq\Omega$ and $(x_P,y_P)\in S_\iota$;
        \item[(f)] for all $(x,y)\in P$:
        \[
            \card{u(x,y) - \left( u(x_P,y_P) + (x-x_P)u_x^P + (y-y_P)u_y^P \right)} \le \iota^5\left( \card{x-x_P} + \card{y-y_P} \right) ;
        \]
        \item[(g)] $\card{ \frac{1}{\Leb(P)} 
        F_P(u)
        - L(u_x^P u_y^P) } < \iota$.
    \end{itemize}
\end{lem}

In words, the set $S_\iota$ covers most of the domain $\Omega$, and on each parallelogram, the function $u$ is almost linear and the function $L(\partial_x u \, \partial_y u)$ is almost constant.
This will allow use to use \Cref{lem: F serves as a local lower bound} on each parallelogram.
Then, patching the resulting local doubly increasing functions into a global one will mainly rely on item (f) above and on the following lemma:

\begin{lem}[From Lemma~9.1 in \cite{S23}]\label{lem: Sjo patch doubly increasing}
    Let $\mathcal{S}$ be an open subset of $\Omega$, and let $u\in\Usjo(\mathcal{S})$.
    Then, there exists $w\in\Usjo(\Omega)$ such that $w\vert_{\mathcal{S}} = u$ and $\overline{u(\mathcal{S})} = \overline{w(\Omega)}$.
\end{lem}

Finally, let us state a lemma about parallelograms that we could not find explicitly stated in \cite{S23} but that we will need nonetheless:

\begin{lem}\label{lem: boundaries of parallelogram}
    Let $u\in\Usjo(\Omega)$, $0<\iota<1/2$, and $P$ be a $(u,\iota)$-parallelogram.
    Let $P'$ denote the open parallelogram obtained by shrinking $P$ by a factor $1\m2\iota$ in width and height, that is:
    \begin{align*}
        P' := \Bigg\{ (x,y)\in\R^2 \,:\, 
        &\card{ \frac{\widetilde u_x^P}{(1-2\iota)\iota c_P} (x-x_P) + \frac{\widetilde u_y^P}{(1-2\iota)\iota c_P} (y-y_P) } < 1 
        \\&\text{ and } \card{ \frac{\widetilde u_x^P}{(1-2\iota)\iota c_P} (x-x_P) - \frac{\widetilde u_y^P}{(1-2\iota)\iota c_P} (y-y_P) } < \frac{1}{\iota} \Bigg\} .
    \end{align*}
    Let $(x_1,y_1) \le (x_2,y_2)$ such that $(x_1,y_1)$ is in $P'$ and $(x_2,y_2)$ is on the boundary of $P$.
    Then necessarily $(x_2,y_2)$ lies on the northeast boundary of $P$, that is:
    \[
        \widetilde u_x^P (x_2-x_P) + \widetilde u_y^P (y_2-y_P) = \iota c_P .
    \]
\end{lem}

\begin{proof}
    Obviously, $(x_2,y_2)$ cannot lie on the southwest boundary of $P$.
    Assume that it lies on the northwest boundary of $P$, that is:
    \[
        \widetilde u_x^P (x_2-x_P) - \widetilde u_y^P (y_2-y_P) = -c_P .
    \]
    Since $(x_1,y_1)\in P'$, we have that:
    \[
        \widetilde u_x^P (x_1-x_P) - \widetilde u_y^P (y_1-y_P) \ge -(1-2\iota)c_P ,
    \]
    whence:
    \[
        \widetilde u_y^P (y_2-y_P) 
        = \widetilde u_x^P (x_2-x_P) + c_P
        \ge \widetilde u_x^P (x_1-x_P) + c_P
        \ge \widetilde u_y^P (y_1-y_P) + 2\iota c_P .
    \]
    Therefore:
    \[
        \widetilde u_x^P (x_2-x_P) + \widetilde u_y^P (y_2-y_P)
        \ge \widetilde u_x^P (x_1-x_P) + \widetilde u_y^P (y_1-y_P) + 2\iota c_P
        \ge -(1-2\iota)\iota c_P + 2\iota c_P
        = \iota c_P + 2\iota^2 c_P,
    \]
    contradicting the definition of $P$.
    The computation works the same if we assume that $(x_2,y_2)$ lies on the southeast boundary of $P$, and this concludes the proof.
\end{proof}

\subsection{The uniform case}\label{sec: proof RS uniform}

Now, let us illustrate the results recalled in the previous section with the case study of i.i.d.~uniform points inside an arbitrary domain.
Note that this is a much simpler framework than the one considered in \cite{S23} (since here we have $\rho=1$), and that \Cref{th: limit shape uniform PPP on arbitrary domain} below is wholly contained in \cite[Theorem~10.2]{S23}.
Nonetheless, stating it and recalling its proof will be useful for our proof of \Cref{th: limit shape w/o fixed points}.

\begin{theorem}[From Theorem~10.2 in \cite{S23}]\label{th: limit shape uniform PPP on arbitrary domain}
    Let $\Omega$ be an arbitrary open subset of $\R^2$, and let $\sigma_n$ be a homogeneous Poisson point process on $\Omega$ with intensity $n$.
    Then for each $r\ge0$, the following convergence holds in probability as $n\to\infty$:
    \[
        \frac1n \LDS_{r\sqrt n}(\sigma_n) 
        \,\longrightarrow\,
        \sup_{u\in\Usjo_{0,r}(\Omega)} F_{\Omega}(u) .
    \]
\end{theorem}

\begin{proof}
    Fix $r\ge0$.
    Thanks to \Cref{eq: link LDS Sjo D}, our goal is to bound $\card{D(w\sqrt n)\cap\sigma_n}$ above and below, uniformly over $w\in\Usjo_{0,r}(\Omega)$.
    Fix $\epsilon>0$.
    As stated in the previous section, the upper bound uses \Cref{lem: F serves as a global upper bound} whereas the lower bound uses \Cref{lem: F serves as a local lower bound} and requires more work.

    \paragraph{Upper bound:}
    For each $u\in\Usjo_{0,r}(\Omega)$, use \Cref{lem: F serves as a global upper bound} to find $\delta_u>0$ such that:
    \begin{equation}\label{eq: Sjo uniform - upper bound in L1 neighbourhood}
        \sup_{w\in\Usjo_{0,r}(\Omega) , \lVert w-u \rVert_{1,\Omega} < \delta_u} \,
        \frac1n \card{ D(w\sqrt n) \cap \sigma_n } < F_\Omega(u) + \epsilon 
    \end{equation}
    holds w.h.p.~as $n\to\infty$.
    Write $\mathcal{B}_u := \left\{ w\in\Usjo_{0,r}(\Omega) : \lVert w-u \rVert_{1,\Omega} < \delta_u \right\}$.
    Then let us use \cite[Proposition~9.2]{S23}, which states that $\Usjo_{0,r}\left(\Omega\right)$ is compact for the $L^1$ metric.
    Since the open subsets $\mathcal{B}_u$ cover $\Usjo_{0,r}\left(\Omega\right)$ when $u$ ranges over $\Usjo_{0,r}\left(\Omega\right)$, there is a finite subcover 
    \[
        \Usjo_{0,r}\left(\Omega\right) \subseteq \mathcal{B}_{u_1} \cup \dots \cup \mathcal{B}_{u_k} .
    \]
    By applying \eqref{eq: Sjo uniform - upper bound in L1 neighbourhood} to $u_1, \dots, u_k$, we get that
    \begin{equation}\label{eq: Sjo uniform - final upper bound}
        \sup_{w\in\Usjo_{0,r}(\Omega)} \,
        \frac1n \card{ D(w\sqrt n) \cap \sigma_n } < \sup_{1\le j\le k}F_\Omega(u_j) + \epsilon 
    \end{equation}
    holds w.h.p.~as $n\to\infty$.
    
    \paragraph{Lower bound:}
    This bound is broadly contained in the proof of \cite[Lemma~10.1]{S23}.
    Fix $u\in\Usjo_{0,r}(\Omega)$.
    We wish to find $w_n\in \Usjo_{0,r}\left(\Omega\right)$ such that $\frac1n \card{D(w_n\sqrt n)\cap\sigma_n} \ge F_{\Omega}(u) - \epsilon$ holds w.h.p.~as $n\to\infty$.

    For each $0<\iota<1/2$, let $S_\iota$ and $\Psjo_\iota$ be given by \Cref{lem: Sjo nice covering}.
    Denote by $\widetilde \Psjo_\iota$ the set of well-behaved parallelograms in $\Psjo_\iota$.
    For each $P\in\Psjo_\iota$, let $\bar u_P \in \Usjo(P)$ be defined by
    \begin{equation*}
        \text{for all }(x,y)\in P,\quad
        \bar u_P(x,y) := u(x_P,y_P) + (x-x_P)u_x^P + (y-y_P)u_y^P 
    \end{equation*}
    if $P\in\widetilde \Psjo_\iota$, and by $\bar u_P := u\vert_P$ if $P\in \Psjo_\iota \setminus \widetilde\Psjo_\iota$.
    For each $P\in\Psjo_\iota$, let $P'$ denote the open parallelogram obtained by shrinking $P$ by a factor $1\m2\iota$ in width and height, as in \Cref{lem: boundaries of parallelogram}.
    Now, For each $P\in \widetilde\Psjo_\iota$, let us apply \Cref{lem: F serves as a local lower bound} to the linear function $\bar u_P$ on the domain $P'$, with corresponding values $a := \frac{u_x^P}{(1-2\iota)\iota c_P}$, $b := \frac{u_y^P}{(1-2\iota)\iota c_P}$, $c := (1-2\iota)\iota c_P$, and $\beta := 1/\iota$.
    We get that, w.h.p.~as $n\to\infty$, there exists
    $w_P^{(n)} \in \Usjo_{u(x_P,y_P) - (1-2\iota)\iota c_P , 2 (1-2\iota) \iota c_P}(P')$ such that
    \begin{equation}\label{eq: Sjo uniform - local control}
        \frac{1}{n \cdot \Leb(P')} \card{D\left(w_P^{(n)}\sqrt n\right) \cap \sigma_n \cap P'} \ge L(u_x^P u_y^P) - o_\iota(1)
    \end{equation}
    where $o_\iota(1) \to 0$ as $\iota\to0$, uniformly in $P$ and $n$.
    By \cite[Lemma~4.9]{S23} and item (b) of \Cref{lem: Sjo nice covering}, the function $L$ is increasing and uniformly continuous on the bounded set 
    \[
        \left\{ \partial_x u(x,y) \cdot \partial_y u(x,y) \,:\, (x,y)\in {\bigcup}_{0<\iota<1/2} S_\iota \right\}
    \]
    and satisfies $L(0)=0$.
    By definition, for any $P \in \Psjo_\iota \setminus \widetilde\Psjo_\iota$ we have $\min\left( u_x^P , u_y^P \right) < \iota^3$, and thus $L(u_x^P u_y^P) \le L( M \iota^3) = o_\iota(1)$ where $M>0$ is the absolute bound given by item (b) of \Cref{lem: Sjo nice covering}.
    Therefore, for each $P\in \Psjo_\iota \setminus \widetilde\Psjo_\iota$ and any $w_P^{(n)} \in \Usjo(P')$, \Cref{eq: Sjo uniform - local control} trivially holds (with an adequate $o_\iota(1)$).
    
    For fixed $\iota$, under the event that $w_P^{(n)}$ exists for all $P\in \widetilde\Psjo_\iota$, define a function $w_n$ on $\bigcup_{P\in \widetilde\Psjo_\iota} P' \,\cup\, \left( \Omega \setminus \bigcup_{P\in \widetilde\Psjo_\iota} \mathring P \right)$ as follows:
    $w_n$ coincides with $w_P^{(n)}$ on $P'$ for each $P\in\widetilde \Psjo_\iota$, whereas is coincides with $u$ on $\Omega \setminus \bigcup_{P\in \widetilde\Psjo_\iota} \mathring P$ (we use the notation $\mathring P$ to denote the interior of $P$).
    Let us check that $w_n$ is doubly increasing on this domain.
    First consider $(x_1, y_1) \le (x_2, y_2)$ such that $(x_1,y_1) \in P'$ for some $P\in \widetilde\Psjo_\iota$ and $(x_2,y_2)$ is on the boundary of $P$.
    According to \Cref{lem: boundaries of parallelogram}, we have:
    \begin{equation}\label{eq: Sjo uniform - patching 1}
        \bar u_P(x_2,y_2) - u(x_P,y_P) = (x_2-x_P)u_x^P + (y_2-y_p)u_y^P 
        = \iota c_P .
    \end{equation}
    Since $w_P^{(n)} \in \Usjo_{u(x_P,y_P) - (1-2\iota)\iota c_P , 2 (1-2\iota) \iota c_P}(P')$ and $(x_1,y_1)\in P'$, we have:
    \begin{equation}\label{eq: Sjo uniform - patching 2}
        \card{ w_P^{(n)}(x_1,y_1) - u(x_P,y_P) } \le (1-2\iota) \iota c_P .
    \end{equation}
    Also, using item (f) of \Cref{lem: Sjo nice covering} and \Cref{lem: boundaries of parallelogram}:
    \begin{equation}\label{eq: Sjo uniform - patching 3}
        \card{ w_P^{(n)}(x_2,y_2) - \bar u_P(x_2,y_2) }
        = \card{ u(x_2,y_2) - \bar u_P(x_2,y_2) }
        \le \iota^5\left( \card{x-x_P} + \card{y-y_P} \right)
        \le c_P (1+\iota) \iota^3 
        \le 2 c_P \iota^2 .
    \end{equation}
    Finally, putting \Cref{eq: Sjo uniform - patching 1,eq: Sjo uniform - patching 2,eq: Sjo uniform - patching 3} together:
    \begin{multline*}
        w_P^{(n)}(x_2,y_2) - w_P^{(n)}(x_1,y_1)
        \ge \bar u_P(x_2,y_2) - u(x_P,y_P)
        - \card{ w_P^{(n)}(x_1,y_1) - u(x_P,y_P) }
        - \card{ w_P^{(n)}(x_2,y_2) - \bar u_P(x_2,y_2) }
        \\\ge \iota c_P - (1-2\iota) \iota c_P - 2 c_P \iota^2
        = 0 ,
    \end{multline*}
    that is, $w_n(x_2,y_2) \ge w_n(x_1,y_1)$.
    If $(x_1, y_1) \le (x_2, y_2)$ are such that $(x_2,y_2) \in P'$ for some $P\in \widetilde\Psjo_\iota$ and $(x_1,y_1)$ is on the boundary of $P$, it can be proven in the same way that $w_n(x_2,y_2) \ge w_n(x_1,y_1)$.
    Since $w_n$ is doubly increasing on each $P'$ for $P\in\widetilde\Psjo_\iota$ and also on $\Omega \setminus \bigcup_{P\in \widetilde\Psjo_\iota} \mathring P$, and since the parallelograms forming the family $\widetilde\Psjo_\iota$ are disjoint, this suffices to prove that $w_n$ is doubly increasing on its domain of definition.

    Using \Cref{lem: Sjo patch doubly increasing}, the function $w_n$ can be extended from $\bigcup_{P\in \widetilde\Psjo_\iota} P' \,\cup\, \left( \Omega \setminus \bigcup_{P\in \widetilde\Psjo_\iota} P \right)$ to $\Omega$, and we still write $w_n \in \Usjo(\Omega)$ for this function.
    Since $u\in\Usjo_{0,r}(\Omega)$, by construction we also have that $w_n\in\Usjo_{0,r}(\Omega)$.
    Using \eqref{eq: Sjo uniform - local control}, we can write w.h.p.~as $n\to\infty$:
    \begin{align}\label{eq: Sjo uniform - minoration 1}
        \frac1n \card{D(w_n\sqrt n) \cap \sigma_n}
        \ge \sum_{P\in\Psjo_\iota} \frac1n \card{D(w_n\sqrt n) \cap \sigma_n \cap P'}
        \ge \sum_{P\in\Psjo_\iota} \Leb(P') \left( L(u_x^P u_y^P) - o_\iota(1) \right) 
    \end{align}
    where $o_\iota(1)$ is independent of $P$ and $n$. 
    By item (g) of \Cref{lem: Sjo nice covering}:
    \begin{align}\label{eq: Sjo uniform - minoration 2}
        \sum_{P\in\Psjo_\iota} \Leb(P') L(u_x^P u_y^P)
        = (1\m2\iota)^2 \sum_{P\in\Psjo_\iota} \Leb(P) L(u_x^P u_y^P)
        \ge (1\m2\iota)^2 
        F_{ \bigcup_{\Psjo_\iota}\!\!P }(u) - o_\iota(1) .
    \end{align}
    Let us also use \cite[Lemma~4.9]{S23}, which states that $0\le L(\partial_x u \, \partial_y u)\le 1$.
    Along with items (a) and (c) of \Cref{lem: Sjo nice covering}, we get:
    \begin{equation*}
        \Leb\left(\big. \bigcup_{P\in\Psjo_\iota} P \right) 
        \ge \Leb\left( \Omega \right) - \epsilon - o_\iota(1) .
    \end{equation*}
    This implies $F_{ \bigcup_{\Psjo_\iota}\!\!P }(u) \ge F_{\Omega}(u) - \epsilon - o_\iota(1)$.
    Along with \eqref{eq: Sjo uniform - minoration 1} and \eqref{eq: Sjo uniform - minoration 2}, for small enough $\iota$ this yields:
    \begin{align*}
        \frac1n \card{D(w_n\sqrt n) \cap \sigma_n}
        \ge F_{\Omega}(u) - 2\epsilon
    \end{align*}
    w.h.p.~as $n\to\infty$.
    Since this holds for abitrary $u\in\Usjo_{0,r}(\Omega)$, we get that:
    \begin{align}\label{eq: Sjo uniform - final lower bound}
        \sup_{w\in\Usjo_{0,r}(\Omega)} \frac1n \card{D(w\sqrt n) \cap \sigma_n}
        \ge \sup_{u\in\Usjo_{0,r}(\Omega)} F_{\Omega}(u) - 3\epsilon
    \end{align}
    holds w.h.p.~as $n\to\infty$.

    \smallskip

    Putting \Cref{eq: Sjo uniform - final upper bound,eq: Sjo uniform - final lower bound} together, we deduce that:
    \begin{align*}
        \sup_{w\in\Usjo_{0,r}(\Omega)} \frac1n \card{D(w\sqrt n) \cap \sigma_n}
        \,\longrightarrow\,
        \sup_{u\in\Usjo_{0,r}(\Omega)} F_{\Omega}(u)
    \end{align*}
    in probability as $n\to\infty$.
    Thanks to \eqref{eq: link LDS Sjo D}, this concludes the proof.
\end{proof}

\Cref{th: LSKV uniform} can in particular be applied to the domain $\Omega = \opencarre$.
Using \Cref{th: limit shape uniform PPP on arbitrary domain}, this yields the formula:
\begin{equation}\label{eq: F_LSKV via Sjo}
    F_\LSKV(r) = \sup_{u\in \Usjo_{0,r}\left(\opencarre\right)} F_{\opencarre}(u) .
\end{equation}

\subsection{Application to our setting}\label{sec: proof RS our setting}

We are now ready to prove \Cref{th: limit shape w/o fixed points}.
Thanks to \Cref{lem: iid outside diagonal t-cyclic}, the proof is very close to that of \Cref{th: limit shape uniform PPP on arbitrary domain};
the main difference is that we need to work individually on a family of rectangular domains which do not intersect the diagonal of $\opencarre$, and we combine them together with the same techniques as in the proof of \Cref{th: limit shape uniform PPP on arbitrary domain}.

\begin{proof}[Proof of \Cref{th: limit shape w/o fixed points}]
    Fix $r\ge0$.
    Let $\pts$ be a geometric construction of $\tau_n$.
    As in the proof of \Cref{th: limit shape uniform PPP on arbitrary domain}, we wish to bound $\card{D(w\sqrt n)\cap\pts}$ above and below, uniformly over $w\in\Usjo_{0,r}(\opencarre)$.
    Consider two essentially disjoint intervals $I,J$ of $[0,1]$ and write $\bigtriangleup$ for the symmetric difference of sets.
    Thanks to \Cref{cor: controlling number of points t-cyclic} and \Cref{lem: iid outside diagonal t-cyclic}, one can construct a homogeneous Poisson point process with intensity $n$ on $I\times J$, say $\sigma_n$, such that ${\card{\sigma_n\bigtriangleup\pts}}/{n} \to 0$ as $n\to\infty$ almost surely.

    Fix a large integer $K$, set $\epsilon:=1/K^2$, and define bands
    \begin{figure}
        \centering
        \includegraphics[scale=.5]{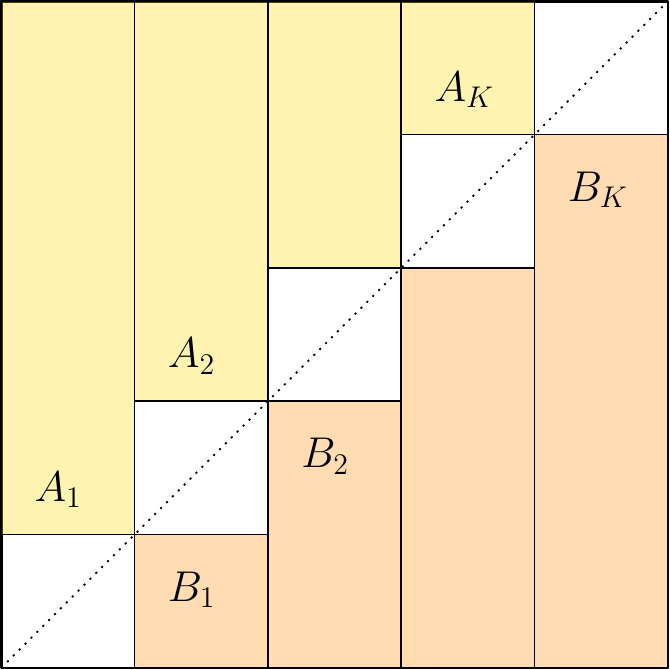}
        \caption{The bands defined in the proof of \Cref{th: limit shape w/o fixed points}. 
        On each $A_i$ and each $B_i$, $\pts$ is almost distributed like a homogeneous Poisson point process.}
        \label{fig: Sjo grid}
    \end{figure}
    \begin{equation*}
        A_i := \left( \frac{i-1}{K+1} , \frac{i}{K+1} \right) \times \left( \frac{i}{K+1} , 1 \right)
        \quad\text{ and }\quad
        B_i := \left( \frac{i}{K+1} , \frac{i+1}{K+1} \right) \times \left( 0 , \frac{i}{K+1} \right)
    \end{equation*}
    for $i\in[K]$, see \Cref{fig: Sjo grid}.
    These bands each fit in the setting of \Cref{lem: iid outside diagonal t-cyclic}, and they satisfy $\Leb(C) = \frac{1}{K+1}$ where $C := \opencarre \setminus \left(\bigcup_i \bar A_i\cup \bar B_i\right)$.
    Moreover, using \Cref{cor: controlling number of points t-cyclic} and \Cref{lem: iid outside diagonal t-cyclic} as above, one can construct a point process $\sigma_n$ on $\opencarre$ such that:
    \begin{enumerate}
        \item for each $i\in[K]$, $\sigma_n\cap A_i$ and $\sigma_n\cap B_i$ are both distributed like homogeneous Poisson point processes with intensity $n$;
        \item $\sigma_n$ coincides with $\pts$ on $C$;
        \item $\card{\sigma_n\bigtriangleup\pts}/{n} \to0$ as $n\to\infty$, almost surely.
    \end{enumerate}
    From here, the arguments for the lower and upper bounds differ.
    
    \paragraph{Upper bound:}
    For each $i\in[K]$, we can apply \Cref{lem: F serves as a global upper bound} to $\sigma_n$ on $A_i$.
    Hence for each $u\in\Usjo(A_i)$, there exists $\delta_{u,A_i}>0$ such that:
    \begin{equation*}
        \sup_{w\in\Usjo(A_i),\, \lVert w-u\rVert_{1,A_i} <\delta_{u,A_i}} \frac1n \card{D(w\sqrt n)\cap\sigma_n\cap A_i} < F_{A_i}(u) + \epsilon
    \end{equation*}
    holds w.h.p.~as $n\to\infty$.
    The same holds with $B_i$ instead of $A_i$.
    Moreover, a.s.~for any $n$:
    \begin{equation*}
        \sup_{w\in\Usjo(C)} \frac1n \card{D(w\sqrt n)\cap\sigma_n\cap C} < \frac1n \card{\sigma_n\cap C} = \frac1n \card{\pts\cap C} .
    \end{equation*}
    Moreover:
    \begin{equation*}
        F_{\opencarre}(u) = F_C(u) + \sum_{i=1}^K (F_{A_i}(u) + F_{B_i}(u)) \ge \sum_{i=1}^K \left(\big. F_{A_i}(u) + F_{B_i}(u) \right) .
    \end{equation*}
    By summing the previous inequalities we deduce that for each $u\in\Usjo\left(\opencarre\right)$, with $\delta_u := \min_{1\le i\le K} \delta_{u,A_i}\wedge\delta_{u,B_i}$:
    \begin{equation*}
        \sup_{w\in\Usjo\left(\opencarre\right),\, \lVert w-u\rVert_{1,\opencarre} <\delta_u} \frac1n \card{D(w\sqrt n)\cap\sigma_n} < F_{\opencarre}(u) + 2K\epsilon + \frac1n \card{\pts\cap C}
    \end{equation*}
    w.h.p.~as $n\to\infty$.
    Therefore:
    \begin{equation*}
        \sup_{w\in\Usjo\left(\opencarre\right),\, \lVert w-u\rVert_{1,\opencarre} <\delta_u} \frac1n \card{D(w\sqrt n)\cap\pts} < F_{\opencarre}(u) + 2K\epsilon + \frac1n \card{\pts\cap C} + \frac1n \card{\sigma_n\bigtriangleup\pts}
    \end{equation*}
    w.h.p.~as $n\to\infty$. 
    Recall that $\epsilon = 1/K^2$ and that $\frac1n \card{\sigma_n\bigtriangleup\pts} \to 0$ a.s.~as $n\to\infty$.
    Moreover, \Cref{cor: controlling number of points t-cyclic} implies that $\frac1n \card{\pts\cap C} \to \Leb(C) = 1/(K+1)$ a.s.~as $n\to\infty$. 
    Hence for each $u\in\Usjo\left(\opencarre\right)$:
    \begin{equation*}\label{eq: open local upper bound Sjo}
        \sup_{w\in\Usjo\left(\opencarre\right),\, \lVert w-u\rVert_{1,\opencarre} <\delta_u} \frac1n \card{D(w\sqrt n)\cap\pts} < F_{\opencarre}(u) + 4/K
    \end{equation*}
    w.h.p.~as $n\to\infty$.
    Using the fact that $\Usjo_{0,r}\left(\opencarre\right)$ is compact for the $L^1$ metric, as in the proof of \Cref{th: limit shape uniform PPP on arbitrary domain}, we get:
    \begin{equation}\label{eq: Sjo final upper bound}
        \sup_{w\in\Usjo_{0,r}(\opencarre)} \frac1n \card{D(w\sqrt n)\cap\pts} < \sup_{u\in \Usjo_{0,r}\left(\opencarre\right)} F_{\opencarre}(u) + 4/K 
    \end{equation}
    w.h.p.~as $n\to\infty$.

    \paragraph{Lower bound:}
    This bound is almost the same as the upper bound of \Cref{th: limit shape uniform PPP on arbitrary domain}, \textit{mutatis mutandis}.
    Fix an arbitrary $u\in \Usjo_{0,r}\left(\opencarre\right)$;
    we wish to exhibit a well-crafted $w_n\in \Usjo_{0,r}\left(\opencarre\right)$ such that $\frac1n \card{D(w_n\sqrt n)\cap\pts} \ge F_{\opencarre}(u) - \epsilon$ holds w.h.p.~as $n\to\infty$.

    First apply \Cref{lem: Sjo nice covering} to each domain $A_i\cup B_i$ for $i\in[K]$.
    For each $0<\iota<1/2$, let $S_\iota^i$ and $\Psjo_\iota^i$ be the resulting sets.
    Then define $S_\iota := \bigcup_{i\in[K]} S_\iota^i$ and $\Psjo_\iota := \bigcup_{i\in[K]} \Psjo_\iota^i$.
    Denote by $\widetilde \Psjo_\iota$ the set of well-behaved parallelograms in $\Psjo_\iota$.
    For each $P\in\Psjo_\iota$, let $P'$ denote the open parallelogram obtained by shrinking $P$ by a factor $1\m2\iota$ in width and height, as in \Cref{lem: boundaries of parallelogram}.

    Note that any parallelogram contained in $A_i\cup B_i$ is either wholly contained in $A_i$, or wholly contained in $B_i$.
    Since $\sigma_n$ is distributed like a homogeneous Poisson point process with intensity $n$ on each $A_i$ and each $B_i$, we can use \Cref{lem: F serves as a local lower bound}:
    for each $P\in \widetilde\Psjo_\iota$, w.h.p.~as $n\to\infty$, there exists
    $w_P^{(n)} \in \Usjo_{u(x_P,y_P) - (1-2\iota)\iota c_P , 2 (1-2\iota) \iota c_P}(P')$ such that
    \begin{equation}\label{eq: Sjo local control}
        \frac{1}{n \cdot \Leb(P')} \card{D\left(w_P^{(n)}\sqrt n\right) \cap \sigma_n \cap P'} \ge L(u_x^P u_y^P) - o_\iota(1)
    \end{equation}
    where $o_\iota(1) \to 0$ as $\iota\to0$, uniformly in $P$ and $n$.
    This also holds for any $P\in \Psjo_\iota \setminus \widetilde\Psjo_\iota$ and any $w_P^{(n)} \in \Usjo(P')$, for the same reason as in the proof of \Cref{th: limit shape uniform PPP on arbitrary domain}.
    
    Under the event that $w_P^{(n)}$ exists for all $P\in \widetilde\Psjo_\iota$, define a function $w_n$ on $\bigcup_{P\in \widetilde\Psjo_\iota} P' \,\cup\, \left( \opencarre \setminus \bigcup_{P\in \widetilde\Psjo_\iota} \mathring P \right)$ as follows:
    $w_n$ coincides with $w_P^{(n)}$ on $P'$ for each $P\in\widetilde \Psjo_\iota$, whereas is coincides with $u$ on $\opencarre \setminus \bigcup_{P\in \widetilde\Psjo_\iota} \mathring P$.
    As in the proof of \Cref{th: limit shape uniform PPP on arbitrary domain}, one can check that $w_n$ is doubly increasing.
    Then by \Cref{lem: Sjo patch doubly increasing}, this function can be extended to $w_n \in \Usjo_{0,r}\left(\opencarre\right)$.
    Therefore, using \eqref{eq: Sjo local control}, we can write w.h.p.~as $n\to\infty$:
    \begin{align}\label{eq: Sjo minoration 1}
        \frac1n \card{D(w_n\sqrt n) \cap \sigma_n}
        \ge \sum_{P\in\Psjo_\iota} \frac1n \card{D(w_n\sqrt n) \cap \sigma_n \cap P'}
        \ge \sum_{P\in\Psjo_\iota} \Leb(P') \left( L(u_x^P u_y^P) - o_\iota(1) \right) 
    \end{align}
    where $o_\iota(1)$ is independent of $P$ and $n$. 
    By item (g) of \Cref{lem: Sjo nice covering}:
    \begin{align}\label{eq: Sjo minoration 2}
        \sum_{P\in\Psjo_\iota} \Leb(P') L(u_x^P u_y^P)
        = (1\m2\iota)^2 \sum_{P\in\Psjo_\iota} \Leb(P) L(u_x^P u_y^P)
        \ge (1\m2\iota)^2 
        F_{ \bigcup_{\Psjo_\iota}\!\!P }(u) - o_\iota(1) .
    \end{align}
    Let us also use \cite[Lemma~4.9]{S23}, which states that $0\le L(\partial_x u \, \partial_y u)\le 1$.
    Along with items (a) and (c) of \Cref{lem: Sjo nice covering}, we get:
    \begin{equation*}
        \Leb\left(\big. \bigcup_{\Psjo_\iota} P \right) 
        \ge \Leb\left( \bigcup_{i\in[K]} A_i \cup B_i \right) - K\epsilon - o_\iota(1)
        = 1 - 2/K - o_\iota(1) .
    \end{equation*}
    This implies $F_{ \bigcup_{\Psjo_\iota}\!\!P }(u) \ge F_{\opencarre}(u) - 2/K - o_\iota(1)$.
    Along with \eqref{eq: Sjo minoration 1} and \eqref{eq: Sjo minoration 2}, this yields:
    \begin{align*}
        \frac1n \card{D(w_n\sqrt n) \cap \sigma_n}
        \ge F_{\opencarre}(u) - 2/K - o_\iota(1) 
    \end{align*}
    w.h.p.~as $n\to\infty$, where $o_\iota(1)$ is uniform in $n$.
    Take $\iota$ such that this $o_\iota(1)$ is less than $1/K$.
    Using the facts that this holds for any $u\in\Usjo_{0,r}\left(\opencarre\right)$ and that $\card{\sigma_n\bigtriangleup\pts}/n \to0$, we deduce that:
    \begin{equation}\label{eq: Sjo final lower bound}
        \sup_{w\in\Usjo_{0,r}(\opencarre)} \frac1n \card{D(w\sqrt n) \cap \pts}
        \ge \sup_{u\in\Usjo_{0,r}(\opencarre)} F_{\opencarre}(u) - 5/K
    \end{equation}
    w.h.p.~as $n\to\infty$.

    \smallskip

    Putting \Cref{eq: Sjo final upper bound,eq: Sjo final lower bound} together, along with \eqref{eq: F_LSKV via Sjo}, we get that:
    \begin{align*}
        \sup_{w\in\Usjo_{0,r}(\opencarre)} \frac1n \card{D(w\sqrt n) \cap \pts}
        \,\longrightarrow\,
        \sup_{u\in\Usjo_{0,r}(\opencarre)} F_{\opencarre}(u)
        = F_\LSKV(r)
    \end{align*}
    in probability as $n\to\infty$.
    Thanks to \eqref{eq: link LDS Sjo D}, this concludes the proof.
\end{proof}


\begin{rem}
    With the same arguments as in the proof of \cite[Lemma~10.1]{S23}, we could refine our proof of \Cref{th: limit shape w/o fixed points} to show that the longest decreasing subsequences of $\pts$ are located on the same limit curves as in the uniform case (see \cite[Theorem~10.2 (a)]{S23} for a precise statement).
\end{rem}

\begin{proof}[Proof of \Cref{cor: limit shape conjugation-invariant}]
    Let $\pts$ be a geometric construction of $\tau_n$.
    Since each decreasing subsequence of $\tau_n$ contains at most one fixed point, we can write:
    \begin{equation*}
        \LDS_{k}(\check\tau_n)
        \le \LDS_{k}(\tau_n)
        \le \LDS_{k}(\check\tau_n) + k
    \end{equation*}
    for any $k\ge0$.
    Hence:
    \begin{equation*}
        \prob{ \card{\frac{1}{\check n} \LDS_{r\sqrt{\check n}}(\tau_n) - F_\mathrm{LSKV}(r)} > \epsilon }
        \le \prob{ \card{\frac{1}{\check n} \LDS_{r\sqrt{\check n}}(\check\tau_n) - F_\mathrm{LSKV}(r)} > \epsilon_n }
    \end{equation*}
    where $\epsilon_n := \epsilon - \frac{r\sqrt{\check n}}{\check n} \cv{n}{\infty} \epsilon$ in probability.
    Then by \Cref{th: limit shape w/o fixed points} and dominated convergence theorem:
    \begin{equation*}
        \prob{ \card{\frac{1}{\check n} \LDS_{r\sqrt{\check n}}(\check\tau_n) - F_\mathrm{LSKV}(r)} > \epsilon_n }
        \\= \mean{ \prob{ \card{\frac{1}{\check n} \LDS_{r\sqrt{\check n}}(\check\tau_n) - F_\mathrm{LSKV}(r)} > \epsilon_n \quad\Bigg| t^{(n)} }}
        \cv{n}{\infty} 0 ,
    \end{equation*}
    which concludes the proof.
\end{proof}

\section{Proofs of the results on records}
\label{sec: proof records}

\subsection{High records}

\begin{lem}\label{lem: leftmost and upmost points in t-cyclic}
    For each $n$, let $t^{(n)}$ be a cycle type of size $n$ such that $t_1^{(n)}=0$, and $\pts$ be a geometric construction.
    Set $C := [0,1/2]\times[1/2,1]$.
    Define $Z_\text{left} = \left(U_\text{left},V_\text{left}\right)$ as the leftmost point in $\pts$, $Z_\text{left}^C = \left(U_\text{left}^C,V_\text{left}^C\right)$ as the leftmost point of $\pts\cap C$, and let $Z_\text{up}$, $Z_\text{up}^C$ be the upmost analogues.
    Those points are well-defined w.h.p.~as $n\to\infty$ and:
    \begin{equation*}
        \left(\big. nU_\text{left} , n\left(1-V_\text{up}\right) \right) \cv{n}{\infty} \left(E,E'\right)
        \quad;\quad
        n U_\text{left}^C \cv{n}{\infty} 2E
        \quad;\quad
        n \left(1-V_\text{up}^C\right) \cv{n}{\infty} 2E
    \end{equation*}
    in distribution, where $E,E'$ are independent $\Exp{1}$ variables.
    Furthermore, w.h.p.~as $n\to\infty$ the inequalities $U_\text{left} \le U_\text{left}^C \le V_\text{left}$ hold.
\end{lem}

\begin{proof}
    Recall from the geometric construction that the lists of x- and y-coordinates of points in $\pts$ are both equal to the same family $\{U_i\}_{i\in\Vert_t}$ of i.i.d.~$\Unif{[0,1]}$ variables.
    Thus $U_\text{left} =: U_{(1)}$ is their minimum and $V_\text{up} =: U_{(n)}$ is their maximum.
    For any $s,t\ge0$ we have:
    \begin{equation*}
        \prob{nU_{(1)} \ge s \;,\; n\left(1-U_{(n)}\right) \ge t}
        = \prob{\big. \forall i,\; U_i\in[s/n,1-t/n]}
        = (1-(s+t)/n)^n
        \cv{n}{\infty} e^{-s}e^{-t} .
    \end{equation*}
    This proves the first convergence in distribution.
    For the second one, use \Cref{lem: iid outside diagonal t-cyclic} and write, conditionally given $\card{\pts\cap C}$:
    \begin{equation*}
        \prob{nU_\text{left}^C \ge s}
        = \left( 1 - \frac{s/(2n)}{1/4} \right)^{\card{\pts\cap C}}
        = ( 1 - 2s/n )^{\card{\pts\cap C}} .
    \end{equation*}
    According to \Cref{cor: controlling number of points t-cyclic}, $\card{\pts\cap C}$ is concentrated around $n/4$.
    This readily implies $\prob{nU_\text{left}^C \ge s} \to e^{-s/2}$ as $n\to\infty$, and proves the second convergence in distribution.
    The third one follows similarly.
    
    For the last claim, it suffices to note that $U_\text{left}$ and $U_\text{left}^C$ are both $o_\P(1)$ whereas $V_\text{left}$ has distribution $\Unif{[0,1]}$, and $U_\text{left} \le U_\text{left}^C$ by definition.
\end{proof}

\begin{proof}[Proof of \Cref{th: high records t-cyclic}]
    Let $\pts$ be a geometric construction of $\tau_n$.
    Decompose it into its points outside the diagonal $\check\pts$ and its points on the diagonal $\pts_\Delta$.
    We shall use the notation of \Cref{lem: leftmost and upmost points in t-cyclic}, applied to the set $\check\pts$ of size $\check n := n-t_1^{(n)}$.
    According to this lemma, w.h.p.~the point $Z_\text{left}$ is above the diagonal $\Delta$ and the rectangle $[U_\text{left} , U_\text{left}^C]\times[V_\text{left} , V_\text{left}^C]$ does not intersect $\Delta$.
    Likewise, w.h.p.~the point $Z_\text{up}$ is above the diagonal and the rectangle $[U_\text{up}^C , U_\text{up}]\times[V_\text{up}^C , V_\text{up}]$ does not intersect $\Delta$.
    Under those events, the records of $\pts$ are all contained in
    \begin{equation*}
        \left( \Delta \cap [0,U_\text{left}]^2 \right)
        \quad\bigcup\quad 
        [U_\text{left} , U_\text{left}^C]\times[V_\text{left} , V_\text{left}^C]
        \quad\bigcup\quad 
        C
        \quad\bigcup\quad 
        [U_\text{up}^C , U_\text{up}]\times[V_\text{up}^C , V_\text{up}]
        \quad\bigcup\quad 
        \left( \Delta \cap [V_\text{up},1]^2 \right) .
    \end{equation*}
    \begin{figure}
        \centering
        \includegraphics[scale=.48]{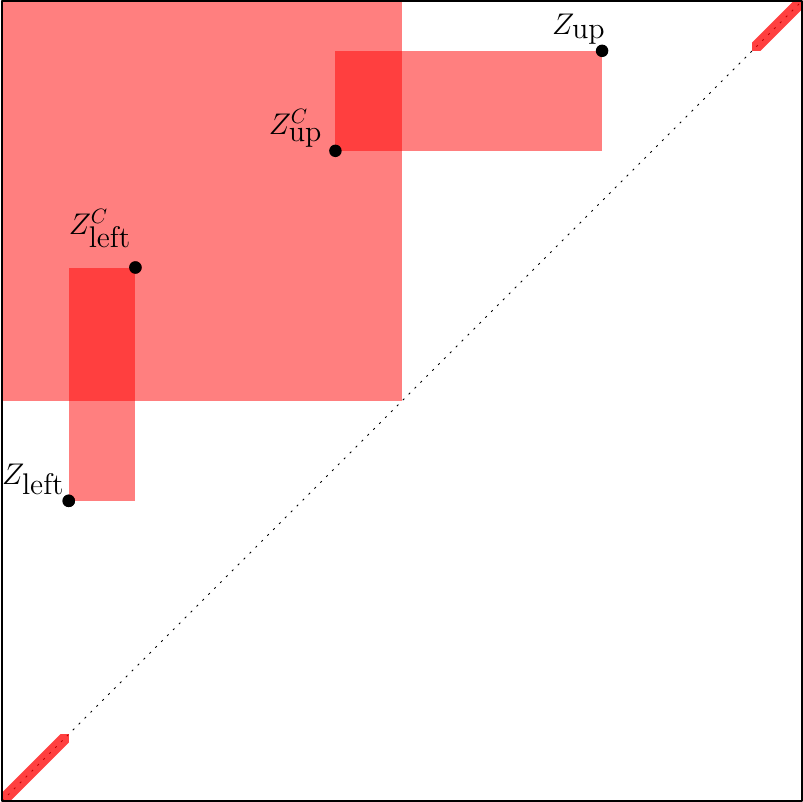}
        \hspace{7em}
        \includegraphics[scale=.48]{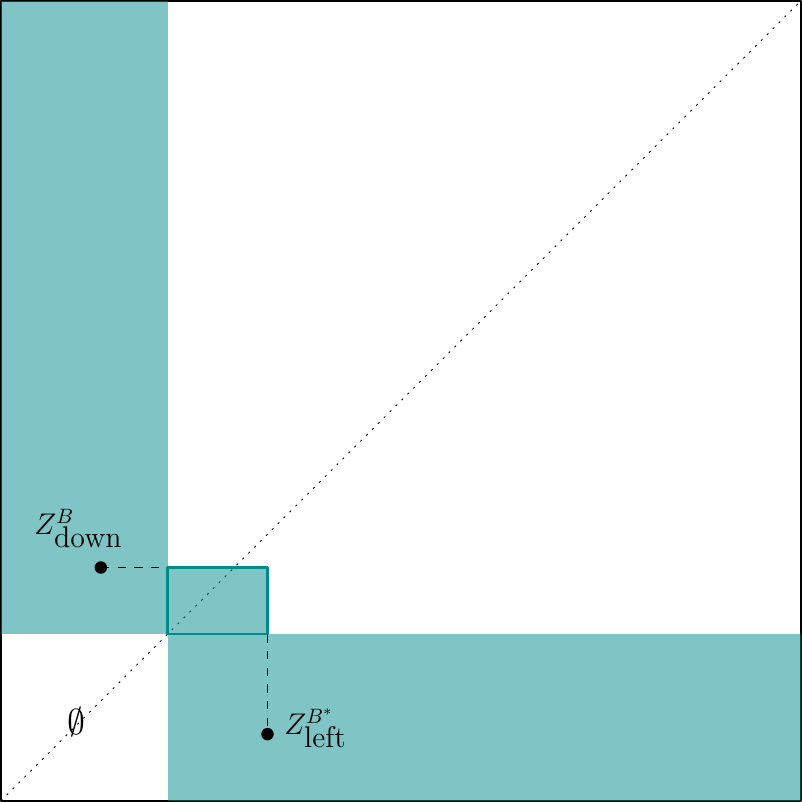}
        \caption{The zones in which we can find records (high ones on the left and low ones on the right).}
        \label{fig: zones records}
    \end{figure}
    This zone is shown on the left-hand side of \Cref{fig: zones records}.
    Note that the rectangles $[U_\text{left} , U_\text{left}^C]\times[V_\text{left} , V_\text{left}^C]$ and $[U_\text{up}^C , U_\text{up}]\times[V_\text{up}^C , V_\text{up}]$ may be flat, i.e.~the identities $U_\text{left} = U_\text{left}^C$ and $V_\text{up}^C = V_\text{up}$ may occur with non-negligible probability.
    
    Invertly, all points of $\pts_\Delta$ in $[0,U_\text{left}]^2\cup[V_\text{up},1]^2$, as well as the ones which are high records in $\pts\cap C$, are high records in $\pts$.
    From the previous two observations we deduce the following inequality w.h.p.~as $n\to\infty$:
    \begin{equation*}
        \card{ \hrec{\pts} - \left(\Big. \hrec{\check\pts\cap C} + S_n \right) }
        \\\le \card{\check\pts\cap \left([U_\text{left} , U_\text{left}^C]\times[V_\text{left} , V_\text{left}^C]\right) }
        + \card{\check\pts\cap \left([U_\text{up}^C , U_\text{up}]\times[V_\text{up}^C , V_\text{up}]\right) } 
    \end{equation*}
    where $S_n := \card{\pts_\Delta\cap[0,U_\text{left}]^2} + \card{\pts_\Delta\cap[V_\text{up},1]^2}$.
    For any $M>0$, we may apply \Cref{cor: controlling number of points t-cyclic} to the rectangle $[0,M/\check n] \times[0,1]$ to obtain $\card{ \check \pts \cap \left( [0,M/\check n] \times[0,1] \right) } \le 2M$ w.h.p.~as $n\to\infty$.
    Since $\check n U_\text{left}^C$ is tight by \Cref{lem: leftmost and upmost points in t-cyclic}, it holds that $U_\text{left}^C \le M/\check n$ with probability tending to $1$ as $M\to\infty$, uniformly in $n$. 
    This implies that the random variable $\card{\check\pts\cap \left([U_\text{left} , U_\text{left}^C]\times[V_\text{left} , V_\text{left}^C]\right)}$ is a $\O_\P(1)$.
    With the same reasoning applied to $1-V_\text{up}^C$ and the corresponding rectangle, we deduce:
    \begin{equation}\label{eq: first approx high records}
        \card{ \hrec{\pts} - \left(\Big. \hrec{\check\pts\cap C} + S_n \right) }
        = \O_\P(1) .
    \end{equation}
    Now it remains to analyze both terms in the sum.
    According to \Cref{lem: iid outside diagonal t-cyclic} and \Cref{cor: controlling number of points t-cyclic}, $\check\pts\cap C$ is a family of i.i.d.~uniform points in $C$, with size $\check n/{4} + o_\P\left(\check n\right)$.
    We can thus apply \Cref{th: records uniform} to deduce
    \begin{equation}\label{eq: square term high records}
        \frac{\hrec{\check\pts\cap C} - \log \check n}{\sqrt{\log \check n}}
        \cv{n}{\infty} \Normal{0}{1}
    \end{equation}
    in distribution.
    Let us turn to the fixed points.
    $\check\pts$ and $\pts_\Delta$ are independent, so $(U_\text{left},V_\text{up})$ and $\pts_\Delta$ are as well.
    Therefore, conditionally given $(U_\text{left},V_\text{up})$, $S_n$ follows a $\Binomial{t_1^{(n)}}{U_\text{left}\p1\m V_\text{up}}$ law.
    Recall from \Cref{lem: leftmost and upmost points in t-cyclic} that 
    \[
        \check n\left(U_\text{left}+1- V_\text{up}\right) \cv{n}{\infty} \Gam{2}{1} := \Exp{1}*\Exp{1} 
    \]
    in distribution, where $\Exp{1}*\Exp{1}$ denotes the law of the sum of two independent $\Exp{1}$ random variables.
    If ${t_1^{(n)}}/{\check n} = \O(1)$ as $n\to\infty$ then $S_n=\O_\P(1)$ (use e.g.~the Bienaym\'e–-Chebyshev inequality), whereas if ${t_1^{(n)}}/{\check n} \to \infty$ as $n\to\infty$ then
    \begin{equation}\label{eq: diagonal term high records}
        \frac{\check n}{t_1^{(n)}} S_n \cv{n}{\infty} \Gam{2}{1}
    \end{equation}
    in distribution.
    Now we can use \Cref{eq: first approx high records,eq: square term high records,eq: diagonal term high records} to identify three regimes in the asymptotics of $\hrec{\pts}$.
    
    \paragraph{``Mostly uniform'' regime:}
    Assume $\frac{\check n}{n/\sqrt{\log n}} \to \infty$.
    Then $\log n \sim \log \check n$ as $n\to\infty$, and since $t_1^{(n)} \le n$, we have $\frac{t_1^{(n)}}{\check n \sqrt{\log\check n}} \to 0$.
    Consequently the fluctuations of $\hrec{\check \pts \cap C}$ dominate, and
    \begin{equation*}
        \frac{\hrec{\pts} - \log \check n}{\sqrt{\log \check n}}
        \cv{n}{\infty} \Normal{0}{1} .
    \end{equation*}
    
    \paragraph{``Mostly fixed points'' regime:}
    Assume $\frac{\check n}{n/\sqrt{\log n}} \to 0$ as $n\to\infty$.
    Since $\log \check n \le \log n$ and $t_1^{(n)} = n - \check n$, we have $\frac{t_1^{(n)}}{\check n \sqrt{\log\check n}} \to \infty$.
    Consequently the fluctuations of $S_n$ dominate, and
    \begin{equation*}
        \frac{\check n}{t_1^{(n)}}\left(\hrec{\pts} - \log \check n \right)
        \cv{n}{\infty} \Gam{2}{1} .
    \end{equation*}

    \paragraph{``Intermediate'' regime:}
    Assume $\frac{\check n}{n/\sqrt{\log n}} \to \alpha \in (0,\infty)$.
    Then $\log \check n \sim \log n$ as $n\to\infty$, and since $t_1^{(n)} = n - \check n$, we deduce $\frac{t_1^{(n)}}{\check n\sqrt{\log\left(\check n\right)}} \to 1/\alpha$.
    Thus we need to understand the interplay between the limit laws of \Cref{eq: square term high records,eq: diagonal term high records}.
    To do this, notice that conditionally given $\left( U_\text{left} , V_\text{up} \right)$, the variables $S_n$ and $\check\pts\cap C$ are independent.
    Recall that under this conditioning, $S_n$ follows a $\Binomial{t_1^{(n)}}{U_\text{left}\p1\m V_\text{up}}$ law.
    According to \Cref{lem: conditionally on coordinates}, the conditional law of $\perm{\check\pts\cap C}$ given $\left( U_\text{left} , V_\text{up} \right) = \left( U_{(1)} , U_{(n)} \right)$ and $\card{\check\pts\cap C}$ is uniform.
    From this property we deduce independence of the limit laws in \Cref{eq: diagonal term high records,eq: square term high records}, whence:
    \begin{multline*}
        \frac{\hrec{\pts} - \log \check n}{\frac{t_1^{(n)}}{\check n} + \sqrt{\log \check n}}
        = \frac{\hrec{\check\pts\cap C} - \log \check n}{(1+1/\alpha+o(1))\sqrt{\log \check n}}
        + \frac{S_n}{(1+\alpha+o(1))\frac{t_1^{(n)}}{\check n}}
        + o_\P(1)
        \\\cv{n}{\infty}
        \frac{\alpha}{\alpha+1}\Normal{0}{1} * \frac{1}{1+\alpha}\Gam{2}{1}
    \end{multline*}
    in distribution, as announced.
\end{proof}

\subsection{Low records}

\begin{lem}\label{lem: downmost and leftmost points in bands}
    For each $n$, let $t^{(n)}$ be a cycle type of size $n$ such that $t_1^{(n)}=0$, and $\pts$ be a geometric construction.
    Fix $\epsilon>0$ and define $B_n := \left[0, \epsilon/ \sqrt{n}\right] \times \left[\epsilon/ \sqrt{n} ,1\right]$ and $B_n^* := \left[\epsilon/ \sqrt{n} ,1\right] \times \left[0, \epsilon/ \sqrt{n}\right]$.
    Then
    \begin{equation*}
        \frac{\card{\pts\cap B_n}}{\epsilon\sqrt n} \cv{n}{\infty} 1
        \quad;\quad
        \frac{\card{\pts\cap B_n^*}}{\epsilon\sqrt n} \cv{n}{\infty} 1
    \end{equation*}
    in probability.
    In particular w.h.p.~we can define $Z_\text{down}^{B} = \left(U_\text{down}^B , V_\text{down}^B\right)$ as the downmost point in $\pts\cap B_n$ and $Z_\text{left}^{B^*} = \left(U_\text{left}^{B^*} , V_\text{left}^{B^*}\right)$ as the leftmost point in $\pts\cap B_n^*$, and they satisfy:
    \begin{equation*}
        {\epsilon\sqrt{n}}V_\text{down}^B -\epsilon^2 \cv{n}{\infty} \Exp{1}
        \quad;\quad
        {\epsilon\sqrt{n}}U_\text{left}^{B^*} -\epsilon^2 \cv{n}{\infty} \Exp{1}
    \end{equation*}
    in distribution.
\end{lem}

\begin{proof}
    By \Cref{cor: controlling number of points t-cyclic}, $\card{\pts\cap B_n}$ is distributed like a sum of three $\Binomial{\lfloor n/3\rfloor \pm1}{\left(1-\epsilon/\sqrt n\right) \big. \epsilon/\sqrt n}$, and likewise for $B_n^*$.
    The first claim readily follows from standard concentration inequalities.
    Then w.h.p.~as $n\to\infty$, the sets $\pts\cap B_n$ and $\pts\cap B_n^*$ are non-empty and $Z_\text{down}^B$, $Z_\text{left}^{B^*}$ are well-defined.
    According to \Cref{lem: iid outside diagonal t-cyclic}, conditionally given $\card{\pts\cap B_n}$, the set $\pts\cap B_n$ consists of i.i.d.~uniform points in $B_n$.
    We can thus write, for any $s>0$ and large enough $n$, conditionally given $\card{\pts\cap B_n}$:
    \begin{equation*}
        \prob{\epsilon\sqrt{n}V_\text{down}^B > s+\epsilon^2}
        = \prob{V_\text{down}^B > \frac{s}{\epsilon\sqrt n}+\epsilon/\sqrt{n}}
        = \left(1 - \frac{1}{1-\epsilon/\sqrt n}\frac{s}{\epsilon\sqrt n}\right)^{\card{\pts\cap B_n}} .
    \end{equation*}
    Hence $\prob{\epsilon\sqrt{n}V_\text{down}^B > s+\epsilon^2} \to e^{-s}$ as $n\to\infty$.
    The same holds for $U_\text{left}^{B^*}$, and this concludes the proof.
\end{proof}

\begin{proof}[Proof of \Cref{th: low records t-cyclic}]
    For each $n$, let $t^{(n)}$ be a cycle type of size $n$ and $\pts$ be a geometric construction of $t^{(n)}$.
    Fix $\epsilon>0$.

    \paragraph{Case 1:}
    Suppose $t_1^{(n)} = \O\left(\sqrt n\right)$.
    
    Define $C_n := \left[ 0,\epsilon/\sqrt{\check n} \right]^2$.
    Then $\card{\check\pts\cap C_n}$ is distributed like a sum of three $\Binomial{\lfloor\check n/3\rfloor \pm1}{\epsilon^2/\check n}$ variables by \Cref{cor: controlling number of points t-cyclic}, and $\card{\pts_\Delta\cap C_n}$ follows a $\Binomial{t_1^{(n)}}{\epsilon/\sqrt{\check n}}$ law.
    Therefore, there exists $\delta(\epsilon) = \Landau{o}{\epsilon\to0}{1}$, independent of $n$, such that:
    \begin{equation}\label{eq: proba empty square low records}
        \limsup_{n\to\infty} \prob{\pts\cap C_n \ne \emptyset} \le \delta(\epsilon) .
    \end{equation}
    Define $B_n := \left[0, \epsilon/ \sqrt{\check n}\right] \times \left[\epsilon/ \sqrt{\check n} ,1\right]$ and $B_n^* := \left[\epsilon/ \sqrt{\check n} ,1\right] \times \left[0, \epsilon/ \sqrt{\check n}\right]$.
    According to \Cref{lem: downmost and leftmost points in bands} we can define w.h.p.~$Z_\text{down}^{B} = \left(U_\text{down}^B , V_\text{down}^B\right)$ as the downmost point in $\check\pts\cap B_n$ and $Z_\text{left}^{B^*} = \left(U_\text{left}^{B^*} , V_\text{left}^{B^*}\right)$ as the leftmost point in $\check\pts\cap B_n^*$.
    See the right-hand side of \Cref{fig: zones records} for a representation.

    Define the event $E_n := \{\pts\cap C_n = \emptyset\} \cap \left\{\check\pts\cap B_n \ne \emptyset\right\} \cap \left\{\check\pts\cap B_n^* \ne \emptyset\right\}$.
    Under this event, all points which are low records in $\check\pts\cap B_n$ or in $\check\pts\cap B_n^*$ are also low records in $\pts$.
    Reciprocally, all low records of $\pts$ lie in $B_n$, in $B_n^*$, or in $\left[ \epsilon/\sqrt{\check n} , U_\text{left}^{B^*} \right]\times\left[ \epsilon/\sqrt{\check n} , V_\text{down}^B\right]$.
    Hence under $E_n$:
    \begin{equation}\label{eq: low records two zones}
        \card{\Big. \lrec{\pts} - \lrec{\check\pts\cap B_n} - \lrec{\check\pts\cap B_n^*} } 
        \le \card{\pts\cap\left( \left[ \epsilon/\sqrt{\check n} , U_\text{left}^{B^*} \right]\times\left[ \epsilon/\sqrt{\check n} , V_\text{down}^B \right] \right)} .
    \end{equation}
    Using \Cref{lem: downmost and leftmost points in bands} and the fact that $t_1^{(n)} = \O\left(\sqrt n\right)$ we obtain:
    \begin{equation*}
        \card{\pts\cap\left( \left[ \epsilon/\sqrt{\check n} , U_\text{left}^{B^*} \right]\times\left[ \epsilon/\sqrt{\check n} , V_\text{down}^B \right] \right)} = \O_\P(1)
    \end{equation*}
    as $n\to\infty$, for fixed $\epsilon>0$.
    Thus it suffices to study the records in $B_n$ and $B_n^*$.
    We can apply \Cref{lem: iid outside diagonal t-cyclic}, \Cref{lem: downmost and leftmost points in bands} and \Cref{th: records uniform} to obtain
    \begin{equation*}
        \frac{\lrec{\check\pts\cap B_n}}{\frac12\log\left(n\right)}
        \cv{n}{\infty} 1
    \end{equation*}
    in probability, and likewise for $B_n^*$.
    Using \Cref{eq: low records two zones,eq: proba empty square low records} we deduce that for any $\eta>0$:
    \begin{equation*}
        \limsup_{n\to\infty}\prob{ \card{\frac{\lrec{\pts}}{\frac12\log\left(n\right)} - 2} > \eta }
        \le \delta(\epsilon) .
    \end{equation*}
    Since this holds for any $\epsilon>0$, the desired convergence in probability follows.

    \paragraph{Case 2:}
    Suppose $\sqrt n = o\left(t_1^{(n)}\right)$.

    Let $Z_\Delta = (U_\Delta,U_\Delta)$ be the leftmost point of $\pts_\Delta$.
    Then it is standard that
    \begin{equation}\label{eq: limit law U Delta}
        t_1^{(n)} U_\Delta \cv{n}{\infty} \Exp{1}
    \end{equation}
    in distribution.
    By \Cref{cor: controlling number of points t-cyclic} and since ${\check n}/{(t_1^{(n)})^2} \to 0$, w.h.p.~it holds that $\check\pts\cap [0,U_\Delta]^2 = \emptyset$.
    Define two bands $B_\Delta := [0,U_\Delta]\times[U_\Delta,1]$ and $B_\Delta^* := [U_\Delta,1]\times[0,U_\Delta]$.
    Then under the event $\left\{ \check\pts\cap [0,U_\Delta]^2 = \emptyset \right\}$ we may write:
    \begin{equation}\label{eq: low records Delta}
        \lrec{\pts} = 1 + \lrec{\check\pts\cap B_\Delta} + \lrec{\check\pts\cap B_\Delta^*} .
    \end{equation}
    Using \Cref{cor: controlling number of points t-cyclic} and \Cref{lem: iid outside diagonal t-cyclic} conditionally given $U_\Delta$, $\check\pts\cap B_\Delta$ is a family of i.i.d.~uniform points in $B_\Delta$ whose size is distributed like a sum of three $\Binomial{\lfloor\check n/3\rfloor \pm1}{U_\Delta(1-U_\Delta)}$ variables.
    The same holds for $B_\Delta^*$.
    \begin{itemize}
        \item If $n=\O\left(t_1^{(n)}\right)$ then, by \Cref{eq: limit law U Delta}, $\card{\pts\cap B_\Delta}$ and $\card{\pts\cap B_\Delta^*}$ are both $\O_\P(1)$ as $n\to\infty$.
        Thus by \Cref{eq: low records Delta}, $\lrec{\pts}$ is also a $\O_\P(1)$.
        \item If $t_1^{(n)}=o(n)$ then, by \Cref{eq: limit law U Delta}, 
        $\frac{t_1^{(n)}}{n}\card{\pts\cap B_\Delta} \cv{n}{\infty} \Exp{1}$
        in distribution.
        Therefore:
        \begin{equation*}
            \frac{\log\card{\pts\cap B_\Delta}}{\log\left(n/t_1^{(n)}\right)}
            \cv{n}{\infty} 1
        \end{equation*}
        in probability.
        In particular $\card{\pts\cap B_\Delta}\to\infty$ in probability and we can apply \Cref{th: records uniform} to deduce:
        \begin{equation*}
            \frac{\lrec{\pts\cap B_\Delta}}{\log\card{\pts\cap B_\Delta}} \cv{n}{\infty} 1
        \end{equation*}
        in probability.
        The same holds for $B_\Delta^*$, and the previous two equations along with \Cref{eq: low records Delta} yield:
        \begin{equation*}
            \frac{\lrec{\pts}}{\log\left(n/t_1^{(n)}\right)}
            \cv{n}{\infty} 2
        \end{equation*}
        as desired.
    \end{itemize}
    This concludes the proof of \Cref{th: low records t-cyclic}.
\end{proof}

\begin{proof}[Proof of \Cref{prop: low records fluctuations}]
    We use the same notation as in the proof of \Cref{th: low records t-cyclic}.
    Recall that under $E_n$:
    \begin{equation}\label{eq: approx low records few fixed}
        \card{ \lrec{\pts} - \lrec{\check\pts\cap B_n} - \lrec{\check\pts\cap B_n^*} } \le \O_\P(1)
    \end{equation}
    where by \Cref{lem: iid outside diagonal t-cyclic}, \Cref{lem: downmost and leftmost points in bands} and \Cref{th: records uniform}:
    \begin{equation}\label{eq: low records CLT each zone}
        \frac{ \lrec{\check\pts\cap B_n} - \frac12 \log n }{\sqrt{\frac12 \log n}}
        \cv{n}{\infty} \Normal{0}{1}
    \end{equation}
    in distribution, and likewise for $B_n^*$.
    It remains to understand how $\lrec{\pts\cap B_n}$ and $\lrec{\pts\cap B_n^*}$ are correlated.

    \paragraph{Products of $2$-cycles:}
    Suppose that $t_1^{(n)}+2t_2^{(n)}=n$, i.e.~that $\tau_n$ is a random involution with $t_1^{(n)}$ fixed points.
    Then the sets $\check\pts\cap B_n$ and $\check\pts\cap B_n^*$ are a.s.~symmetric with respect to the diagonal.
    Therefore, $\perm{\check\pts\cap B_n}$ and $\perm{\check\pts\cap B_n^*}$ are a.s.~inverse of each other, and $\lrec{\perm{\check\pts\cap B_n}} = \lrec{\perm{\check\pts\cap B_n^*}}$ almost surely.
    Using \Cref{eq: proba empty square low records,eq: approx low records few fixed,eq: low records CLT each zone} this yields:
    \begin{equation*}
        \frac{\lrec{\pts}-\log n}{\sqrt{\frac12\log n}}
        \cv{n}{\infty} \Normal{0}{4}
    \end{equation*}
    as announced.
    
    \paragraph{Products of $3$-cycles:}
    Suppose that $t_1^{(n)}+3t_3^{(n)}=n$, i.e.~that $\tau_n$ is a random product of $3$-cycles.
    We claim the following: conditionally on the event $E_n$ and given the family $\pts\cap B_n$, the permutation $\perm{\pts\cap B_n^*}$ has size $\card{\pts\cap B_n}$ and is uniformly random.
    Let us explain why.

    Set $\epsilon_n := \epsilon/\sqrt{\check n}$ and $D_n := \carre\setminus\left( C_n\cup B_n\cup B_n^*\right)$.
    Consider three independent variables $U,V,W$ distributed under $\Unif{[0,1]}$.
    We may a.s.~distinguish four outcomes:
    \begin{itemize}
        \item $U,V,W$ are all less than $\epsilon_n$, i.e.~$(U,V),(V,W),(W,U)$ all lie within $C_n$;
        \item $U,V,W$ are all greater than $\epsilon_n$, i.e.~$(U,V),(V,W),(W,U)$ all lie within $D_n$;
        \item at least one point, say $(U,V)$, is in $B_n$.
        Then either $W<\epsilon_n$, in which case $(V,W)\in B_n^*$ and $(W,U)\in C_n$, or $W>\epsilon_n$, in which case $(V,W)\in D_n$ and $(W,U)\in B_n^*$.
    \end{itemize}
    This holds for each triple $U_{3,k}^1,U_{3,k}^2,U_{3,k}^3$ in the construction of $\pts$.
    Under the event $E_n$, only two possibilities remain: either $Z_{3,k}^1,Z_{3,k}^2,Z_{3,k}^3\in D_n$, or exactly one point is in $B_n$ and another one is in $B_n^*$, and they share a coordinate.
    The conditional law of $\pts\cap B_n^*$ under $E_n$ and $\pts\cap B_n$ can thus be described as follows:
    write $\pts\cap B_n = \{(U_i,V_i)\}$; let $(W_i)$ be i.i.d.~$\Unif{[\epsilon_n,1]}$ variables; and set $\pts\cap B_n^* = \{(W_i,U_i)\}$.
    One can easily see that for any distinct $u_1,\dots,u_r$, if $W_1,\dots,W_r$ are i.i.d.~$\Unif{[\epsilon_n,1]}$ variables then $\perm{(W_1,u_1),\dots,(W_r,u_r)}$ is uniformly random (e.g.~switching $W_i$ with $W_j$ shows that this is the Haar measure on the symmetric group).
    This proves that conditionally on the event $E_n$ and given the family $\pts\cap B_n$, the permutation $\perm{\pts\cap B_n^*}$ has size $\card{\pts\cap B_n}$ and is uniformly random.

    This claim implies that conditionally on the event $E_n$ and given the number $\card{\pts\cap B_n}$, the random permutations $\perm{\pts\cap B_n}$ and $\perm{\pts\cap B_n^*}$ have size $\card{\pts\cap B_n}$ and are independent uniform.
    Using \Cref{eq: proba empty square low records,eq: approx low records few fixed,eq: low records CLT each zone} we deduce:
    \begin{equation*}
        \frac{\lrec{\pts}-\log n}{\sqrt{\frac12\log n}}
        \cv{n}{\infty} \Normal{0}{2},
    \end{equation*}
    and this concludes the proof of the proposition.
\end{proof}

\section{Proofs of the results on pattern counts}
\label{sec: proof pattern}

\subsection{Dependency graphs and asymptotic normality}\label{sec: dependency graph}

We recall the notion of dependency graph, as introduced in \cite{PL82}.
Let $(Z_i)_{i\in V}$ be a family of random variables indexed by some set $V$.
We say that a graph $(V,E)$ is a {\em dependency graph} for $(Z_i)_{i\in V}$ if it has the following property:
for any disjoint subsets $V_1,V_2$ of $V$ which are not linked by any edge in $E$, the two families $(Z_i)_{i\in V_1}$ and $(Z_i)_{i\in V_2}$ are independent.

Let $t$ be a cycle type of size $n$, and $\L_t$ be the graph defined in \Cref{sec: notation}.
The following lemma is a direct consequence of the definitions.

\begin{lem}\label{lem: dependency graph}
    In the geometric construction of \Cref{lem: geometric construction t-cyclic}, the graph $\L_t=(\Vert_t,\Edges_t)$ is a dependency graph for the variables $(Z_i)_{i\in\Vert_t}$.
\end{lem}

Now we can use $\L_t$ to construct another dependency graph.
Let $\binom{\Vert_t}{r}$ denote the set of subsets of size $r$ in $\Vert_t$.
Define a graph $\binom{\L_t}{r}$ with vertices $\binom{\Vert_t}{r}$ as follows:
$I,J\in\binom{\Vert_t}{r}$ share an edge if and only if there exist $i\in I$ and $j\in J$ such that $i=j$ or $i\sim j$ in $\L_t$.
Using \Cref{lem: dependency graph}, it is straightforward to check that $\binom{\L_t}{r}$ is a dependency graph for the family of random variables $\left(\big. (Z_i)_{i\in I} \right)_{I\in\binom{\Vert_t}{r}}$.

The idea that the variables $(Z_i)_{i\in\Vert_t}$ are ``partly dependent'' is motivated by the fact that these graphs have low degrees:
the degrees of $\L_t$ are bounded by $2$, and a rough upper bound for the degrees of $\binom{\L_t}{r}$ is $3r\binom{n}{r-1} = \O\left( n^{r-1} \right)$.
We could then apply \cite[Theorem~2]{J88} to deduce asymptotic normality for the pattern count $X_\pi(\tau_n)$, but we rather use Stein's method to obtain a stronger result.

\begin{proof}[Proof of \Cref{th: pattern Stein}]
    We aim to apply \cite[Theorem~3.5]{H18}, which is based on the results of \cite{CR10} and \cite{R11}, to the decomposition \eqref{eq: write X as sum of partly dependent} of $X_\pi(\tau_n)$.
    According to \Cref{lem: dependency graph} and the discussion following it, $\binom{\L_t}{r}$ is a dependency graph for the family of random variables $\left( \One{ \perm{Z_i,i\in I} = \pi} \right)_{I\in\binom{\Vert_t}{r}}$.
    Thus we can use \cite[Theorem~3.5]{H18} with $B=1$, $D\le 3r\binom{n}{r-1}$, and $N=\binom{n}{r}$.
    This yields the announced bound on the Kolmogorov distance.

    Now suppose $\var{X_\pi(\tau_n)}=\sigma^2 n^{2r-1} + o(n^{2r-1})$ as $n\to\infty$, for some $\sigma\ge0$.
    If $\sigma=0$ then
    \begin{equation*}
        \frac{X_\pi(\tau_n)-\mean{X_\pi(\tau_n)}}{n^{r-1/2}} \cv{n}{\infty} 0
    \end{equation*}
    in probability, by Bienaymé--Chebychev's inequality.
    If $\sigma>0$ then the previous bound on the Kolmogorov distance is a $\O(n^{-1/2})$.
    In particular it goes to $0$ as $n\to\infty$, which implies convergence in distribution of the random variables.
    This concludes the proof.
\end{proof}

\subsection{Computation of the variance}

\begin{proof}[Proof of \Cref{th: pattern normality t-cyclic with fixed points}]
    We only prove marginal convergence, for convenience.
    Joint convergence is proved by the same method, using the Cram\'er--Wold theorem.

    Fix a pattern $\pi\in\S_r$.
    Throughout the proof we may omit the indices $\pi,p_1,p_2,n$.
    Thanks to \Cref{th: pattern Stein}, it suffices to compute the mean and variance of $X:=X_\pi(\tau)$.
    Let $\pts=(Z_i)_{i\in\Vert_t}$ be a geometric construction of $\tau_n$.
    The first idea is to approximate $X$ with a new variable
    \begin{equation*}
        Y := \sum_{I\in\binom{\Vert_t}{r}\text{ indep.}}\One{\perm{Z_i,i\in I}=\pi} ,
    \end{equation*}
    where the sum ranges over all subsets $I\in\binom{\Vert_t}{r}$ which induce an independent (i.e.~empty) subgraph of $\L_t$.
    Since the number of non-independent $r$-sized subgraphs of $\L_t$ is bounded by $2(r-1)\binom{n}{r-1}$, almost surely it holds that $\card{X - Y} \le 2(r-1)\binom{n}{r-1}$.
    In particular $\mean{X} = \mean{Y} + \O\left(n^{r-1}\right)$ and $\sqrt{\var{X}} = \sqrt{\var{Y}} + \O\left(n^{r-1}\right)$, so it suffices to prove the desired result for $Y$.
    To do this, we shall rewrite $Y$ while keeping track of how many fixed points are picked in the sum.

    Decompose $\Vert_t$ as $\Vert_1\sqcup \Vert_2\sqcup \Vert_3$ where $\Vert_m := \{(p,k,l)\in\Vert_t : p=m\}$ for $m\in\{1,2\}$ correspond to fixed points and points in $2$-cycles, and $\Vert_3 := \{(p,k,l)\in\Vert_t : p\ge3\}$ corresponds to all other points.
    Let $U_1, \dots, U_r, V_1, \dots, V_r$ be i.i.d.~$\Unif{[0,1]}$ random variables.
    For any $q\in \{0,1,\dots,r\}$, define:
    \begin{equation*}
    \mu_{q}:=\prob{\perm{\big. (U_1,U_1), \dots, (U_q,U_q), (U_{q+1},V_{q+1}), \dots, (U_r,V_r)}=\pi} ,
    \end{equation*}
    so that 
    \begin{equation*}
        \mu := \mu_\pi^{p_1} = \sum_{q=0}^r 
        \binom{r}{q}p_1^q(1-p_1)^{r-q}
        \mu_{q} .
    \end{equation*}
    Define $\mathcal{I}_q := \left\{ I\in\binom{\Vert_t}{r} \text{ indep., }\card{I\cap \Vert_1}=q \right\}$ and $c_q:=\card{\mathcal{I}_q}$.
    For any $I\in\mathcal{I}_q$, the variables $Z_i,i\in I$ are independent by \Cref{lem: dependency graph}, and $\prob{\perm{Z_i,i\in I}=\pi}=\mu_q$.
    Therefore, we can write $\E Y = \sum_{q=0}^r c_q \mu_{q}$.
    Since the number of non-independent $I\in\binom{\Vert_t}{r}$ is $\O(n^{r-1})$ and the total number of $I\in\binom{\Vert_t}{r}$ with $|I\cap \Vert_1|=q$ is 
    \begin{equation*}
        \binom{|\Vert_1|}{q}\binom{|\Vert_t\setminus \Vert_1|}{r-q} = \frac{(np_1)^q+o(n^{q\m 1/2})}{q!}\frac{(n(1\mm p_1))^{r\m q}+o(n^{r\m q\m 1/2})}{(r\m q)!} = \binom{n}{r}\binom{r}{q}p_1^q(1\mm p_1)^{r\m q} + o(n^{r\m 1/2}) ,
    \end{equation*}
    we deduce $\E Y = \binom{n}{r}\mu + o(n^{r-1/2})$.
    It remains to study the variance of $Y$, and for this we rely on the method of $U$-statistics developed in \cite{H48}.
    The idea is to decompose $Y$ using its univariate projections.
    First define:
    \begin{equation*}
        \bar Y := Y-\E Y = \sum_{q=0}^r\sum_{I\in\mathcal{I}_q}\bar\phi_q(Z_i,i\in I) 
        \quad\text{where}\quad
        \bar\phi_q(z_1,\dots,z_r):=\One{\perm{z_1,\dots,z_r}=\pi}-\mu_{q} .
    \end{equation*}
    Then set for any $q\in \{0,1,\dots,r\m1\}$ and $z\in\carre$:
    \begin{equation*}
        \varphi_q(z) := \prob{\perm{\big. (U_1,U_1), \dots, (U_q,U_q), (U_{q+1},V_{q+1}), \dots, (U_{r-1},V_{r-1}), z} = \pi} ,
    \end{equation*}
    so that for any $z\in\carre$:
    \begin{equation}\label{eq: decomposition psi with fixed points}
        \psi(z) := \psi_\pi^{p_1}(z) = \sum_{q=0}^{r-1} 
        \binom{r-1}{q}p_1^q(1-p_1)^{r-q-1}
        \varphi_q(z) .
    \end{equation}
    It follows from definition that for any $I\in\mathcal{I}_q$ we have:
    \begin{align}\label{eq: conditioning fixed point or not}
        \mean{\bar\phi_q(Z_i,i\in I) \;\big| Z_j} =
        \left\{
        \begin{array}{ll}
            \varphi_q(Z_j) - \mu_q & \text{ if }j\in I\setminus\Vert_1 ;\\
            \varphi_{q-1}(Z_j) - \mu_q & \text{ if }j\in I\cap\Vert_1 .
        \end{array}
        \right.
    \end{align}
    Now for any $I\in\mathcal{I}_q$ and $(z_i)_{i\in I}\in\left(\carre\right)^I$, define the ``residual function'':
    \begin{equation*}
        \phi_I^*(z_i,i\in I)
        := \bar\phi_q(z_i,i\in I) 
        - \sum_{i\in I\cap \Vert_1} \left(\varphi_{q\m1}(z_i)-\mu_{q}\right)
        - \sum_{i\in I\setminus \Vert_1} \left(\varphi_q(z_i)-\mu_{q}\right) .
    \end{equation*}
    Notice that if the sum over $i\in I\cap \Vert_1$ is non-empty then $q\ge1$ i.e.~$q\m1\ge0$.
    A key property of $\phi_I^*$ is that its univariate projections vanish:
    for any $j\in I$, a.s.~$\mean{\phi_I^*(Z_i,i\in I) \;\big| Z_j}=0$.
    This is checked directly by distinguishing between $j\in\Vert_1$ and $j\notin\Vert_1$ to use \eqref{eq: conditioning fixed point or not}, observing that most terms have zero mean and that two of them cancel out.
    With this notation, we can decompose $\bar Y$ as:
    \begin{equation*}
        \bar Y 
        = \sum_{I\in\binom{\Vert_t}{r}\text{ indep.}} \phi_I^*(Z_i,i\in I)
        + \sum_{q=0}^r\sum_{I\in\mathcal{I}_q}\left(\sum_{i\in I\cap \Vert_1}\varphi_{q\m1}(Z_i) + \sum_{i\in I\setminus \Vert_1}\varphi_q(Z_i) - r\mu_{q}\right)
        = Y_* + \hat Y.
    \end{equation*}
    When expanding $Y_*^2$, one can distinguish the pairs $(I,J)$ of $r$-sized independent subgraphs of $\L_t$ into three categories: the ones sharing no common vertex and connected by no edge; the ones sharing exactly one vertex or connected by exactly one edge; and the others.
    If $(I,J)$ is in the first category then $\mean{\phi_I^*(Z_i,i\in I)\phi_J^*(Z_j,j\in J)}=0$ by independence (see \Cref{lem: dependency graph}).
    If $(I,J)$ is in the second category, write $(\alpha,\beta)\in I\times J$ for the unique pair of connected or equal vertices.
    Then $I\setminus\{\alpha\}$ is independent of $J\setminus\{\beta\}$ and we can write a.s.:
    \begin{equation*}
        \mean{\left.\big. \phi_I^*(Z_i,i\in I)\phi_J^*(Z_j,j\in J) \;\right| Z_\alpha,Z_\beta}
        = \mean{\left.\big. \phi_I^*(Z_i,i\in I) \;\right| Z_\alpha} \mean{\left.\big. \phi_J^*(Z_j,j\in J) \;\right| Z_\beta} = 0 .
    \end{equation*}
    In particular $\mean{\phi_I^*(Z_i,i\in I)\phi_J^*(Z_j,j\in J)}=0$.
    Since the number of pairs in the third category is bounded by $9r^2\binom{n}{r}\binom{n}{r-2}$, we deduce $\mean{Y_*^2} = \O\left(n^{2r-2}\right)$ as $n\to\infty$.
    Thus $\sqrt{\V\bar Y} = \sqrt{\V\hat Y} + \O\left(n^{r-1}\right)$, and it suffices to study $\hat Y$.
    Since $\sum_{q=0}^r c_q\mu_q = \binom{n}{r}\mu + o\left(n^{r-1/2}\right)$, we can rewrite it as:
    \begin{equation*}
        \hat Y = \sum_{i\in \Vert_1}\sum_{q=1}^r c_{i,q}\varphi_{q-1}(Z_i)
        + \sum_{i\in \Vert_t\setminus \Vert_1}\sum_{q=0}^{r-1} c_{i,q}\varphi_{q}(Z_i)
        - r\binom{n}{r}\mu + o\left(n^{r-1/2}\right) ,
    \end{equation*}
    where the $o$ is deterministic, and for any $q\in\{0,1,\dots,r\}$ and $i\in \Vert_t$, $c_{i,q}$ is the number of $I\in\mathcal{I}_q$ containing $i$.
    These numbers satisfy, as $n\to\infty$:
    \begin{equation*}
        \left\{
        \begin{array}{ll}
        c_{i,q} = \binom{|\Vert_1|-1}{q-1}\binom{|\Vert_t\setminus \Vert_1|}{r-q} + \O(n^{r-2}) = \binom{n}{r-1}\binom{r-1}{q-1}p_1^{q-1}(1\m p_1)^{r-q} + o(n^{r-3/2})
        \text{ uniformly in }i\in \Vert_1\vspace{.5em} ;\\
        c_{i,q} = \binom{|\Vert_1|}{q}\binom{|\Vert_t\setminus \Vert_1|-1}{r-q-1} + \O(n^{r-2}) = \binom{n}{r-1}\binom{r-1}{q}p_1^{q}(1\m p_1)^{r-q-1} + o(n^{r-3/2})
        \text{ uniformly in }i\in \Vert_t\setminus \Vert_1 .
        \end{array}
        \right.
    \end{equation*}
    Then define
    \begin{equation*}
        \widetilde Y := \binom{n}{r\m1}\sum_{i\in \Vert_t} \left(\big. \psi(Z_i) - \mu\right) .
    \end{equation*}
    The previous calculation and \eqref{eq: decomposition psi with fixed points} imply that $\card{\hat Y - \widetilde Y} = o(n^{r-1/2})$ a.s.~for some deterministic $o$, hence as before we just need to compute the variance of $\widetilde Y$.
    Using \Cref{lem: dependency graph} we can write:
    \begin{align*}
        \var{\binom{n}{r\m 1}^{-1}\widetilde Y}
        &= \sum_{i,j\in \Vert_t} \cov{\psi(Z_i)}{\psi(Z_j)}
        \\&= \sum_{i\in \Vert_t} \var{\psi(Z_i)}
        + \sum_{i\in \Vert_2} \cov{\psi(Z_i)}{\psi(Z_{\s(i)})}
        \\&\quad+ \sum_{i\in \Vert_3} \left( \cov{\psi(Z_i)}{\psi(Z_{\s(i)})} + \cov{\psi(Z_i)}{\psi(Z_{\s^{-1}(i)})} \right)
        \\&= n p_1\var{\psi(U,U)}
        + n(1\m p_1)\var{\psi(U,V)}
        + n p_2\cov{\psi(U,V)}{\psi(V,U)}
        \\&\quad+ 2n(1\m p_1\m p_2)\cov{\psi(U,V)}{\psi(V,W)} + o(n) 
    \end{align*}
    where $U,V,W$ are i.i.d.~$\Unif{[0,1]}$ variables.
    Thus $\var{X}=\sigma^2 n^{2r-1} + o(n^{2r-1})$ with $\sigma^2 := \Sigma_{\pi,\pi}^{p_1,p_2}$ as defined in the statement of \Cref{th: pattern normality t-cyclic with fixed points}.
    Thanks to \Cref{th: pattern Stein}, this concludes the proof.
\end{proof}

\begin{proof}[Proof of \Cref{cor: pattern normality conjugation-invariant}]
    Recall that conditionally given $t^{(n)}$, $\tau_n$ is a uniform $t^{(n)}$-cyclic permutation.
    Using \Cref{th: pattern normality t-cyclic with fixed points}, we can thus write for any $x\in\R$:
    \begin{align*}
        \prob{ X_\pi(\tau_n) - \binom{n}{r}\mu_\pi^{p_1} < n^{r-1/2}x }
        = \mean{ \prob{\left. X_\pi(\tau_n) - \binom{n}{r}\mu_\pi^{p_1} < n^{r-1/2}x \;\right| \left(t^{(n)}\right)_n }}
        \cv{n}{\infty} g_\pi^{p_1,p_2}(x)
    \end{align*}
    by dominated convergence theorem, where $g_\pi^{p_1,p_2}$ is the distribution function of $\Normal{0}{\Sigma_{\pi,\pi}^{p_1,p_2}}$.
    Joint convergence is proved by the same method and the Cram\'er--Wold theorem.
\end{proof}

\subsection{Non-degeneracy for involutive patterns}

\begin{lem}\label{lem: regularity psi with fixed points}
    For any $p_1\in[0,1]$ and $\pi\in\S_r$, the function $\psi_\pi^{p_1}$ is $(r\m1)$-Lipschitz with respect to the $L^1$-norm on $\carre$.
    Moreover, its restriction to $\{0\}\times[0,1]$ is a polynomial which is non-constant if $p_1<1$ and $r\ge 2$..
\end{lem}

\begin{proof}
    Write $\psi:=\psi_\pi^{p_1}$. Recall the variables $\hat Z_1,\dots,\hat Z_r$ used to define $\psi$.
    For any $z=(u,v),z'=(u',v')\in\carre$:
    \begin{equation*}
        \card{\psi(z)-\psi(z')}
        \le \prob{\perm{\hat Z_1,\dots,\hat Z_{r-1},z} \ne \perm{\hat Z_1,\dots,\hat Z_{r-1},z'}} .
    \end{equation*}
    Define $\mathcal{B}_{z,z'}$ as the set of points in $\carre$ having x-coordinate between $u$ and $u'$ or y-coordinate between $v$ and $v'$.
    For the permutations $\perm{\hat Z_1,\dots,\hat Z_{r-1},z}$ and $\perm{\hat Z_1,\dots,\hat Z_{r-1},z'}$ to differ, there needs to be a $\hat Z_i$ whose position with respect to $z$ (NW, NE, SW, or SE) is not the same as with respect to $z'$.
    Therefore:
    \begin{equation*}
        \card{\psi(z)-\psi(z')}
        \le \sum_{i=1}^{r-1}\prob{\hat Z_i \in \mathcal{B}_{z,z'}}
        \le (r\m1)\left(|u-u'|+|v-v'|\right)
    \end{equation*}
    and $\psi$ is indeed $(r\m1)$-Lipschitz.
    Now fix $v\in[0,1]$ and let $Q$ be the number of $\hat Z_i's$, $i\le r\m1$ on the diagonal $\Delta$.
    The random variable $Q$ follows a $\Binomial{r\m1}{p_1}$ law.    
    Now let $q$, $j_{\rm SW}$, $j_{\rm SE}$, $j_{\rm NE}$, $j_{\rm NW}$, $k_{\rm SW}$, $k_{\rm NE}$ be nonnegative integers satisfying
    \begin{equation}\label{eq: phi - decoupage points si diago}
        \left\{
        \begin{array}{lllll}
        q \in\{0,1,\dots,r-1\} ;\\
        j_{\rm SW}+j_{\rm SE}+j_{\rm NE}+j_{\rm NW}+k_{\rm SW}+k_{\rm NE} = r-1 ;\\
        k_{\rm SW}+k_{\rm NE} = q ;\\
        j_{\rm SW}+j_{\rm SE}+k_{\rm SW} = \pi(1)-1 ;\\
        j_{\rm NE}+j_{\rm NW}+k_{\rm NE} = r-\pi(1) ;
        \end{array}
        \right.
    \end{equation}
    (the last line is redundant, but we keep it for clarity).
    Under this condition we write $A_{q,\bf j,k}(v)$ for the event that $Q=q$; $j_{\rm SW},j_{\rm SE},j_{\rm NE},j_{\rm NW}$ are the numbers of non-diagonal $\hat Z_i$'s in each quadrant defined by $(v,v)$; and $k_{\rm SW}$, $k_{\rm NE}$ are the numbers of diagonal $\hat Z_i$'s lying southwest, resp.~northeast of $(v,v)$.
    If $\perm{\hat Z_1,\dots,\hat Z_{r-1},(0,v)}=\pi$ then $A_{q, \bf j,k}(v)$ is verified for some $q,\bf j,k$ satisfying \eqref{eq: phi - decoupage points si diago}.
    Moreover, conditionally on $A_{q,\bf j,k}(v)$, the probability of this event does not depend on $v$.
    Indeed one can construct a straightforward coupling between $(\hat Z_i)_{1\le i\le r-1}$ conditioned on $A_{q,\bf j,k}(v)$ and $(\hat Z_i')_{1\le i\le r-1}$ conditioned on $A_{q,\bf j,k}(v')$ such that a.s.~$\perm{\hat Z_1,\dots,\hat Z_{r-1},(0,v)}=\perm{\hat Z_1',\dots,\hat Z_{r-1}',(0,v')}$.
    Therefore:
    \begin{align*}
        \psi(0,v) 
        &= \sum_{{q,\bf j,k}\text{ s.t. }\eqref{eq: phi - decoupage points si diago}}
        \prob{\left.\perm{\hat Z_1,\dots,\hat Z_{r-1},(0,v)}=\pi \;\right| A_{q,\bf j,k}(v)}
        \prob{A_{q,\bf j,k}(v)}
        \\&= \sum_{{q,\bf j,k}\text{ s.t. }\eqref{eq: phi - decoupage points si diago}}
        c_{\pi, q,\bf j,k} \binom{r\m1}{q} p_1^q (1\mm p_1)^{r\m q\m 1} 
        \binom{r\m q\m 1}{j_{\rm SW},j_{\rm SE},j_{\rm NE},j_{\rm NW}}
        v^{k_{\rm SW}\p 2j_{\rm SW}\p j_{\rm SE}\p j_{\rm NW}} 
        (1\mm v)^{k_{\rm NE}\p 2j_{\rm NE}\p j_{\rm SE}\p j_{\rm NW}}
        \\&= v^{\pi(1)\m 1}(1\mm v)^{r\m \pi(1)} 
        \sum_{{q,\bf j,k}\text{ s.t. }\eqref{eq: phi - decoupage points si diago}}
        c_{\pi, q,\bf j,k} \binom{r\m1}{q} p_1^q (1\mm p_1)^{r\m q\m 1}
        \binom{r\m q\m 1}{j_{\rm SW},j_{\rm SE},j_{\rm NE},j_{\rm NW}}
        v^{j_{\rm SW}\p j_{\rm NW}} (1\mm v)^{j_{\rm NE}\p j_{\rm SE}}
        \\&= v^{\pi(1)-1}(1-v)^{r-\pi(1)} \sum_{q=0}^{r-1} \sum_{m=0}^{r-q-1}
        c_{\pi, q, m}'\, v^m (1-v)^{r-q-1-m}
    \end{align*}
    for some non-negative constants $c_{\pi, q,\bf j,k}$ and $c_{\pi, q, m}'$.
    Thus $\psi(0,v)$ is a polynomial in $v$.
    It remains to see that it is non-constant if $p_1<1$ and $r\ge2$.
    
    First suppose that $\pi(1)\ne1$.
    The polynomial $\psi(0,v)$ contains a monomial $v^{\pi(1)-1} \sum_{q=0}^{r-1} c_{\pi,q,0}'$ where $c_{\pi,q,0}'\ge0$, along with terms of higher degrees.
    Since $p_1<1$, the event $A_{0,\bf j,0}(v)$ has positive probability for any $(0,\bf j,0)$ satisfying \eqref{eq: phi - decoupage points si diago}, and also $c_{\pi,0,\bf j,0} >0$.
    In particular for $j_{\rm SW}+j_{\rm NW} = 0$, we deduce $c_{\pi,0,0}'>0$.
    Thus $\psi(0,v)$ is indeed non-constant. 
    
    On the other hand if $\pi(1)=1$ then $\pi(1)\ne r$ and we can do the same reasoning with the polynomial $\psi(0,1-v)$ to conclude.
\end{proof}

\begin{proof}[Proof of \Cref{prop: non degenerate involution pattern}]
    By the law of total covariance, if $U,V,W$ are i.i.d.~$\Unif{[0,1]}$ variables:
    \begin{equation*}
        \cov{\psi(U,V)}{\psi(V,W)} = 
        \mean{\cov{\psi(U,V)}{\psi(V,W)\;\big|V}}
        + \cov{\mean{\psi(U,V)\;\big| V}}{\mean{\psi(V,W)\;\big| V}} .
    \end{equation*}
    Conditionally given $V$, the variables $\psi(U,V)$ and $\psi(V,W)$ are independent, and thus the first term is null.
    Now let us study the second term.
    Since $\pi$ is an involution, for any $u_1,v_1,\dots,u_r,v_r$ we have $\perm{(u_1,v_1),\dots,(u_r,v_r)}=\pi$ if and only if $\perm{(v_1,u_1),\dots,(v_r,u_r)}=\pi$.
    Since the distribution $p_1\Leb_\Delta + (1-p_1)\Leb_{\carre}$ is symmetric with respect to $\Delta$, we deduce that $\psi(u,v)=\psi(v,u)$ for any $(u,v)\in\carre$.
    Hence:
    \begin{equation*}
        \cov{\psi(U,V)}{\psi(V,W)}
        = \cov{\mean{\psi(U,V)\;\big| V}}{\mean{\psi(W,V)\;\big| V}}
        = \var{\mean{\psi(U,V)\;\big| V}}
        \ge 0 ,
    \end{equation*}
    and in particular $(r-1)!^2\, \Sigma_{\pi,\pi}^{p_1,p_2}\ge (1\m p_1\p p_2)\var{\psi(U,V)}$.
    By \Cref{lem: regularity psi with fixed points}, $\psi$ is continuous and non-constant on $\carre$, and thus $\var{\psi(U,V)}>0$.
    As $p_1<1$, this concludes the proof.
\end{proof}

\subsection{Cycle types with few fixed points}

In this section we suppose $p_1=0$, and drop the index $p_1$.
In particular $\mu_\pi = 1/r!$ for any $\pi\in\S_r$,
\begin{equation*}
    \psi_\pi(z) = \prob{\perm{\hat Z_1,\dots,\hat Z_{r-1},z}=\pi}
\end{equation*}
where $\hat Z_1,\dots,\hat Z_{r-1}$ are i.i.d.~$\Unif{\carre}$ variables, and the matrix of \Cref{th: pattern normality t-cyclic with fixed points} rewrites as:
\begin{equation*}
    (r\m1)!^2 \Sigma_{\pi,\rho}^{p_2} = \cov{\psi_\pi(U,V)}{\psi_\rho(U,V)} + p_2\cov{\psi_\pi(U,V)}{\psi_\rho(V,U)}
    + 2(1\m p_2)\cov{\psi_\pi(U,V)}{\psi_\rho(V,W)}
\end{equation*}
where $U,V,W$ are i.i.d.~$\Unif{[0,1]}$ variables.
To establish \Cref{th: pattern normality t-cyclic w/o fixed points}, we need to study the function $\psi$ a bit more.

\begin{lem}\label{lem: pattern null cov t-cyclic}
    If $p_1=0$ then for any $\pi,\rho\in\S_r$:
    \begin{equation*}
        \cov{\psi_\pi(U,V)}{\psi_\rho(V,W)} = 0.
    \end{equation*}
\end{lem}

\begin{proof}
    For any distinct $u_1,v_1,\dots,u_r,v_r\in[0,1]$, we can write $\perm{(u_1,v_1),\dots,(u_r,v_r)}=\sigma_2^{-1}\circ\sigma_1$
    where $\sigma_1,\sigma_2\in\S_r$ are characterized by $u_{\sigma_1(1)}<\dots<u_{\sigma_1(r)}$ and $v_{\sigma_2(1)}<\dots<v_{\sigma_2(r)}$.
    If $(U_1,V_1),\dots,(U_r,V_r)$ are i.i.d.~$\Unif{\carre}$ variables then $\sigma_1,\sigma_2$ are i.i.d.~uniform permutations of size $r$.
    Conditionally given $(V_1,\dots,V_r)$, $\sigma_2$ is determined and $\sigma_1$ is still uniformly random in $\S_r$.
    Subsequently 
    \begin{equation*}
        \prob{\perm{(U_1,V_1),\dots,(U_r,V_r)}=\pi \;\big| V_1,\dots,V_r} 
        = \prob{\sigma_1 = \sigma_2\circ\pi \;\big| \sigma_2}
        = 1/r!
    \end{equation*}
    almost surely, and similarly
    \begin{equation*}
        \prob{\perm{(U_1,V_1),\dots,(U_r,V_r)}=\rho \;\big| U_1,\dots,U_r} 
        = \prob{\sigma_2 = \sigma_1\circ\rho^{-1} \;\big| \sigma_1}
        = 1/r!
    \end{equation*}
    almost surely. 
    We can thus write $\mean{\psi_\pi(U,V)\psi_\rho(V,W) \;\big| V} = \mean{\psi_\pi(U,V)\,\big| V}\mean{\psi_\rho(V,W) \,\big| V} = (1/r!)^2$ a.s., and the lemma readily follows.
\end{proof}

This yields the announced formula for $\Sigma_{\pi,\rho}^{p_2}$.
It remains to compute the rank and show non-degeneracy.
For this we shall give an explicit polynomial formula for $\psi_\pi$; it can already be found in \cite{JNZ15} (see Equation (5) therein for its rescaled expression) but we provide a proof below for completeness.

\begin{lem}\label{lem: study varphi}
    Let $\pi\in\S_r$.
    We can write, for any $(u,v)\in\carre$:
    \begin{equation*}
        \psi_\pi(u,v) = \frac{1}{(r-1)!}\sum_{i=1}^r g_i(u) g_{\pi(i)}(v)
    \end{equation*}
    where for any $j\in[r]$ and $w\in\R$, $g_j(w) := \binom{r-1}{j-1}w^{j-1}(1-w)^{r-j}$.
    Moreover, if $r\ge2$ then the symmetrized function $(u,v)\mapsto\psi^{\rm sym}_\pi(u,v):=\psi_\pi(u,v)+\psi_\pi(v,u)$ is not constant.
\end{lem}

\begin{proof}
    We omit the index $\pi$ in this proof.
    Write $\psi = \psi_1+\dots+\psi_r$ where for any $i\in[r]$ and $(u,v)\in\carre$, if $\hat Z_1,\dots,\hat Z_{r-1}$ denote i.i.d.~$\Unif{\carre}$ random variables:
    \begin{equation*}
        \psi_i(u,v) := \P\left( \perm{\hat Z_1,\dots,\hat Z_{r-1}, (u,v)}=\pi \;,\; (u,v) \text{ is $i$-th from the left} \right).
    \end{equation*}
    For this event, say $A_i(u,v)$, to be realized, there needs to be exactly $i\m1$ points left of $(u,v)$ and $\pi(i)\m1$ below $(u,v)$.
    More precisely there needs to be exactly $j_\mathrm{SW},j_\mathrm{SE},j_\mathrm{NE},j_\mathrm{NW}$ points in each quadrant defined by $(u,v)$, where these quantities depend only on $\pi$ and $i$.
    Let $B_i(u,v)$ be this last event; its probability can be written as $\binom{r-1}{j_\mathrm{SW} ,\, j_\mathrm{SE} ,\, j_\mathrm{NE} ,\, j_\mathrm{NW}} u^{i-1} (1-u)^{r-i} v^{\pi(i)-1} (1-v)^{r-\pi(i)}$.
    Then conditionally on $B_i(u,v)$, the probability of $A_i(u,v)$ does not depend on $(u,v)$ anymore.
    Indeed by a coupling argument, one can e.g.~see that $\prob{A_i(u,v)\big|B_i(u,v)}$ is the probability that $j_\mathrm{SW}$ uniform points in $[0,1/2]^2$, $j_\mathrm{SE}$ uniform points in $[1/2,1]\!\times\![0,1/2]$, $j_\mathrm{NE}$ uniform points in $[1/2,1]^2$, and $j_\mathrm{NW}$ uniform points in $[0,1/2]\!\times\![1/2,1]$, all independent, form the permutation induced by $\pi$ on $[r]\setminus\{i\}$.
    Hence $\prob{A_i(u,v)\big|B_i(u,v)}$ equals a positive constant, independent of $(u,v)$, and we deduce that
    \begin{equation*}
        \psi_i(u,v)
        = \prob{A_i(u,v)\big|B_i(u,v)} \prob{B_i(u,v)}
        = c_i u^{i-1} (1-u)^{r-i} v^{\pi(i)-1} (1-v)^{r-\pi(i)} .
    \end{equation*}
    To find the value of $c_i$, first notice that $\mean{\psi_i(U,V)} = \frac{1}{r}\frac{1}{r!}$.
    Then in the previous formula:
    \begin{equation*}
        \frac{1}{r}\frac{1}{r!}
        = c_{i} B\big(i , r\m i\p 1\big) B\big(\pi(i) , r\m \pi(i)\p 1\big)
    \end{equation*} 
    where $B$ denotes the Beta function, and thus
    \begin{equation*}
        c_{i} = \frac{(r-1)!}{(i\m 1)!(r\m i)!(\pi(i)\m 1)!(r\m \pi(i))!} 
        = \frac{1}{(r-1)!}\binom{r\m1}{i\m1}\binom{r\m1}{\pi(i)\m1} .
    \end{equation*}
    This yields the announced formula for $\psi = \sum_i \psi_i$.
    Now let us turn to the last claim of the lemma.
    By reindexing we can write $\psi^{\rm sym}(u,v) = \sum_{i=1}^r u^{i-1} (1-u)^{r-i} Q_i(v)$ where 
    $$ Q_i(v) := c_i v^{\pi(i)-1} (1-v)^{r-\pi(i)} + c_{\pi^{-1}(i)} v^{\pi^{-1}(i)-1} (1-v)^{r-\pi^{-1}(i)} .$$
    Suppose that $\psi^{\rm sym}$ is constant on $\carre$.
    Then by setting $u=0$ we see that the polynomial $Q_1(v)$ is constant in $v$.
    If $\pi(1)>1$ then $\pi^{-1}(1)>1$ and the lowest degree term of $Q_1$ has positive degree and coefficient, which contradicts the fact that $Q_1$ is constant.
    Hence $\pi(1)=1$.
    Likewise by setting $u=1$ we see that $Q_{r}$ is constant, which implies $\pi(r)=1$ or $\pi^{-1}(r)=1$.
    Hence $r=1$ and this concludes the proof.
\end{proof}

\begin{proof}[Proof of \Cref{th: pattern normality t-cyclic w/o fixed points}]
    First, let us show non-degeneracy.
    Fix $p_2\in[0,1]$, $r\ge2$, $\pi\in\S_r$.
    Since $\psi_\pi$ is continuous and not constant on $\carre$ according to \Cref{lem: study varphi}, it holds that $\var{\psi_\pi(U,V)}>0$.
    
    By Cauchy--Schwarz inequality, $\card{\cov{\psi_\pi(U,V)}{\psi_\pi(V,U)}} \le \var{\psi_\pi(U,V)}$.
    Suppose by contradiction that $\Sigma_{\pi,\pi}^{p_2}=0$. 
    Then necessarily $p_2=1$ and $\cov{\psi_\pi(U,V)}{\psi_\pi(V,U)} = -\var{\psi_\pi(U,V)}$.
    This equality case can only happen if $\psi_\pi(U,V)$ and $\psi_\pi(V,U)$ are linearly dependent, and here we would have $\psi_\pi(U,V)=-\psi_\pi(V,U)$ almost surely.
    However it follows from \Cref{lem: study varphi} that $\psi_\pi(U,V)+\psi_\pi(V,U)$ is not a.s.~constant, and this proves by contradiction that $\Sigma_{\pi,\pi}^{p_2}>0$.
    
    Now, Let us study the rank of $\Sigma^{p_2}$.
    Fix $\alpha,\beta \ge 0$ such that $\alpha^2+\beta^2=1$ and $2\alpha\beta=p_2$ (this is possible since $p_2\le 1$), and define for any $\pi\in\S_r$ and $(u,v)\in\carre$: 
    \begin{equation*}
        \chi_\pi(u,v) := (r-1)!\left(\alpha\bar\psi_\pi(u,v) + \beta\bar\psi_\pi(v,u)\right)
    \end{equation*}
    where $\bar\psi_\pi := \psi_\pi - \frac{1}{r!}$.
    Then:
    \begin{align*}
        \frac{1}{(r-1)!^2}\cov{\chi_\pi(U,V)}{\chi_\rho(U,V)}
        &= \cov{\alpha\psi_\pi(U,V) + \beta\psi_\pi(V,U)}{\alpha\psi_\rho(U,V) + \beta\psi_\rho(V,U)}
        \\&= (\alpha^2+\beta^2)\cov{\psi_\pi(U,V)}{\psi_\rho(U,V)}
        + 2\alpha\beta\cov{\psi_\pi(U,V)}{\psi_\rho(V,U)}
        \\&= (r-1)!^2\Sigma^{p_2}_{\pi,\rho} .
    \end{align*}
    Hence, up to a multiplicative constant, $\Sigma^{p_2}$ is the Gram matrix associated to the family $(\chi_\pi)_{\pi\in\S_r}$ in $L^2(\carre)$.
    Therefore, the rank of $\Sigma^{p_2}$ equals the rank of this family.
    Recall the notation of \Cref{lem: study varphi}.
    For any $(u,v)$ we have:
    \begin{equation*}
        1 = \sum_{\pi\in\S_r} \psi_\pi(u,v)
        = \sum_{\pi\in\S_r}\frac{1}{(r-1)!}\sum_{i=1}^r g_i(u) g_{\pi(i)}(v)
        = \sum_{1\le i,j\le r} g_i(u) g_j(v) .
    \end{equation*}
    Thus as in \cite{JNZ15} we can write:
    \begin{equation*}
        \bar\psi_\pi(u,v)
        = \frac{-1}{r!} + \frac{1}{(r-1)!}\sum_{i=1}^r g_i(u) g_{\pi(i)}(v)
        = \frac{1}{(r-1)!}\sum_{1\le i,j\le r} \left(\delta_{j,\pi(i)}-\frac1r\right)g_i(u)g_j(v)
    \end{equation*}
    where $\delta$ is the Kronecker symbol.
    Define $A^\pi$ as the matrix $\left(\delta_{j,\pi(i)}-\frac1r\right)_{1\le i,j\le r}$.
    Then 
    \begin{equation*}
        \chi_\pi(u,v)
        = \sum_{1\le i,j\le r}\left( \alpha A^\pi_{i,j} + \beta A^{(\pi^{-1})}_{i,j} \right)g_i(u)g_j(v) .
    \end{equation*}
    Since the family $\left(\big.(u,v)\mapsto g_i(u)g_j(v)\right)_{1\le i,j\le r}$ is linearly independent in $L^2(\carre)$, the rank of $(\chi_\pi)_{\pi\in\S_r}$ is equal to the rank of the family $\left(\alpha A^\pi + \beta A^{(\pi^{-1})}\right)_{\pi\in\S_r}$ in $\mathcal{M}_r(\R)$.
    As explained in \cite{JNZ15} (see the proof of Theorem~2.5 therein), the family $(A^\pi)_{\pi\in\S_r}$ spans the space of matrices with rows summing to 0 and columns summing to 0, and thus has rank $(r-1)^2$.

    First suppose $p_2<1$, which corresponds to $\alpha\ne\beta$.
    This implies that each matrix in $(A^\pi)_{\pi\in\S_r}$ is a linear combination of matrices in $\left(\alpha A^\pi + \beta A^{(\pi^{-1})}\right)_{\pi\in\S_r}$ and vice versa, hence the latter family also has rank $(r-1)^2$.

    Now suppose $p_2=1$ i.e.~$\alpha=\beta$.
    It is easily checked that the family $\left(A^\pi + A^{(\pi^{-1})}\right)_{\pi\in\S_r}$ spans the space of {\em symmetric} matrices with rows summing to 0 and columns summing to 0.
    Such a matrix $A$ is exactly characterized by its coefficients $A_{i,j}$ with $i\le j<r$, hence this space has dimension $r(r-1)/2$.
    This concludes the proof.
\end{proof}

\bibliographystyle{plainurl}
\bibliography{bibli}

\end{document}